\documentclass[12pt, a4paper]{article}
\usepackage{amsfonts}
\usepackage{amsmath}
\usepackage[latin1]{inputenc}
\usepackage{amsthm}
\usepackage{amssymb}
\usepackage[francais,english]{babel}
\usepackage{graphics,amssymb,epsfig,stmaryrd,color}

\numberwithin{equation}{section}
\theoremstyle{plain}
\newtheorem{thm}{Theorem}[section]

\newtheorem{lem}[thm]{Lemma}

\theoremstyle{definition}

\theoremstyle{remark}

\newcommand{\ve}{\varepsilon}

\newcommand{\deb}{\rightharpoonup}

\newcommand{\pical}{\mathcal{P}}

\newcommand{\R}{\mathbb{R}}

\let \eps\varepsilon

\def\<#1,#2>{\left<#1,#2\right>}
\let\bar\overline

\begin{document}

\title{From Knothe's transport to Brenier's map and a continuation method
for optimal transport}
\author{G. Carlier\thanks{{\scriptsize CEREMADE, Universit\'e Paris IX
Dauphine \texttt{carlier@ceremade.dauphine.fr}}}, A. Galichon\thanks{%
{\scriptsize D\'epartement d'Economie, Ecole Polytechnique, \texttt{%
alfred.galichon@polytechnique.edu}}}, F. Santambrogio\thanks{{\scriptsize %
CEREMADE, Universit\'e Paris IX Dauphine, \texttt{%
filippo@ceremade.dauphine.fr}}}}
\maketitle

\begin{abstract}
A simple procedure to map two probability measures in $\mathbb{R}^d$ is the
so-called \emph{Knothe-Rosenblatt rearrangement}, which consists in
rearranging monotonically the marginal distributions of the last
coordinate, and then the conditional distributions, iteratively. We show
that this mapping is the limit of solutions to a class of Monge-Kantorovich
mass transportation problems with quadratic costs, with the weights of the
coordinates asymptotically dominating one another. This enables us to design
a continuation method for numerically solving the optimal transport problem.
\end{abstract}



\section{Introduction}

\label{intro}

Given two Borel probability measures $\mu$ and $\nu$ on $\mathbb{R}^d$, a
Borel map $S$ : $\mathbb{R}^d\to \mathbb{R}^d$ is said to be a transport map
between $\mu$ and $\nu$ if $S\sharp \mu=\nu$ where $S\sharp \mu$ denotes the
push-forward (or image measure) of $\mu$ through $\nu$ (i.e. $S\sharp \mu
(B)=\mu(S^{-1}(B))$ for every Borel $B$). In the present article, we will
focus on two particular transport maps: the Knothe-Rosenblatt transport and
the Brenier's map.

\smallskip

\textbf{The Knothe-Rosenblatt transport.} The Knothe-Rosenblatt
rearrangement was independently proposed by Rosenblatt \cite{Rosenblatt52}
for statistical purposes and by Knothe \cite{Knothe57} in order to extend
the Brunn-Minkowski inequalities. The principle is the following, as
explained in Villani \cite{Villani06}. Let $\mu $ and $\nu $ be two Borel
probability measures on ${\mathbb{R}}^{d}$ and assume for simplicity for the moment 
that $\mu$ is absolutely continuous with respect to the Lebesgue measure. Let us denote by
$\mu^d$ (respectively $\nu^d$) the $d$-th marginal of $\mu$ (respectively $%
\nu$) and $\mu^{d-1}_{x_d}$, $\mu^{d-2}_{(x_d, x_{d-1})}$,\ldots, $%
\mu^1_{(x_d,...,x_2)}$ (respectively $\nu^{d-1}_{y_d}$, $\nu^{d-2}_{(y_d,
y_{d-1})}$,\ldots, $\nu^1_{(y_d,...,y_2)}$) the successive disintegrations
(or conditional measures) of $\mu$ (respectively $\nu$) given $x_d$, $(x_d,
x_{d-1})$,..., $(x_d,..., x_2)$ (respectively given $y_d$, $(y_d, y_{d-1})$,
..., $(y_d,..., y_2)$). Now let $T_d=T_d(x_d)$ be the monotone nondecreasing
map transporting $\mu^d$ to $\nu^d$, such a map is well-defined and unique
as soon as $\mu_d$ has no atoms and in this case, it is explicitly given by $%
T_{d}=G_{d}^{-1}\circ F_{d}$ (with $F_d(\alpha):=\mu((-\infty, \alpha])$ and
$G_d(\alpha):=\nu((-\infty, \alpha])$). Then let $T_{d-1}=T_{d-1}(x_{d-1},
x_d)$ be such that $T_d(., x_{d})$ is monotone and maps $\mu^{d-1}_{x_d}$ to
$\nu^{d-1}_{T_d(x_d)}$. One repeats the construction (well-known by
statisticians under the name of conditional quantile transforms) iteratively
until we define $T_{1}(x_{1},x_{2},...,x_{d})$, which is monotone in $x_1$
and transports $\mu ^{1}_{(x_{d},..., x_2)}$ onto $\nu^1_{T_2(x_2,...,x_d)}$%
. Finally, the \emph{Knothe-Rosenblatt rearrangement} $T$ is defined by $%
T(x)=\left(
T_{1}(x_{1},x_{2},...,x_{d}),...,T_{d-1}(x_{d-1},x_{d}),T_{d}(x_{d})\right) $%
. Obviously, $T$ is a transport map from $\mu $ to $\nu $, i.e. $T\sharp \mu
=\nu $. By construction, the Knothe transport $T$ has a triangular Jacobian
matrix with nonnegative entries on its diagonal. Note
also that the computation of the Knothe transport only involves
one-dimensional monotone rearrangements and that it is well defined as soon the measures one transports have no atoms. The precise assumption is the following.

{\bf Assumption {\it (H-source)}:} the measure $\mu^d$, as well as $\mu^d-$almost all the measures $\mu^{d-1}_{x_d}$, and the measures $\mu^{d-2}_{x_d,x_{d-1}}$ for $\mu^d-$a.e. $x_d$ and $\mu^{d-1}_{x_d}-$a.e. $x_{d-1}$\dots up to almost all the measures $\mu^1_{x_d,\dots,x_2}$, which are all measures on the real line, must have no atoms. 

Notice that {\it (H-source)} is satisfied as soon as $\mu$ is  absolutely continuous with respect to the Lebesgue measure. 

\textbf{The Monge-Kantorovich problem and the Brenier's map.} Optimal
transportation theory provides an alternative way to transport $\mu$ to $\nu$%
. We recall that in the case of the quadratic cost, the Monge-Kantorovich
problem reads as
\begin{equation}  \label{mk}
\inf_{\pi \in {\Gamma}(\mu,\nu) } \int_{\mathbb{R}^d\times \mathbb{R}^d}
 \vert x -y\vert^2 d\pi(x, y)
\end{equation}
where $\Gamma(\mu,\nu)$ denotes the set of transport plans between $\mu$ and
$\nu$ i.e. the set of probability measure on ${\mathbb{R}}^d \times {\mathbb{%
R}}^d $ with marginals $\mu$ and $\nu$. We refer to the books of Villani
\cite{Villani03}, \cite{Villani06} for a modern account of optimal
transportation theory. The linear problem (\ref{mk}) is a relaxation of the
Monge problem
\begin{equation}  \label{monge}
\inf_{S \; : \; S\sharp \mu=\nu } \int_{\mathbb{R}^d} \vert x-S(x)\vert^2
d\mu(x)
\end{equation}
When $\mu$ is absolutely continuous with respect to the Lebesgue measure,
Brenier \cite{Brenier91} proved that (\ref{monge}) has a unique solution
which is characterized by the fact that it is the gradient of some convex
function. More precisely, there exists a unique (up to constants and $\mu$%
-a.e. equivalence) convex function $V: \; {\mathbb{R}}^d \to {\mathbb{R}}$
such that $\nabla V \sharp \mu = \nu$. Also $(id\times \nabla V)\sharp \mu$
is the only solution of (\ref{mk}) and $\nabla V$ is called the Brenier's
map between $\mu$ and $\nu$.

\textbf{The Knothe-Rosenblatt as a limit of optimal transportation plans. }
Let us slightly modify the quadratic cost in (\ref{mk}) and replace it with
the weighted quadratic cost
\begin{equation*}
c_\eps(x,y):=\sum_{i=1}^d \lambda_i(\eps) (x_i - y_i)^2
\end{equation*}
where the $\lambda_i(\eps)$'s are positive scalars depending on a parameter $%
\eps>0$. If $\mu$ is absolutely continuous with respect to the Lebesgue
measure, the corresponding optimal transportation problem admits a unique
solution $T_\eps$. When in addition, for all $k \in \{1,..., d-1\}$, $%
\lambda_{k}(\eps) / \lambda_{k+1}(\eps) \to 0$ as $\eps\to 0$, it is natural
to expect the convergence of $T_\eps$ to the Knothe transport $T$. 
We will
show that this convergence holds provided $\nu$ satisfies some additional
condition, and namely

{\bf Assumption {\it (H-target)}:} the measure $\nu^d$, as well as $\nu^d-$almost all the measures $\nu^{d-1}_{y_d}$, and the measures $\nu^{d-2}_{y_d,y_{d-1}}$ for $\nu^d-$a.e. $y_d$ and $\nu^{d-1}_{y_d}-$a.e. $y_{d-1}$\dots up to almost all the measures $\nu^2_{y_d,\dots,y_3}$, which are all measures on the real line, must have no atoms. 

Notice that {\it (H-target)} is not natural as {\it (H-source)} is. Yet, we will show a counter-example to the convergence result when it is not satisfied. {\it (H-target)} as well is satisfied should $\nu$ be absolutely continuous (actually, this assumption is slightly weaker then {\it (H-source)}, since the last disintegration measures are not concerned).

This convergence result was conjectured by Y. Brenier as a very natural one, and actually its proof is not hard. Yet, it was not known before that extra assumptions on $\nu$ were needed. This makes one of the point of interest of this paper.

The other point is what we investigate later in the paper, i.e. the other direction: from Knothe to
Brenier. We will study the dependence $\eps\mapsto T_\eps$ by means of the
evolution with respect to $\eps$ of the dual variables. This will enable us,
to design a numericaly strategy to approximate all the optimal transports $%
T_\eps$ taking as initial condition the (cheap to compute) Knothe transport $%
T$.

\textbf{An example. } To illustrate the problem in a particular case where
explicit solutions are available, take $d=2$, and $\mu $ and $\nu $ two
Gaussian measures where $\mu =N\left( 0,I_{2}\right) $ and $\nu =N\left( 0,%
\begin{pmatrix}
a & b \\
b & c%
\end{pmatrix}%
\right) $. Take $\lambda_{1}\left( \varepsilon \right) =\varepsilon $ and $%
\lambda _{2}\left( \varepsilon \right) =1$. Then it can be verified that $%
T_{\varepsilon }$ is linear, and that its matrix in the canonical basis of ${%
\mathbb{R}}^{2}$ is
\begin{equation*}
T_{\varepsilon }=\frac{1}{\sqrt{a\eps ^{2}+c+2\eps \sqrt{ac-b^{2}}}}%
\begin{pmatrix}
a\eps +\sqrt{ac-b^{2}} & b \\
b\eps & c+\eps \sqrt{ac-b^{2}}%
\end{pmatrix}%
\end{equation*}%
which converges as $\varepsilon \rightarrow 0$ to $T=%
\begin{pmatrix}
\sqrt{a-b^{2}/c} & b/\sqrt{c} \\
0 & \sqrt{c}%
\end{pmatrix}%
$, which is precisely the matrix of the Knothe transport from $\mu $ to $\nu
$.

\smallskip

\textbf{Organization of the paper. } In section \ref{limit}, we show, under
suitable assumptions, that the optimal transportation maps for the cost $c_{%
\eps}$ converge to Knothe's transport map as the parameter $\eps$ goes to $0$%
, we will also emphasize that some conditions are to be imposed on $\nu $
for the convergence to hold. In section \ref{odeforp}, we show that the
evolution of the dual variables in the optimal transportation problem for
cost the $c_{\eps}$ is given by a well-posed ordinary differential equation.
Finally in section \ref{numres}, we discretize this equation and give
several numerical results.

\section{Knothe transport as a limit of quadratic optimal transports}

\label{limit}

We directly state our first result, whose proof, in the spirit of $\Gamma-$convergence developments (see \cite{introgammaconve}), will require several steps.

\begin{thm}
\label{t1}Let $\mu$ and $\nu$ be two probability
measures on ${\mathbb{R}}^d$ satisfying {\it (H-source)} and {\it (H-target)}, respectively, with finite second moments, and $\gamma_{\eps}$ be an
optimal transport plan for the costs $%
c_\eps(x,y)=\sum_{i=1}^d \lambda_i(\eps) (x_i - y_i)^2 $, for some weights $%
\lambda_{k}(\eps)>0$. Suppose that for all $k \in \{1,..., d-1\}$, $%
\lambda_{k}(\eps) / \lambda_{k+1}(\eps) \to 0$ as $\eps\to 0$. Let $T$ be the
Knothe-Rosenblatt map between $\mu$ and $\nu$ and $\gamma_K\in\pical(\R^d\times\R^d)$ the associated transport plan (i.e. $\gamma_K:=(id \times T)\sharp \mu$). Then $\gamma_\eps\deb\gamma_K$ as $\ve\to 0$. 

Moreover, should the plans $\gamma_\ve$ be induced by transport maps $T_\ve$, then these maps would
converge to $T$ in $L^2(\mu)$ as $\eps \to 0$.
\end{thm}

\begin{proof}

Take the plans $\gamma_\ve$ that are optimal for the Brenier-like cost $c_\ve$ given by
$$c_\ve(x,y)= \sum_{i=1}^d \lambda_i(\ve) (x_i - y_i)^2$$
(we suppose for simplicity $\lambda_d(\ve)=1$ and $\lambda_i(\ve)/\lambda_{i+1}(\ve)\to 0$).
Suppose (which is possible, up to subsequences) $\gamma_\ve\deb\gamma$. We want to prove $\gamma=\gamma_K$.

\smallskip

By comparing $\gamma_\ve$ to $\gamma_K$ and using optimality we first get
\begin{equation}\label{optimality}
\int c_\ve\,d\gamma_\ve\leq \int c_\ve\,d\gamma_K
\end{equation}
and, passing to the limit as $\ve\to 0$, since $c_\ve$ converges locally uniformly to $c^{(d)}(x,y)=(x_d-y_d)^2$, we get
$$\int c^{(d)} d\gamma\leq \int c^{(d)}d\gamma_K.$$
Yet, the function $c^{(d)}$ only depends on the variables $x_d$ and $y_d$ and this shows that the measure $(\pi_d)_\sharp\gamma$ gets a better result than $(\pi_d)_\sharp\gamma_K$ with respect to the quadratic cost ($\pi_d$ being the projection onto the last coordinates, i.e. $\pi_d(x,y)=(x_d,y_d)$). Yet, the measure $\gamma_K$ has been chosen on purpose to get optimality from $\mu_d$ to $\nu_d$ with respect to this cost, and the two measures share the same marginals. Moreover, thanks to the assumptions on $\mu_d$, this optimal transport plan (which is actually induced by a transport map) is unique. This implies $(\pi_d)_\sharp\gamma=(\pi_d)_\sharp\gamma_K$. Let us call $\gamma^d$ this common measure.

\smallskip

We go back to \eqref{optimality} and go on by noticing that all the measures $\gamma_\ve$ have the same marginals as $\gamma_K$ and hence their (separate) projection onto $x_d$ and $y_d$ are $\mu_d$ and $\nu_d$, respectively. This implies that $(\pi_d)_\sharp\gamma_\ve$ must realize a result which is worse than $(\pi_d)_\sharp\gamma_K$ as far as the quadratic cost is concerned and consequently we have
\[\begin{split}
&\int |x_d-y_d|^2d(\pi_d)_\sharp\gamma_K(x_d,y_d)+ \sum_{i=1}^{d-1} \lambda_i(\ve) \int (x_i - y_i)^2d\gamma_\ve\\
& \leq \int c_\ve\,d\gamma_\ve \leq \int c_\ve\,d\gamma_K\\
&=\int |x_d-y_d|^2d(\pi_d)_\sharp\gamma_K(x_d,y_d)+ \sum_{i=1}^{d-1} \lambda_i(\ve) \int (x_i - y_i)^2d\gamma_K,
\end{split}\]
which implies, by simplifying the common term in $d(\pi_d)_\sharp\gamma_K$, dividing by $\lambda_{d-1}(\ve)$ and passing to the limit,
$$\int c^{(d-1)}d\gamma\leq \int c^{(d-1)}d\gamma_K$$
(we use the general notation $c^k(x,y)=|x_k-y_k|^2$).
We can notice that both integrals depend on the variables $x_{d-1}$ and $y_{d-1}$ only. Anyway, we can project onto the variables $(x_{d-1},x_d)$ and $(y_{d-1},y_d)$ (obtaining measures $(\pi_{d-1})_\sharp\gamma$ and $(\pi_{d-1})_\sharp\gamma_K$) so that we disintegrate with respect to
the measure $\gamma^d$. We have
\begin{eqnarray}
\int d\gamma^d(x_d,y_d) \int |x_{d-1}-y_{d-1}|^2 d\gamma^{d-1}_{(x_d,y_d)}(x_{d-1},y_{d-1})\nonumber\\
\leq
\int d\gamma^d(x_d,y_d) \int |x_{d-1}-y_{d-1}|^2 d\gamma^{d-1}_{(x_d,y_d),K}(x_{d-1},y_{d-1}). \label{disintegree}
\end{eqnarray}
It is is sufficient to prove that the measures $\gamma^{d-1}_{(x_d,y_d)}$ share the same marginals on $x_{d-1}$ and $y_{d-1}$ as the corresponding $\gamma^{d-1}_{(x_d,y_d),K}$ to get that their quadratic performance should be worse than the corresponding performance of $\gamma^{d-1}_{(x_d,y_d),K}$ (this is because the Knothe measure has been chosen exactly with the intention of being quadratically optimal on $(x_{d-1},y_{d-1})$ once $x_d$ and $y_d$  are fixed). Yet, \eqref{disintegree} shows that, on average, the result given by the those measures is not worse than the results of the optimal ones. Thus, the two results coincide for almost any pair $(x_d,y_d)$ and, by uniqueness of the optimal transports (this relies on the assumptions on the measures $\mu^{d-1}_{x_d}$), we get  $\gamma^{d-1}_{(x_d,y_d)}= \gamma^{d-1}_{(x_d,y_d),K}$. To let this proof work it is sufficient to prove that the projections of the two measures coincide for $\gamma^d-$a.e. pair $(x_d,y_d)$. For fixed $(x_d,y_d)$ we would like to prove, for any $\phi$
$$\int \phi(x_{d-1})d\gamma^{d-1}_{(x_d,y_d)}=\int \phi(x_{d-1})d\gamma^{d-1}_{(x_d,y_d),K}$$
(and to prove an analogous equality for functions of $y_{d-1}$). Since we accept to prove it for a.e. pair $(x_d,y_d)$, it is sufficient to prove this equality:
\[\begin{split}
&\int d\gamma^d(x_d,y_d)\psi(x_d,y_d)\int \phi(x_{d-1})d\gamma^{d-1}_{(x_d,y_d)}\\
&=\int d\gamma^d(x_d,y_d)\psi(x_d,y_d)\int \phi(x_{d-1})d\gamma^{d-1}_{(x_d,y_d),K}
\end{split}\]
for any $\phi$ and any $\psi$. This means proving
$$\int \psi(x_d,y_d)\phi(x_{d-1})d\gamma^{d-1}=\int \psi(x_d,y_d)\phi(x_{d-1})d\gamma^{d-1}_K,$$
which is not trivial since we only know that the two measures $\gamma^{d-1}$ and $\gamma^{d-1}_K$ have the same marginals with respect to the pairs $(x_{d-1},x_d)$, $(y_{d-1},y_d)$ (since they have the same projections onto $x$ and onto $y$) and $(x_d,y_d)$ (since we just proved it). But here there is a function of the three variables $(x_{d-1},x_d, y_d)$. Yet, we know that the measure $\gamma^d$ is concentrated on the set $y_d=T_d(x_d)$ for a certain map $T_d$, and this allows to replace the expression of $y_d$, thus getting rid of one variable. This proves that the function $\psi(x_d,y_d)\phi(x_{d-1})$ is actually a function of $(x_{d-1},x_d)$ only, and that equality holds when passing from $\gamma$ to $\gamma_K$.The same can be performed on functions $\psi(x_d,y_d)\phi(y_{d-1})$ but we have in this case to ensure that we can replace $x_d$ with a function of $y_d$, i.e. that we can invert $T_d$. This is possible thanks to the assumption on $\nu_d$, since $T_d$ is the optimal transport from $\mu_d$ to $\nu_d$, but an optimal transport exists in the other direction as well and it gives the same optimal plan (thanks to uniqueness). These facts prove that the measures $\gamma^{d-1}_{(x_d,y_d)}$ and $\gamma^{d-1}_{(x_d,y_d),K}$ have the same marginals and hence, since they are both optimal, they coincide for a.e. pair $(x_d,y_d)$. This implies $\gamma^{d-1}=\gamma^{d-1}_K$.

\smallskip

Now, it is possible to go on by steps: once we have proven that $\gamma^{h}=\gamma^{h}_K$, let us take  \eqref{optimality} and estimate all the terms with $|x_i-y_i|^2$ and $i\geq h$ thanks to the optimality of $\gamma_K$, thus getting
\[\begin{split}
&\sum_{i\geq h}\lambda_i(\ve)\int |x_i-y_i|^2d\gamma_K+ \sum_{i=1}^{h-1} \lambda_i(\ve) \int (x_i - y_i)^2d\gamma_\ve\\
&\leq \int c_\ve\,d\gamma_\ve\leq \int c_\ve\,d\gamma_K\\
&=\sum_{i\geq h}\lambda_i(\ve)\int |x_i-y_i|^2d\gamma_K+ \sum_{i=1}^{h-1}\lambda_i(\ve) \int (x_i - y_i)^2d\gamma_K,
\end{split}\]
and consequently, by dividing by $\lambda_{h-1}(\ve)$ and passing to the limit,
$$\int c^{(h-1)}d\gamma\leq \int c^{(h-1)}d\gamma_K.$$
We disintegrate with respect to $\gamma^h$ and we act exacly as before: proving that the marginals of the disintegrations coincide is sufficient to prove equality of the measures. Here we will use test-functions fo the form
\[\psi(x_h,x_{h+1},\dots,x_d,y_h,y_{h+1},\dots,y_d)\phi(x_{h-1})\]
and
\[\psi(x_h,x_{h+1},\dots,x_d,y_h,y_{h+1},\dots,y_d)\phi(y_{h-1}).\]
 The same trick as before, i.e. replacing the variables $y$ with functions of the variables $x$ is again possible. To invert the trick and replace $x$ with $y$ one needs to invert part of Knothe's transport. This is possible since our assumptions imply that all the monotone transports we get are invertible. In the end we get, as before, $\gamma^{h-1}=\gamma^{h-1}_K$. This procedure may go on up to $h=2$, thus arriving at $\gamma=\gamma_K$.

\smallskip

We have now proven $\gamma_\ve\deb\gamma_K$. Yet, if all these transport plans come from transport maps,  it is well known that $(T_\ve\times id)_\sharp\mu\deb(T\times id)_\sharp\mu$ implies $T_\ve\to T$ in $L^p(\mu)$, for any $p>1$, as far as $T_\ve$ is bounded in $L^p(\mu)$. Actually, weak convergence is a simple consequence of boundedness: to go on, we can look at Young's measures. The assumption (the limit is a transport map as well) exactly means that all the Young measures are dirac masses, which implies strong convergence. In particular we get $L^2(\mu)$ convergence and $\mu$-a.e. convergence on a subsequence.
\end{proof}

Let us remark here that if instead of considering the quadratic cost $c_\eps
$, one considers the more general separable cost
\begin{equation*}
c_\eps(x,y):=\sum_{i=1}^d \lambda_i(\eps) c_i(x_i-y_i)
\end{equation*}
where each $c_i$ is a smooth strictly convex function (with suitable
growth), then the previous convergence proof carries over.

{\textbf{A counterexample when the measures have atoms}} We now show that
interestingly, and perhaps counterintuitively, the hypothesis of absence of atoms
 in theorem \ref{t1} is necessary not only for $\mu$, but also for
$\nu$. We propose a very simple example in ${\mathbb{R}}^2$ where $\mu$ is
absolutely continuous with respect to the Lebesgue measure but $\nu$ does not satisfy {\it (H-target)}, and we show that the conclusion of theorem \ref{t1} fails to hold. On
the square $\Omega := [-1,1 ] \times [-1,1 ] $, define $\mu$ such that $%
\mu(dx) = 1_{ \{ x_1x_2 < 0 \} } dx / 2$ so that the measure $\mu$ is
uniformly spread on the upper left and the lower right quadrants, and $\nu
=\mathcal{H}^1_{|S}/2$, being $S$ the segment $[-1,1]\times\{0\}$. 

The
Knothe-Rosenblatt map is easily computed as $(y_1,y_2)=T(x):=(2(x_1 +
sgn(x_2)),0)$. The solution of any symmetric transportation problem with $%
\lambda^\ve=(\ve,1)$ is $(y_1,y_2)=T^{0}(x):=(x_1,0)$ (no transport may do better than this one, which projects on the support of $\nu$). Therefore, in this example the optimal
transportation maps fail to tend to the Knothe-Rosenblatt map. The reason is the atom in the measure $\nu^2=\delta_{0}$.

{\textbf{Convergence even with atoms}} The convergence result of theorem \ref%
{t1} requires the absence of atoms in the projections of $\nu$. This is
obviously not the case if $\nu$ itself is purely atomic! Yet, this will
precisely be the case we will consider in the algorithm we propose in the
sequel. The same proof may be extended to this case under the following
assumption. Keep the same assumptions on $\mu$ but suppose that $\nu$ is
concentrated on a set $S$ with the property
\begin{equation*}
y, z\in S,\quad y\neq z \Rightarrow y_d\neq z_d.
\end{equation*}
This means that, if we restrict ourselves to $S$, then all the variables $%
y_i $ for $i<d$ are actually a function of the last variable $y_d$. This is
particularly useful when $\nu$ is purely atomic, concentrated on a finite
(or countable) set of points with different $y_d$ components.

Just come back to the proof. The equality
\begin{equation*}
\begin{split}
& \int d\gamma ^{d}(x_{d},y_{d})\psi (x_{d},y_{d})\int \phi (x_{d-1})d\gamma
_{(x_{d},y_{d})}^{d-1} \\
& =\int d\gamma ^{d}(x_{d},y_{d})\psi (x_{d},y_{d})\int \phi
(x_{d-1})d\gamma _{(x_{d},y_{d}),K}^{d-1}
\end{split}%
\end{equation*}%
only relied on $y_{d}$ being a function of $x_{d}$, which is still true. The
other equality, namely
\begin{equation*}
\begin{split}
& \int d\gamma ^{d}(x_{d},y_{d})\psi (x_{d},y_{d})\int \phi (y_{d-1})d\gamma
_{(x_{d},y_{d})}^{d-1} \\
& =\int d\gamma ^{d}(x_{d},y_{d})\psi (x_{d},y_{d})\int \phi
(y_{d-1})d\gamma _{(x_{d},y_{d}),K}^{d-1}
\end{split}%
\end{equation*}%
gives some extra troubles. It is not any more true that $x_{d}$ is a
function of $y_{d}$. Yet, it is true that $y_{d-1}$ is a function of $y_{d}$
and this allows us to reduce the expression to functions of $(x_{d},y_{d})$
only, which is sufficient to get equality. The same procedure may be
performed at subsequent steps as well.

\section{An ODE for the dual variables}

\label{odeforp}

In this section, we consider for simplicity the case $d=2$ (although our
method extends to higher dimensions), $\mu$ uniform on some convex
polyhedron $\Omega$ (for the sake of simplicity we will assume $\vert
\Omega\vert=1$) and $\nu=\frac{1}{N} \sum_{i=1}^N \delta_{y_i}$ where all
the points $y_i\in \Omega$ have a different second coordinate $y_i^{(2)}$.
For every $\eps\geq 0$, let $A_\eps$ be the diagonal $2\times 2$ matrix with
diagonal entries $(\eps, 1)$ and let $c_\eps$ be the quadratic cost defined
by $c_\eps(x,y)=A_\eps(x-y)(x-y)$. We are interested in solving the family
of optimal transportation problems
\begin{equation}  \label{mke}
\inf_{\pi \in {\Gamma}(\mu,\nu) } \int_{\mathbb{R}^d\times \mathbb{R}^d}
c_\eps(x,y) d\pi(x, y)
\end{equation}
for all values of the parameter $\eps\in[0,1]$. It is well-known, that (\ref%
{mke}) can be conveniently solved by the dual problem formulated in terms of
prices:
\begin{equation}  \label{duale}
\sup_p \Phi(p,\varepsilon):=\frac 1N \sum_{i=1}^N p_i+\int_\Omega
p_\eps^*(x)dx,
\end{equation}
where $p_\eps^*(x)=\min_i \{c_\ve(x,y_i)-p_i\}$ and we impose as a
normalization $p_1=0$. For each $\varepsilon$, there is a unique maximizer $%
p(\varepsilon)$. For each $(p,\varepsilon)$ we define $C(p,\varepsilon)_i=%
\{x\in\Omega\,:\,\inf_j c_\ve(x,y_j)-p_j=c_\ve(x,y_i)-p_i\}$. It is easy to
check that $\Phi_\ve:=\Phi(. , \eps)$ is concave differentiable and that the
gradient of $\Phi_\ve$ is given by
\begin{equation*}
\frac{\partial \Phi_\ve}{\partial p_i}(p)=\frac
1N-|C(p,\varepsilon)_i|.
\end{equation*}
By concavity of $\Phi_\eps$, the solution $p(\eps)$ of (\ref{duale}) is
characterized by the equation $\nabla\Phi_\ve(p(\eps)) = 0$. The
optimal transportation between $\mu$ and $\nu$ for the cost $c_\eps$ is then
the piecewise map taking value $y_i$ in the cell $C(p(\varepsilon),
\varepsilon))_i$. Our aim is to charcacterize the evolution of $%
p(\varepsilon)$ as $\eps$ varies. Formally, differentiating the optimality
condition  $\nabla\Phi(p(\eps), \eps) = 0$, we obtain a differential
equation for the evolution of $p(\varepsilon)$:
\begin{equation}  \label{odep}
\frac{\partial}{\partial\varepsilon}\nabla_p\Phi(p(\varepsilon),%
\varepsilon)+D^2_{p,p}\Phi(p(\varepsilon),\varepsilon)\cdot \frac{dp}{%
d\varepsilon}(\varepsilon)=0.
\end{equation}
Our aim now is to show that the equation (\ref{odep}) is well-posed starting
with the initial condition $p(0)$ (corresponding to horizontal cells of area
$1/N$); this will involve computing the second derivatives of $\Phi$ in (\ref%
{odep}), proving their Lipschitz behavior as well as obtaining a negative
bound on the larger eigenvalue of the negative semidefinite matrix $%
D^2_{p,p}\Phi$ .

\smallskip

The price vector $p(\varepsilon)$, along the evolution, will always be such
that all the areas $|C(p,\varepsilon)_i|$ are equal (and are equal to $1/N$%
). Yet, we need to prove that the differential equation is well posed and we
will set it in an open set,
\begin{equation}  \label{defdeO}
{\mathcal{O}}=\left\{(p,\varepsilon)\,:\,\frac{1}{2N}<|C(p,\varepsilon)_i|<%
\frac{2}{N}\;\mbox{for every }i\right\}.
\end{equation}
The initial datum of the equation will be such that $|C(p(0),0)_i|=1/N$ and
we will look at the solution only up to the first moment where it exits ${%
\mathcal{O}}$. Yet inside the set it will be well-posed and it will imply
conservation of the areas. Hence we will never exit ${\mathcal{O}}$.

\smallskip

All the quantities we are interested in depend on the position of the
vertices of the cells $C(p,\varepsilon)_i$, which are all polygons. Let us
call $x(p,\varepsilon)_{i,j}^\pm$ the two extremal points of the common
boundary between $C(p,\varepsilon)_i$ and $C(p,\varepsilon)_j$, that we call
$D(\varepsilon,p)_{i,j}$ (if such a common boundary exists; should it be a
single point we consider the two points as coinciding). Each one of this
points is obtained as the intersection of at least this common boundary with
another one, or with the boundary of $\Omega$ (which is supposed to be a
polygon as well, so that in the neighbourhood of almost any point locally
the boundary is a line). We want to investigate the dependence of these
points with respect to $(p,\varepsilon)$ and prove that this dependence is
Lipschitz. Notice that each point $x(p,\varepsilon)_{i,j}^\pm$ is not
defined for any value of $(p,\varepsilon)$ but only on a certain (closed)
subset of the space $\mathbb{R}^N\times (0,1)$.

\begin{lem}
The positions of the vertices $x(p,\varepsilon)_{i,j}^\pm$ depend in a
Lipschitz way on $p$ and $\varepsilon$.
\end{lem}

\begin{proof}

Locally it is true that the same point $x(p,\varepsilon)_{i,j}^\pm$ is
defined either by a system of equations
\begin{equation}  \label{syst deux equations}
\begin{cases}
A_\ve(x-y_i)(x-y_i)-p_i=A_\ve(x-y_j)(x-y_j)-p_j, \\
A_\ve(x-y_i)(x-y_i)-p_i=A_\ve(x-y_h)(x-y_h)-p_h%
\end{cases}%
\end{equation}
in the case of a point on the intersection of two common boundaries, or by a
system
\begin{equation}  \label{syst eqn and boundary}
\begin{cases}
A_\ve(x-y_i)(x-y_i)-p_i=A_\ve(x-y_j)(x-y_j)-p_j, \\
Lx=l_0%
\end{cases}%
\end{equation}
in the case of intersection with the boundary $\partial\Omega$ (locally being given by the equation $Lx=l_0$) . The first system,
after simplifying, reads as
\begin{equation}  \label{syst deux equations simplified}
\begin{cases}
2A_\ve(y_j-y_i)(x)=A_\ve(y_j)(y_j)-A_\ve(y_i)(y_i)-p_j+p_i, \\
2A_\ve(y_h-y_i)(x)=A_\ve(y_h)(y_h)-A_\ve(y_i)(y_i)-p_h+p_i,%
\end{cases}%
\end{equation}
and the second as well may be simplified the same way. This means that they
are of the form
\begin{equation*}
M(\varepsilon)x=V(\varepsilon,p)
\end{equation*}
for a matrix $M(\varepsilon)$ which reads, in usual coordinates,
\begin{equation*}
\begin{split}
&M(\varepsilon)= \left(
\begin{array}{cc}
\eps (y_j^{(1)}-y_i^{(1)}) & y_j^{{(2)}} -y_i^{(2)} \\
\eps(y_h^{(1)}-y_i^{(1)}) & y_h^{(2)}-y_i^{(2)}%
\end{array}
\right)\mbox{ in the first case and } \\
&M(\varepsilon)= \left(
\begin{array}{cc}
\eps(y_j^{(1)}-y_i^{(1)}) & y_j^{(2)}-y_i^{(2)} \\
l_1 & l_2%
\end{array}
\right)\mbox{ in the second.}
\end{split}%
\end{equation*}
The vector $V(\varepsilon,p)$ is obtained regarding the right hand sides of
the system, as in (\ref{syst deux equations simplified}). Both $M$ and $V$
depend Lipschitzly on $(\varepsilon,p)$, with uniformly bounded Lipschitz
constants. Hence, to check that the dependence of $x$ on $(\varepsilon,p)$
is Lipschitz, we only need to bound (away from zero) $\det M(\varepsilon)$.
The determinant of a $2\times 2$ matrix is given by the product of the
modulus of the two vector composing its lines, times the sinus of the angle
between them. In the first case the vectors are $A_\ve(y_j-y_i)$ and $%
A_\ve(y_h-y_i)$, while in the second they are $A_\ve(y_j-y_i)$ and $%
l=(l_1,l_2)$. In both cases they are the normal vectors to the sides of the
cell we are considering. This implies, thanks to the lower bound on the
areas of the cells, that the angles between these vectors may not be too
small. Actually, since $\Omega$ is bounded and the cells are convex, the
area of each cell is smaller than $(\mathrm{diam}\Omega)^2\sin\alpha/2,$ $%
\alpha$ being any angle between two neighbour sides. This implies a lower
bound on $\sin\alpha$. The lower bound on the moduli comes from the fact
that the vectors $l$ are chosen with modulus one and the modulus of $%
A_\ve(y_j-y_i)$ is always greater than its vertical component, which is $%
y_j^{(2)}-y_i^{(2)} $, which is supposed different from zero for any pair $%
(i,j)$. Notice that in the first case the matrix has $\varepsilon$ in the
determinant, even if we proved a lower bound on such a determinant,
independent on $\varepsilon$: this agrees with the fact that actually, this
kind of crossing (between two common boundaries of two cells) will only
happen for $\varepsilon\geq\varepsilon_0$ (for $\varepsilon<\varepsilon_0$
we only have almost horizontal strips crossing non-horizontal sides of $\Omega$).\end{proof}


\begin{lem}
The function $\Phi$ admits pure second derivatives with respect to $p$ and mixed second derivatives with respect to $p$ and $\ve$, and these second derivatives are Lipschitz continuous.
\end{lem}

\begin{proof}
We have proven that the positions of the points $x(p,\varepsilon)_{i,j}^\pm$
depend Lipschitzly, with uniformly bounded Lipschitz constants, on $p$ and $%
\varepsilon$. Since the volumes of the cells $C(p,\varepsilon)_i$ are
Lipschitz functions of these points, this implies that $\nabla_p\Phi$ is $%
C^{0,1}$. Hence it admits derivatives almost everywhere and we can compute
them in the following way.

\smallskip

The derivative of a volume of a polygonal cell is given, side by side, by
the length of the side times the average of the components which are normal
to such a side of the two derivatives $\dot{x}(p,\varepsilon)_{i,j}^+$ and $%
\dot{x}(p,\varepsilon)_{i,j}^-$ (the other terms - which are mainly the terms at the corners - are of higher order). The equation of the side $%
D(\varepsilon,p)_{i,j}$ is, as we know,
\begin{equation}  \label{Di}
2A_\ve(y_j-y_i)(x)=A_\ve(y_j)(y_j)-A_\ve(y_i)(y_i)-p_j+p_i
\end{equation}
and the normal unit vector to the side is
\begin{equation*}
n=\frac{A_\ve(y_j-y_i)}{|A_\ve(y_j-y_i)|}.
\end{equation*}
Let us start from the derivatives of the cell $C(p,\varepsilon)_i$ with
respect to a variable $p_j$ with $j\neq i$. We differentiate \eqref{Di} with
respect to $p_j$ and we get
\begin{equation*}
2A_\ve(y_j-y_i)(\dot{x})=-1.
\end{equation*}
This formula only works for $x=x(p,\varepsilon)_{i,j}^+$ and $%
x=x(p,\varepsilon)_{i,j}^-$. Obviously it ony works where they are
differentiable, i.e. almost everywhere. Hence the derivative, by summing up
and rescaling the normal vector, is given by
\begin{equation}  \label{dd1}
\frac{\partial |C(p,\varepsilon)_i|}{\partial p_j}=-\frac{l_{i,j}}{2
|A_\ve(y_j-y_i)|},
\end{equation}
where $l_{i,j}$ is the length of $D(\varepsilon,p)_{i,j}$.

As far as the derivative with respect to $p_i$ is concerned, it is not
difficult to check that we have (by summing up the results on every side)
\begin{equation}  \label{dd2}
\frac{\partial |C(p,\varepsilon)_i|}{\partial p_i}=\sum_{j}\frac{l_{i,j}}{
2|A_\ve(y_j-y_i)|},
\end{equation}
where the sum is performed on all the indices $j$ such that the cell $%
C(p,\varepsilon)_j$ is in contact with the cell $C(p,\varepsilon)_i$.

Automatically, since these derivatives only depend on the values of $l_{i,j}$%
, which depend in a Lipschitz manner on the positions of $%
x=x(p,\varepsilon)_{i,j}^\pm$, they are Lipschitz as well. This proves that $%
\Phi_{\varepsilon}$ is actually $C^{2,1}$ and that these derivatives are
well defined and admit the previous expressions (\ref{dd1})-(\ref{dd2})
everywhere.

\smallskip

The computation of the derivatives with respect to $\varepsilon$ is a bit
trickier. We derive again \eqref{Di}, but with respect to $\varepsilon$.
Since $d A_\ve / d\varepsilon = B$, we get
\begin{equation*}
\begin{split}
2A_\ve(y_j-y_i)(\dot{x})&=-2B(y_j-y_i)(x)+B(y_j)(y_j)-B(y_i)(y_i) \\
&=2B(y_j-y_i)\left(\frac{y_j+y_i}{2}-x\right).
\end{split}%
\end{equation*}
Then we renormalize the normal vector, sum up the results for $%
x=x(p,\varepsilon)_{i,j}^+$ and $x=x(p,\varepsilon)_{i,j}^-$, multiply by
the lengths and sum up the results for all the sides, and get
\begin{equation}  \label{dd3}
\frac{\partial |C(p,\varepsilon)_i|}{\partial \varepsilon}=\sum_{j}l_{i,j}
\frac{B(y_j-y_i)(y_j+y_i-x(p,\varepsilon)_{i,j}^+-x(p,\varepsilon)_{i,j}^-)%
} {2 |A_\ve(y_j-y_i)|} .
\end{equation}
In this case as well the result is Lipschitz in $(p,\varepsilon)$ and hence $%
\nabla_p\Phi$ is differentiable everywhere with respect to $\varepsilon$,
with Lipschitz derivative.

\end{proof}

We can come now back to the evolution of $p=p(\varepsilon)$ and consider
again the differential equation (\ref{odep}). To solve this equation we need
to prove that the matrix $D^2_{p,p} \Phi$ is actually invertible (for
numerical purpose, we will also need  to bound its eigenvalues away from zero). It is important to recall that we
look at the evolution of the vector $p=(p_2,\dots,p_{N})$, since we may
assume $p_1(\varepsilon)=0$ for all $\varepsilon$. Hence, we will not look
at the entries $1$ in the vectors or the matrices. The matrix we consider is
$M:=-(D^2_{p,p}\Phi)_{i,j=2,\dots,N}$ has the following properties:

\begin{itemize}
\item on each line, outside the diagonal we have negative terms $M_{i,j}=-%
\frac{l_{i,j}}{2|A_\ve(y_j-y_i)|}$;

\item each element on the diagonal is the sum of minus all the others on the
same line (hence it is positive), and possibly of the term which should be
in the same line at the first column;

\item an entry $(i,j)$ of the matrix is non-zero if and only if the cells $%
C(p,\varepsilon)_i$ and $C(p,\varepsilon)_j$ share a common boundary with
positive length;

\item in particular, for any pair $(i,j)$, even if the entry at place $(i,j)$
is zero, it is possible to find a path $i=i_0,\,i_1,\,i_2,\dots,i_k=j$ so
that the matrix has non-zero values at all the positions $(i_h,i_{h+1})$;

\item the total of the entries of the first column (the one which is not
present in the matrix) is strictly positive.
\end{itemize}

The invertibility of $M$ is ensured by the following:

\begin{lem}\label{invm}
Let the matrix $M$ satisfy the following properties
\begin{equation*}
\begin{split}
&(H1) \mbox{ for all } i, \quad M_{i,i}\geq \sum_{j\neq i} |M_{i,j}|, \\
&(H2) \mbox{ there exists $i$ such that } M_{i,i}> \sum_{j\neq i} |M_{i,j}|,
\\
&(H3) \mbox{ for any pair $(i,j)$ there is a sequence }\;
i_0,\,i_1,\,i_2,\dots,i_k \\
&\mbox{with }\;i_1=i,\;i_k=j,\; \mbox{ and } M_{i_h,i_{h+1}}\neq 0.
\end{split}%
\end{equation*}
then $M$ is invertible.
\end{lem}

\begin{proof}
Let $x\in
\mathrm{Ker}(M)$ and let $\bar{i}$ be an index such that $|x_i|$ is maximal.
We may suppose for simplicity that $x_{\bar i}$ is positive. Then we have
\begin{equation*}
0= M_{\bar i,\bar i}x_{\bar i}-\sum_j M_{\bar i,j}x_{j}\geq M_{\bar i,\bar
i}x_{\bar i}-\sum_j M_{\bar i,j}x_{\bar i}=x_{\bar i}\left(M_{\bar i,\bar
i}-\sum_j M_{\bar i,j}\right)\geq 0.
\end{equation*}
This implies that all inequalities are equalities and in particular $%
x_j=x_{\bar i}$ whenever $M_{\bar i,j}\neq 0$. Hence, the entries of $x$ on
all the indices which are ``neighbours" of $\bar i$ equal $x_{\bar i}$ (and
they are maximal as well). This allows to repeat the argument replacing $%
\bar i$ with another maximizing index $j$ and so on... since any index is
connected by a chain of neighbours to $\bar i$, we get that all the entries
are equal. But this implies that the vector in the kernel we selected must
be a multiple of the vector $(1,1,\dots, 1)$. Yet, this vector is not in the
kernel since the sum of the elements on each line is not zero for all lines,
by assumption $(H2)$. This proves that $M$ is invertible. \end{proof}

Finding a lower bound for the modulus of the eigenvalues of $M$, i.e.
quantifying its inversibility is not straightforward. Indeed, the properties
$(H1)$, $(H2)$, $(H3)$ are not sufficient to get this bound, even if we fix
the norm of the remainding column as the following counter-example shows.
The determinant of the matrices
\begin{equation*}
M_\ve= \left(
\begin{array}{ccc}
1 & -\varepsilon & 0 \\
-\varepsilon & 1 & -(1-\varepsilon) \\
0 & -(1-\varepsilon) & 1%
\end{array}
\right)
\end{equation*}
is $2\varepsilon(1-\varepsilon)\to 0$, which implies that some eigenvalue as
well goes to $0$.

\smallskip

We will obtain a positive lower bound by a compactness argument, but we will
use something stronger than simply assumptions $(H1)$, $(H2)$, $(H3)$. The
idea is that assumption $(H3)$ is not closed, but it stays closed when we
replace it with the stronger condition of the matrix being associated to a
partition (as is the case for $M=-D^2_{p,p} \Phi$). In this case if one
connection degenerates (i.e. a common boundary reduces to a point), some
other connections will play the role.

\begin{lem}
There is a positive uniform lower bound on the least eigenvalue of any matrix $M$ associated to the cell partition corresponding to a pair $(p,\ve)\in\mathcal{O}$.
\end{lem}

\begin{proof}
The proof will be obtained by contradiction. To this aim, take a sequence of partitions
of $\Omega$ into sets $(\Omega_i^n)_{i=1,\dots,N}$. We assume these sets to
be convex polygons with a bounded number of sides. This is the case for the cells associated to pairs $(p,\ve)$. We also know that
their areas are always bounded between $1/2N$ and $2/N$). These partitions give rise to a certain topology of connections between
the cells. Up to subsequences, we may suppose that this topology is always
the same on all the partitions of the sequence (since the number of possible
topologies is finite). Up to subsequences, we also have convergence in the
Hausdorff distance. This means that for any $i$ we have $\Omega_i^n\to%
\Omega_i$, and this convergence, which is the same as the convergence of all
the vertices, preserves the areas, the convexity, the upper bound on the
number of sides, the fact of being a partition... The matrices associated to
these partitions depend continuously on these sets (with respect to this
convergence, since they actually depend on the positions of the vertices).
Notice that it is possible that a side reduces its length along the sequence
up to becoming a single point in the limit. Yet, two cells which share a
boundary along the sequence will do it along the whole sequence (thanks to
our choice of not changing the topology) and at the limit, either they share
a side as well, or they share a point only, but in this case the terms $%
l_{i,j}$ converged to zero. Hence we can associate to all the partitions
their matrices and this correspondence is continuous. We have a sequence of
matrices $M_n\to M$ and let us suppose that some eigenvalue $\lambda_1^{(n)}$
goes to zero. This would imply that the matrix $M$ is associated to a
partition but has a zero eigenvalue. This is not possible, thanks to the 
proof of lemma \ref{invm} above. $M$ is associated to a partition and hence satisfies
assumptions $(H1)$ and $(H3)$. To check $(H2)$ we observe that the column
that we remove, the first one, is associated to the first cell and, up to
some rescaling but bounded factor $|A_{\varepsilon}(y_1-y_j)|$, its entries
are the lengths of its sides. Yet, this cell conserves the area bounds we
had on the sequence and its entry cannot be all zero.
\end{proof}

From the previous results on the form and the regularity of the derivatives
of $\nabla_p \Phi$, we deduce from the Cauchy-Lipschitz Theorem that the ODE
(\ref{odep}) governing the evolution of the dual variables is well posed and
actually characterizes the optimal prices:

\begin{thm}\label{3.3}
\label{wp} Let $p(\eps)$ be the solution of the dual problem (\ref{duale})
(recall the normalization $p_1(\eps)=0$), then it is the only solution of
the ODE:
\begin{equation}  \label{ode}
\frac{dp}{d\eps}(\eps)=-D^2_{p,p}\Phi(p(\varepsilon), \varepsilon)^{-1}
\left( \frac{\partial}{\partial\varepsilon}\nabla_p\Phi(p(\varepsilon),
\varepsilon) \right)
\end{equation}
with initial condition $p(0)$ such that all the horizontal strips $%
C(p(0),0)_i$ have area $1/N$.
\end{thm}

\section{Numerical results}

\label{numres}

\subsection{Algorithm}

The algorithm we propose consists simply in discretizing (\ref{ode})
(together with the initial condition $p(0)$ determined as in Theorem \ref%
{3.3}) by an explicit Euler scheme. Let $n$ be some positive integer and $%
h:=n^{-1}$ be the step size. Let us set $p_0=p(0)$ and define prices
inductively as follows.

\begin{itemize}
\item While $(p_k, kh)$ belongs to the open set ${\mathcal{O}}$ defined by (%
\ref{defdeO}), compute:
\begin{equation*}
A_ k:= - D^2_{p,p}\Phi(p_k; kh), \; \delta_k:=\frac{\partial}{%
\partial\varepsilon}\nabla_p\Phi(p_k, kh).
\end{equation*}
Note that computing $A_k$ and $\delta_k$ by formulas (\ref{dd1})-(\ref{dd2})
and (\ref{dd3}), requires to construct the cells $C(kh, p_k)_i$.

\item Solve the linear system $A_k z=\delta_k$; by taking advantage of $A_k$
being positive definite, we use the conjugate gradient algorithm for
minimizing $J_k(z)=A_k (z)(z)-2 \delta_k\cdot z$ to solve this system
exactly in $N-1$ steps. We denote by $z_k$ the solution.

\item Update the prices by setting
\begin{equation*}
p_{k+1}=p_k +h z_k.
\end{equation*}
\end{itemize}

Thanks to the Lipschitz properties established in section \ref{odeforp}, it
is easy to check that for $h$ small enough, $(p_k, kh)$ always remain in ${%
\mathcal{O}}$ and then $p_k$ is well-defined for every $k$ up to $n$. For
such an $h$ and since (\ref{ode}) is generated by a Lipschitz function on ${%
\mathcal{O}}$, it is well-known that the convergence of the Euler scheme is
linear (see for instance \cite{cm}). Denoting by $p^h$ the piecewise
constant function having values $p_k$ on intervals $[kh, (k+1)h)$, we thus
get the following convergence:

\begin{thm}
\label{conv} For $h$ small enough, the algorithm above is well-defined and
the uniform error between $p^h$ and the optimal price $p$ is $O(h)$.
\end{thm}

\subsection{Numerical experiments}

The construction of the cells at each step is achieved efficiently by an
implementation in Matlab. In the setting described above where $d=2$, $%
\Omega =\left[ 0,1\right] ^{2}$, $\mu $ is the uniform distribution on $%
\Omega $, $\nu =\frac{1}{N}\sum_{k=1}^{N}\delta _{y_{k}}$, our algorithm
computes the cells $\Omega _{k}^{\varepsilon }=\left\{ x\in \Omega
:T_{\varepsilon }\left( x\right) =y_{k}\right\} $ as well as the prices $%
p_{k}^{\varepsilon }$ of the cell $y_{k}$ for $\varepsilon =0$ to $%
\varepsilon =1$. For $\varepsilon =0$ which is the case where the
transportation plan is the Knothe one, the computational task amounts to
sorting the second component of the $y_{k}$'s. Appropriate discretization
steps are then chosen for the transition $\varepsilon =0$ to $\varepsilon =1$%
. At each step, the tesselation of $\Omega $ into the polyhedral cells $%
\Omega _{k}^{\varepsilon }$ is computed based on the prices $%
p_{k}^{\varepsilon }$. Adjacency information on these cells is computed, as
well as the length of the facet between two cells and the coordinates of its
extreme points. This information allows one to formulate a discretized
version of ODE (\ref{ode}) using an Euler discretization scheme.

For geometric computations we use the Multi-Parametric Toolbox library,
available online at http://control.ee.ethz.ch/ mpt. In particular, \texttt{%
polytope} computes the tesselation of $\Omega $ into the polyhedral cells $%
\Omega _{k}^{\varepsilon }$. (A slighlty modified version of) \texttt{%
mpt\_buildAdjacency} extracts adjacency information on this tesselation,
from which the vertices and the lengths of the sides between two cells can
be deduced. The library also incorporates convenient graphical routines.

\smallskip

\textbf{Error analysis} Some numerical examples are presented below, for
which relative errors in cell areas (i.e. deviation from the optimality
conditions) are given as well as a comparison between the tesselations
obtained with our method and the true solution. Another way to test our
method is as follows. Let us consider the set
\begin{equation*}
C:=\left\{z(\gamma):=\left(\int_{\mathbb{R}^2\times \mathbb{R}^2} x_1 y_1 d
\gamma, \int_{\mathbb{R}^2\times \mathbb{R}^2} x_2 y_2 d\gamma \right), \;
\gamma \in \Gamma(\mu, \nu)\right\}.
\end{equation*}
$C$ is a closed convex subset of $\mathbb{R}^2$ and it is strictly convex in
the sense that its boundary contains no line segment. Denoting by $%
\gamma_\eps$ the solution of (\ref{mke}), it is easy to check that $z(\eps%
):=z(\gamma_\eps)$ is an extreme point of $C$ and that $\eps z_1+z_2=\eps %
z_1(\eps)+ z_2(\eps)$ is the equation of a supporting line of $C$ at $z(\eps)
$ and this supporting line intersects $C$ only at $z(\eps)$. If we consider
the \emph{correlation} curve $\eps\in (0,1)\mapsto z(\eps)$ it can be
represented as the graph of a concave decreasing function whose slope (when
it exists) at point $z(\eps)$ is $-\eps$. In our numerical test, we will
also present graphs comparing the true concave correlation curve to the one
computed by our method.

\smallskip

We give three instances of executions of our algorithm, with samples of
respectively 5, 10, and 15 points. Taking weights $(\eps, \eps^{-1})$ with $\eps\in (0,+\infty)$ (rather than $(\eps, 1)$ with $\eps\in (0,1)$) we get the full evolution of the optimal transports from one Knothe's transport (horizontal strips) to the other (vertical strips). Further examples as well as videos can be found at \texttt{http://alfred.galichon.googlepages.com/anisotropic}. It should also be pointed  that our method does not approximate the solution of a single optimal transportation problem but a whole family of such problems (which actually explains relatively high running times).

\textbf{Five sample points. }We take as our sample set a sample of five
points. We get the following errors:
$$
\begin{tabular}{||l|l|l|l|l|l|l||}
\hline\hline
\textbf{\# steps} & \multicolumn{5}{|l|}{\textbf{Relative errors in cell
areas}} & \textbf{Time} \\ \hline
100 & -4.41\% & 2.66\% & 3.41\% & -2.46\% & 0.80\% & 66 s \\ \hline
500 & -0.88\% & 0.54\% & 0.68\% & -0.49\% & 0.16\% & 349 s \\ \hline\hline
\end{tabular}
$$
for which we draw in Figure \ref{Figure:Sample5} the partition obtained for $%
\varepsilon =1$ using an exact method, as well as the true evolutions of the
componentwise correlations of the $x$'s and the $y$'s from $\varepsilon =0$
until $\varepsilon =1$.

\begin{figure}[h]
\begin{center}
\epsfig{figure=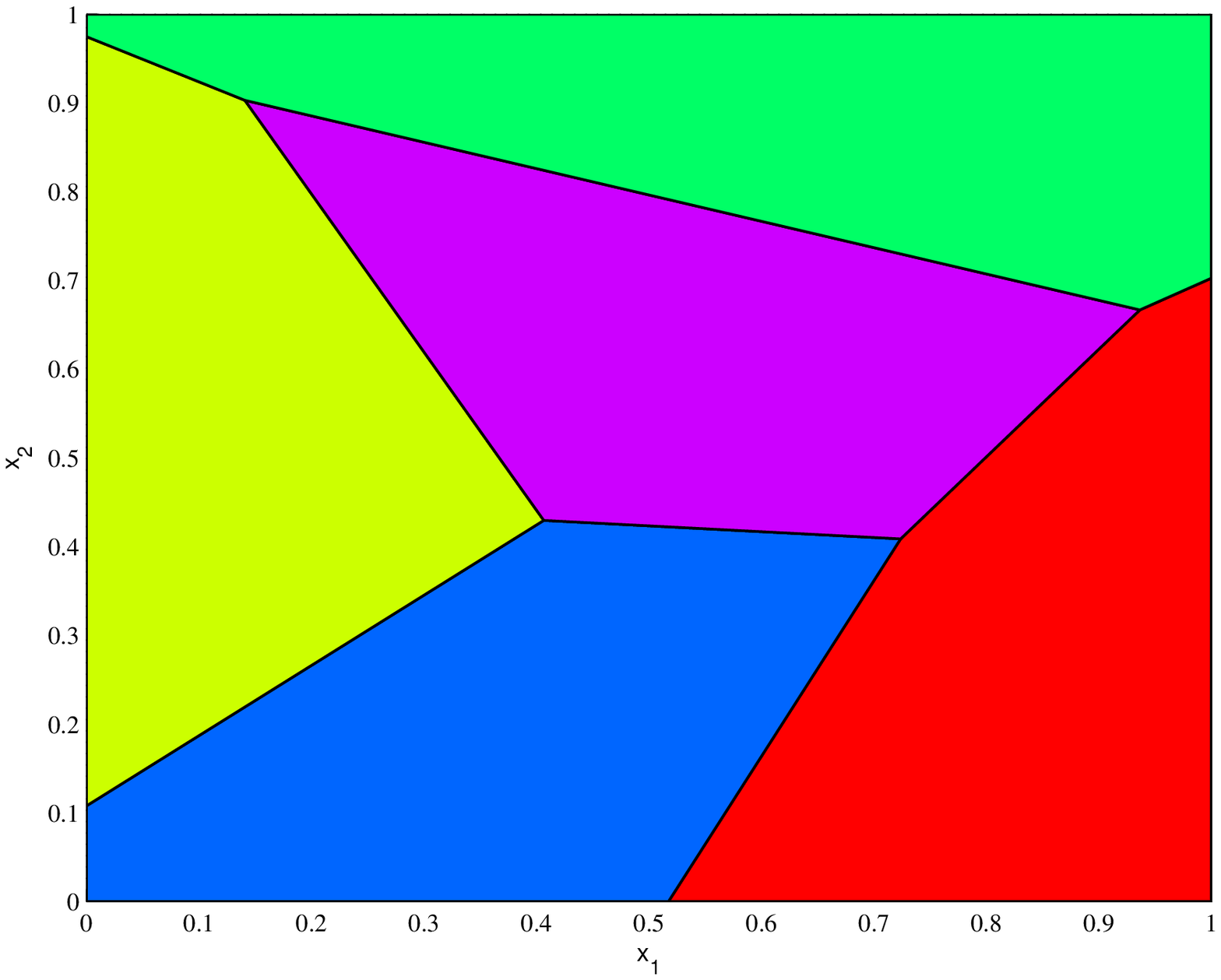, width=2.5in,height=2.5in} %
\epsfig{figure=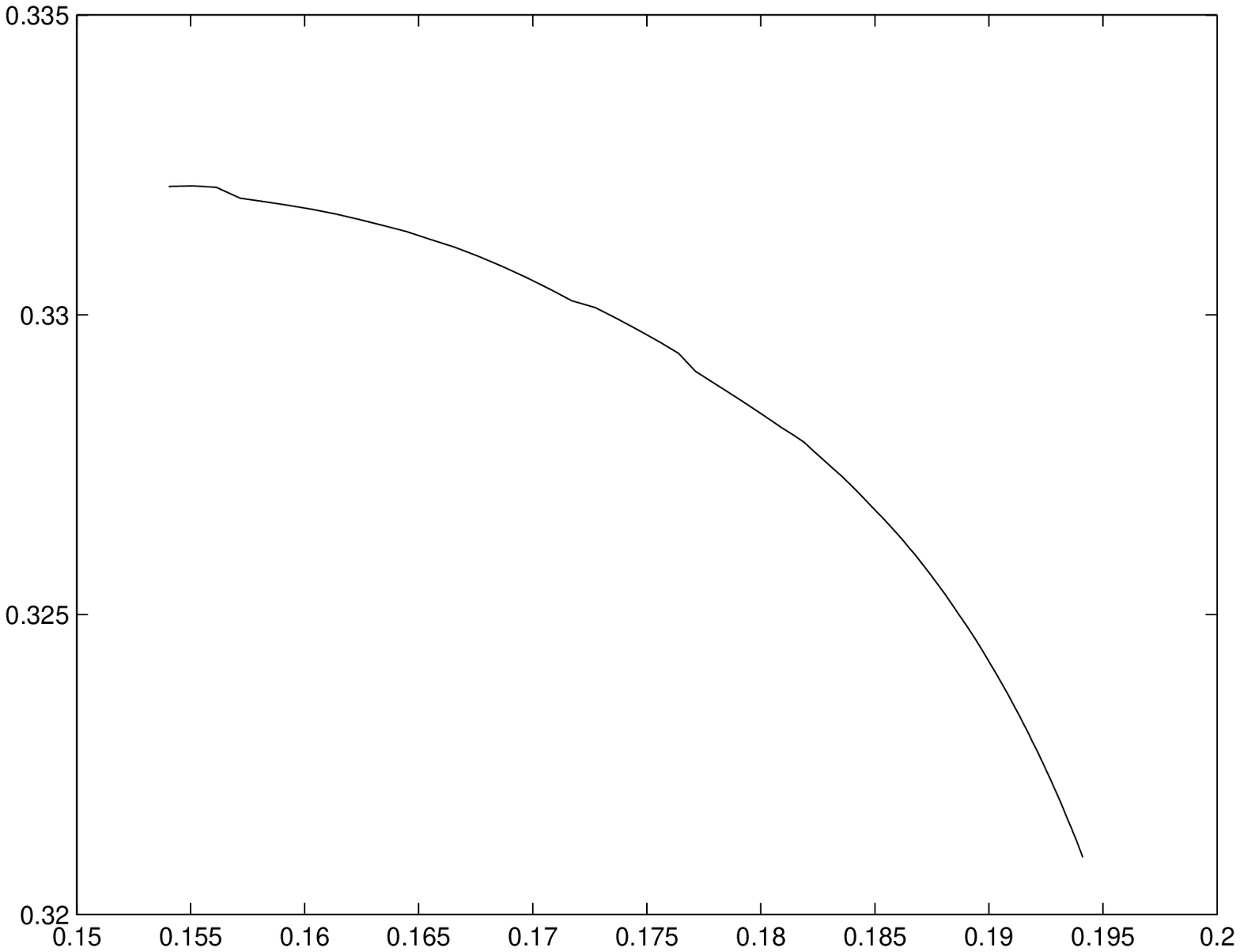,width=2.5in,height=2.5in} %
\epsfig{figure=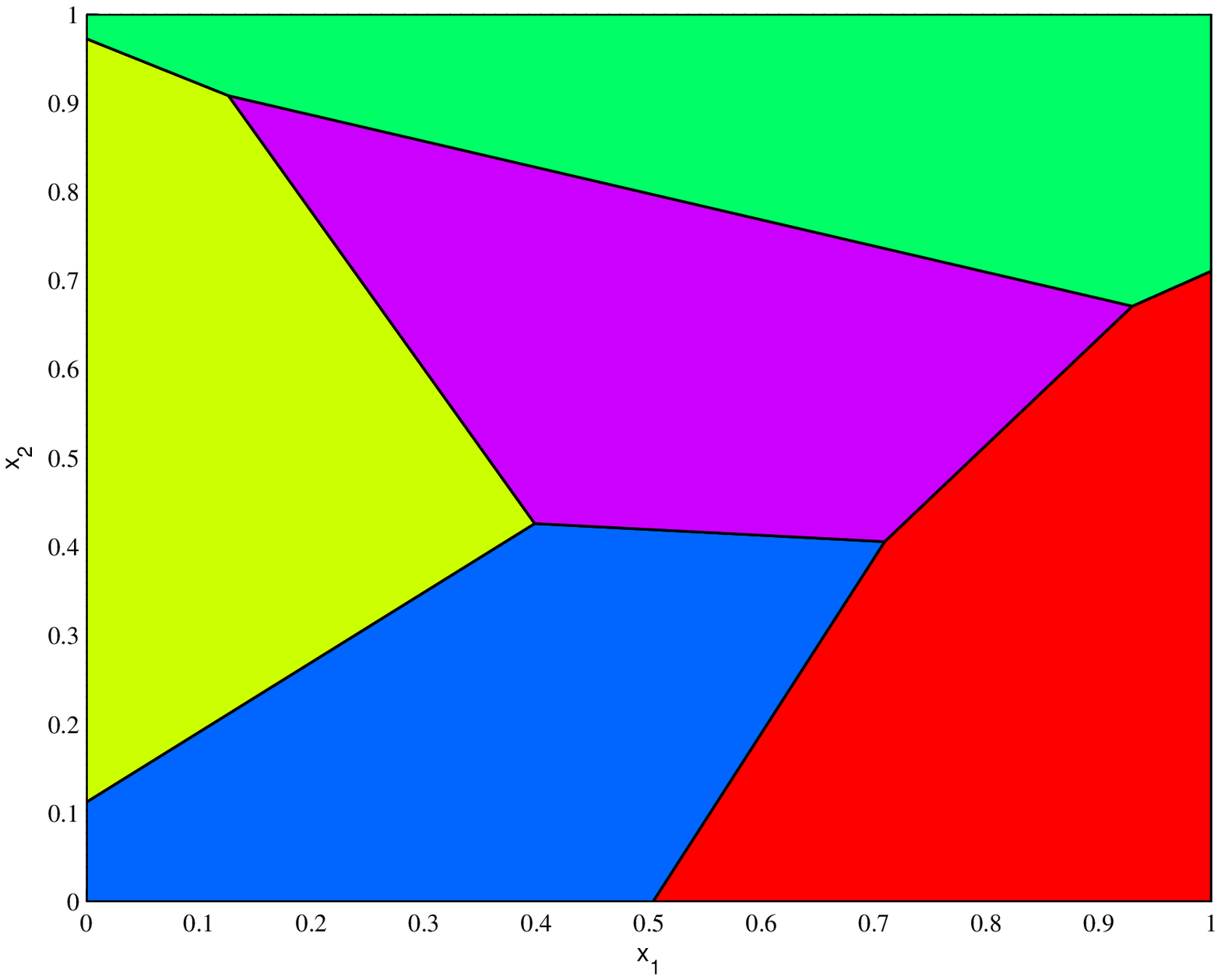,width=2.5in,height=2.5in} %
\epsfig{figure=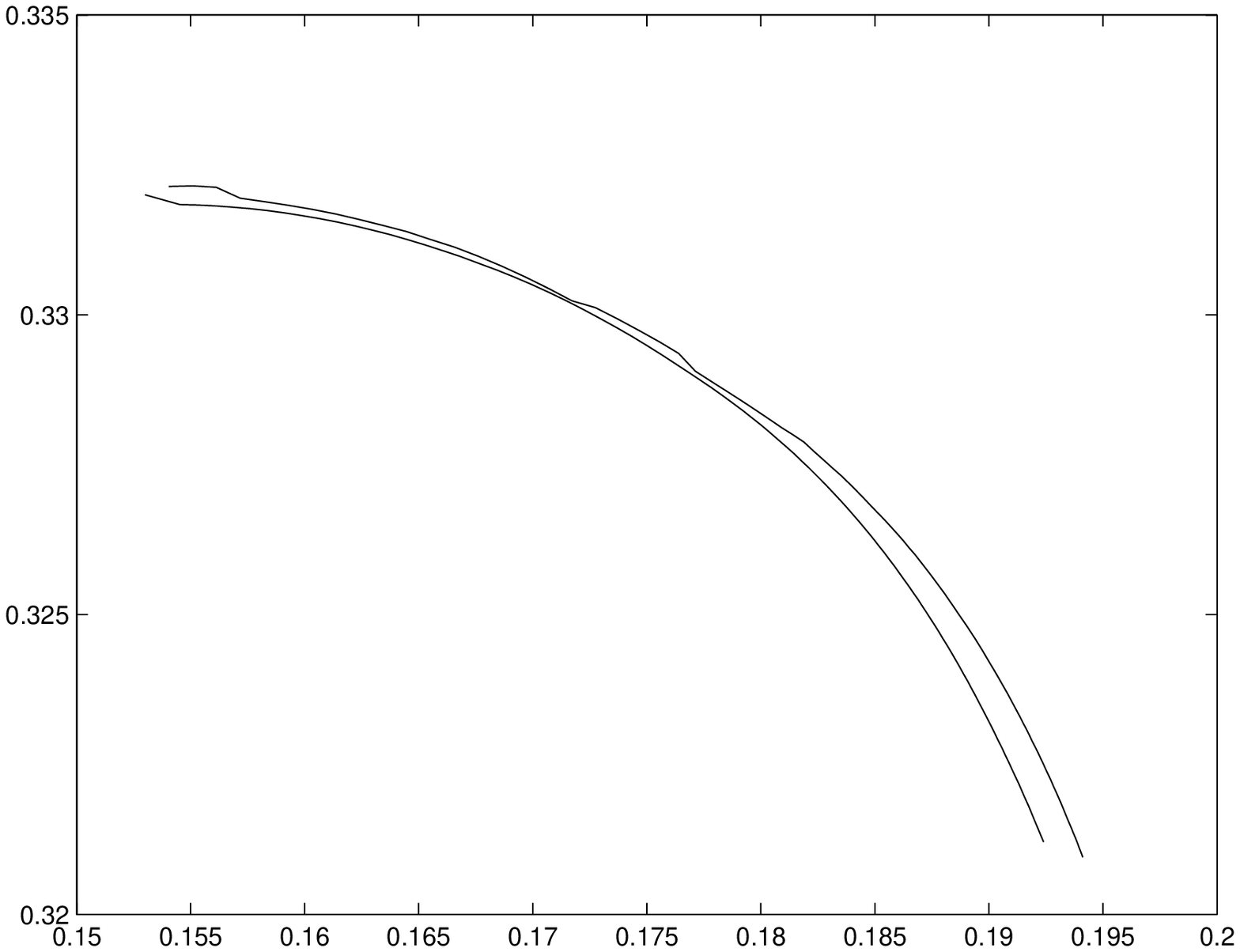, width=2.5in,height=2.5in} %
\epsfig{figure=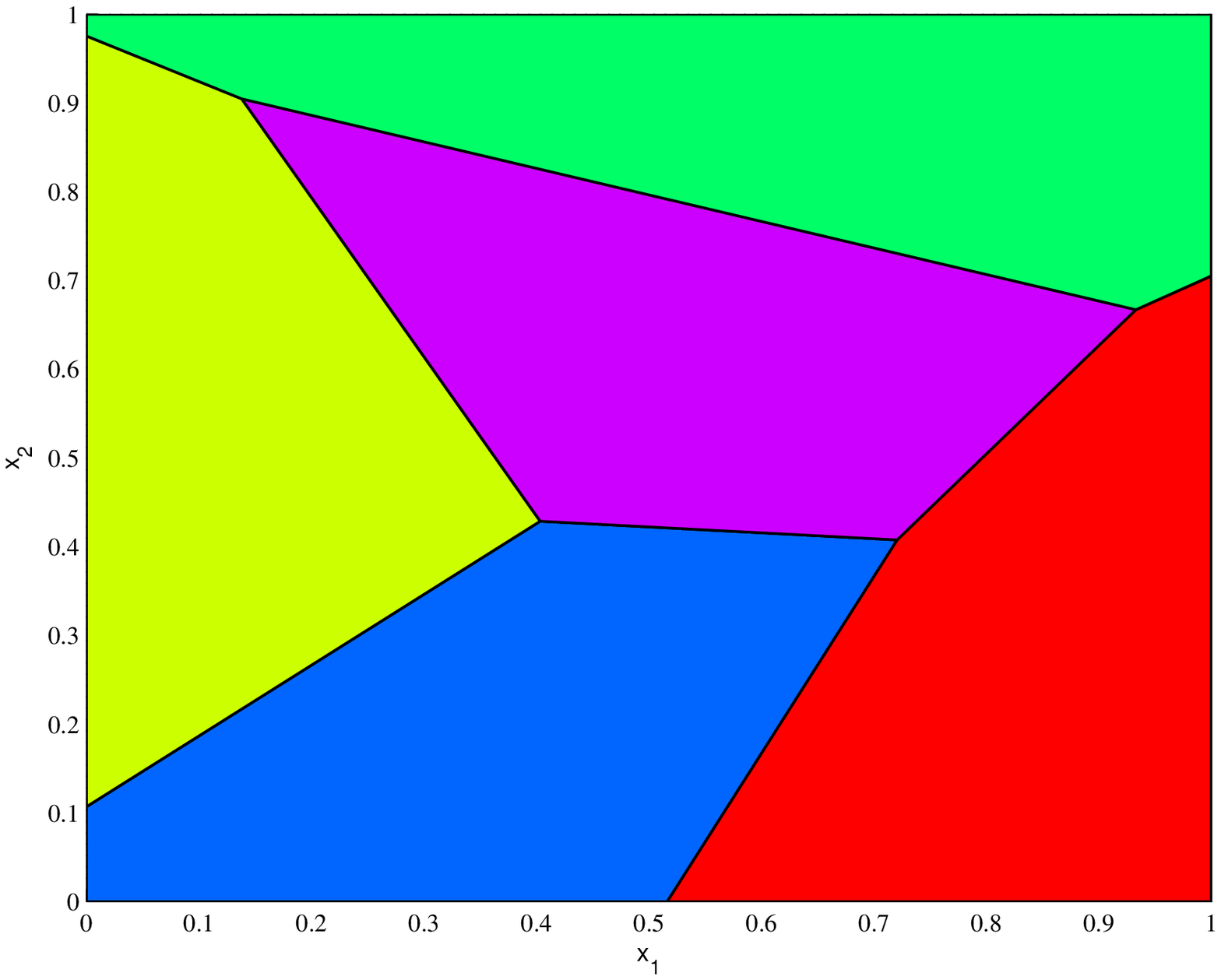,width=2.5in,height=2.5in} %
\epsfig{figure=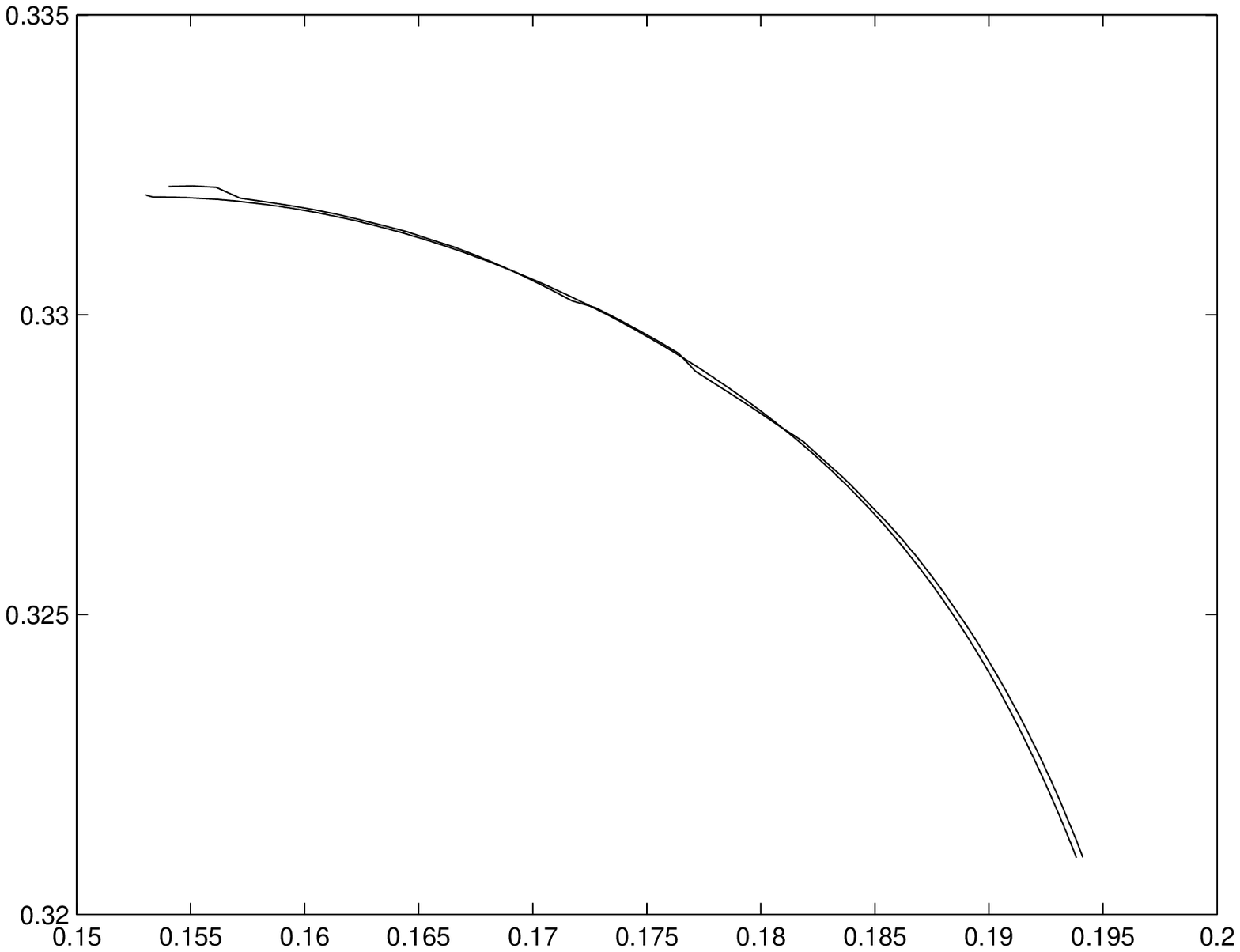, width=2.5in,height=2.5in}
\end{center}
\caption{Five sample points. Top row: exact algorithm (gradient method).
Middle row: continuation algorithm, 100 steps. Bottom row: continuation
algorithm, 500 steps.}
\label{Figure:Sample5}
\end{figure}

\begin{figure}[h]
\begin{center}
\epsfig{figure=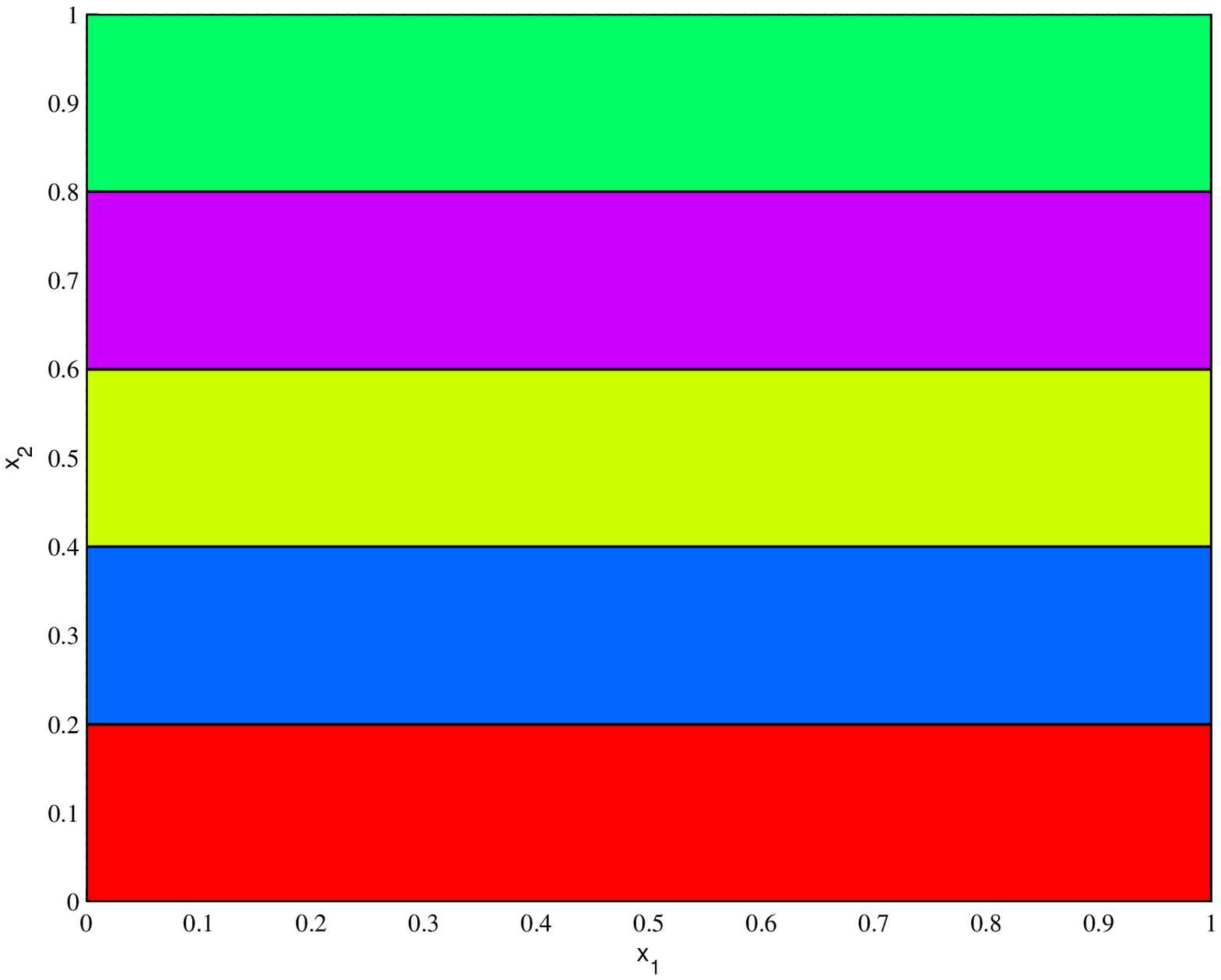, width=1.5in,height=1.5in} %
\epsfig{figure=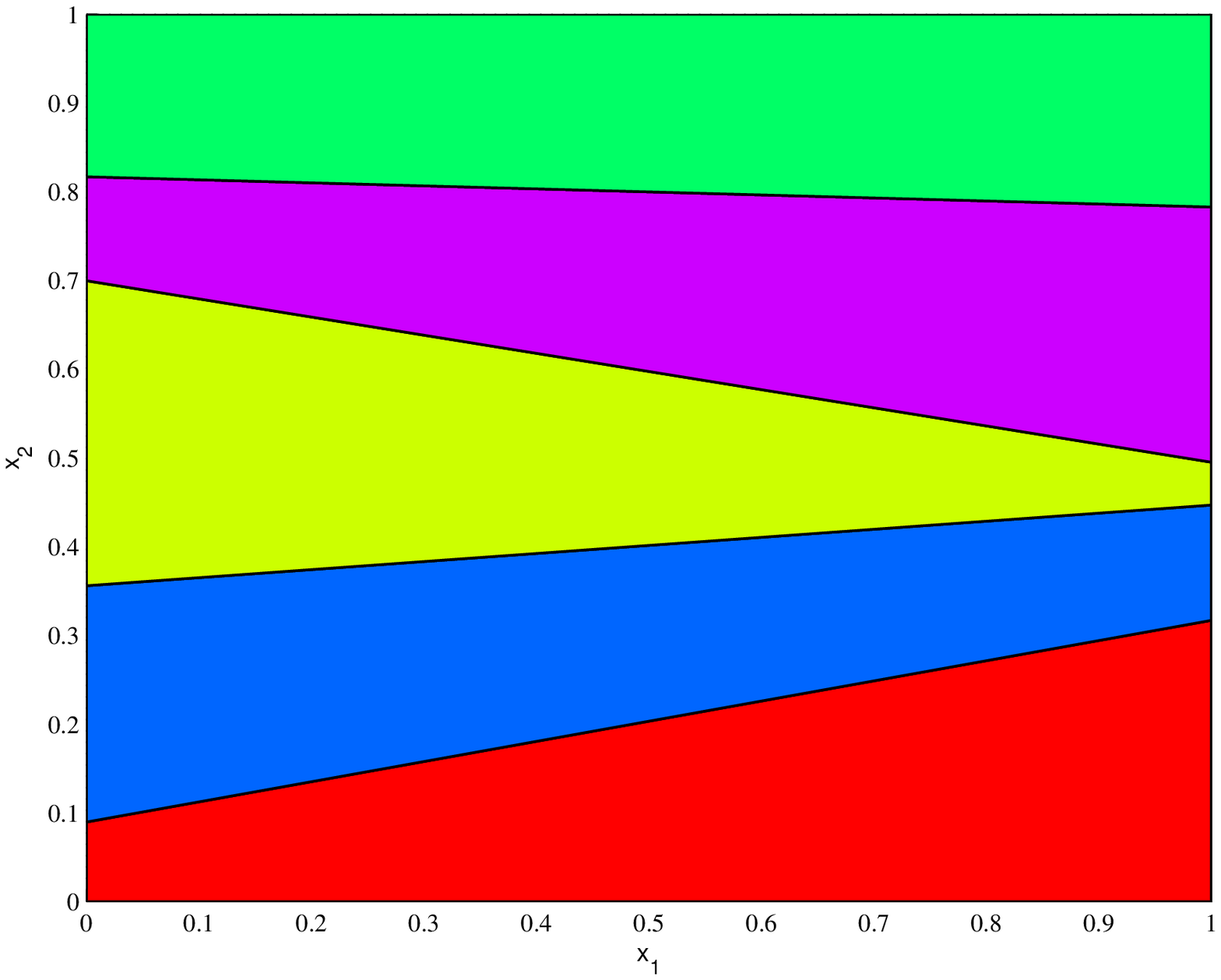, width=1.5in,height=1.5in} %
\epsfig{figure=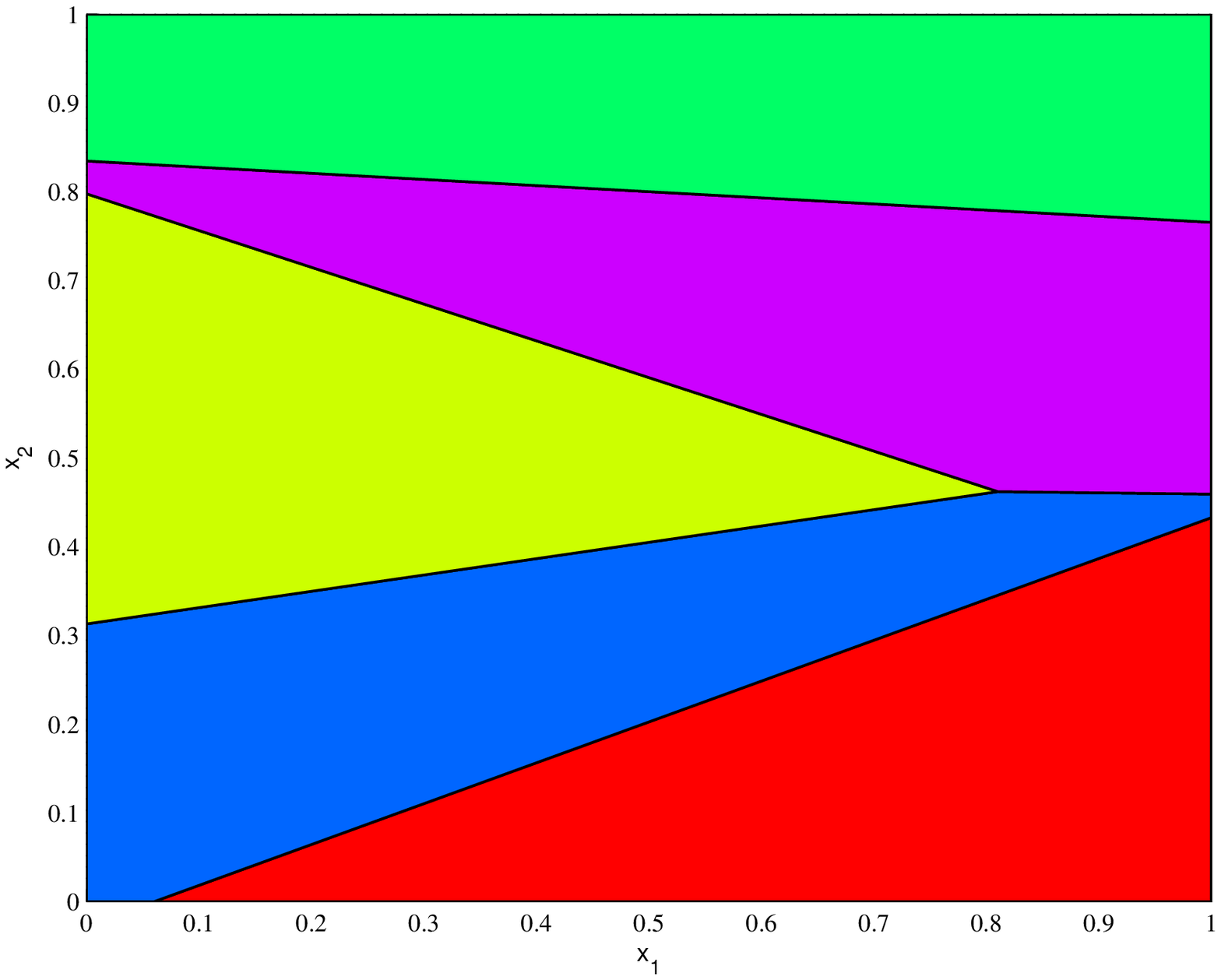,width=1.5in,height=1.5in} %
\epsfig{figure=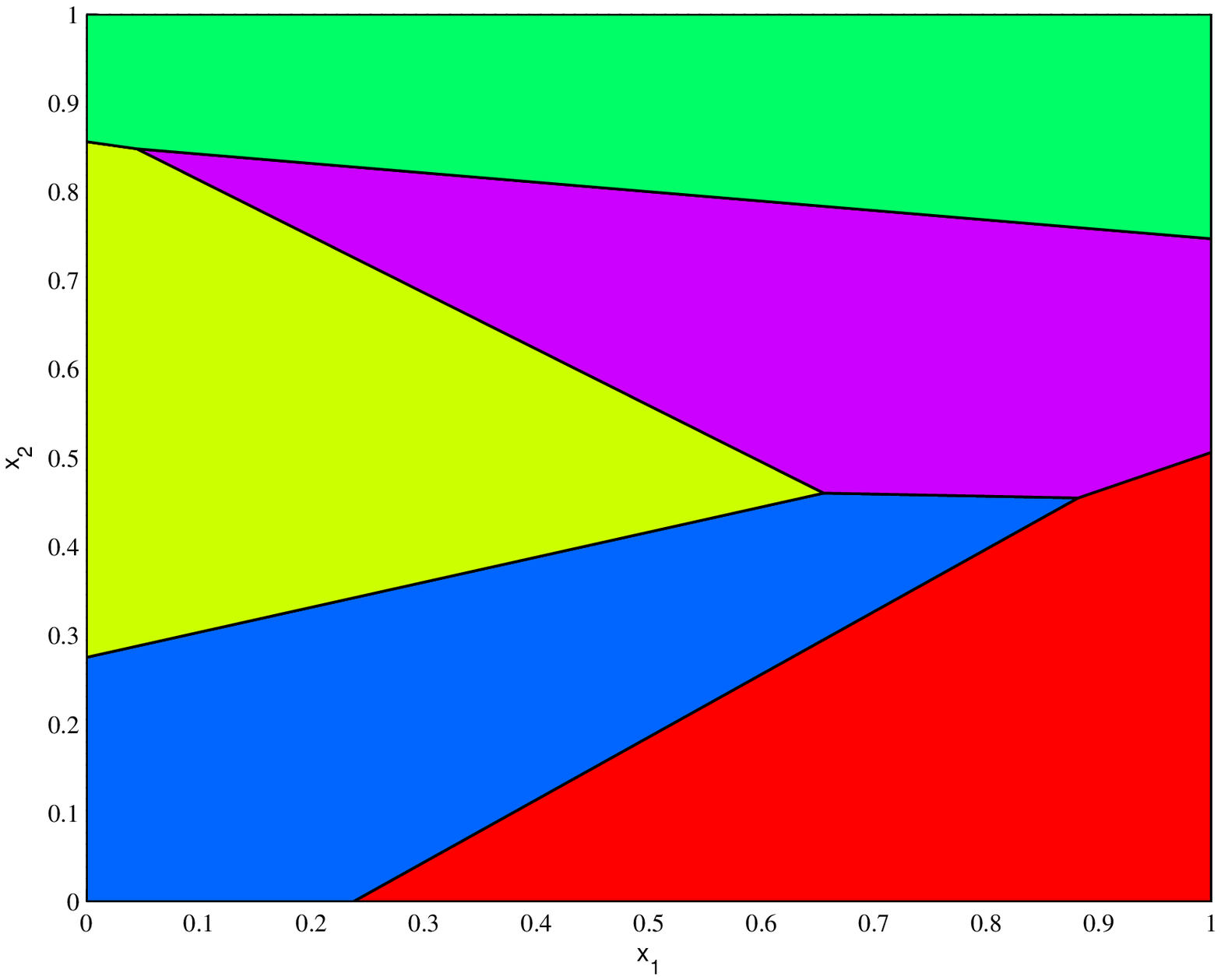,width=1.5in,height=1.5in} %
\epsfig{figure=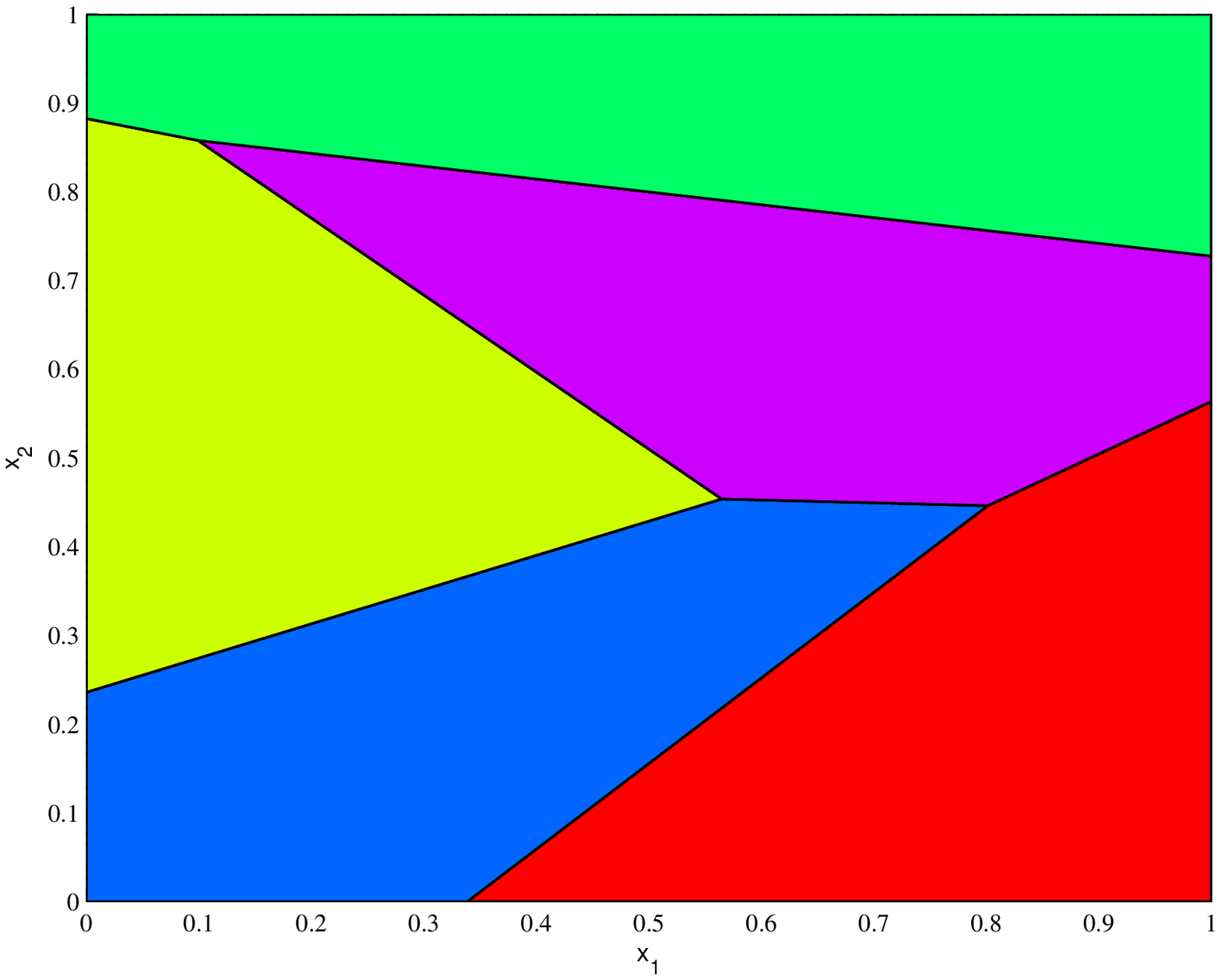,width=1.5in,height=1.5in} %
\epsfig{figure=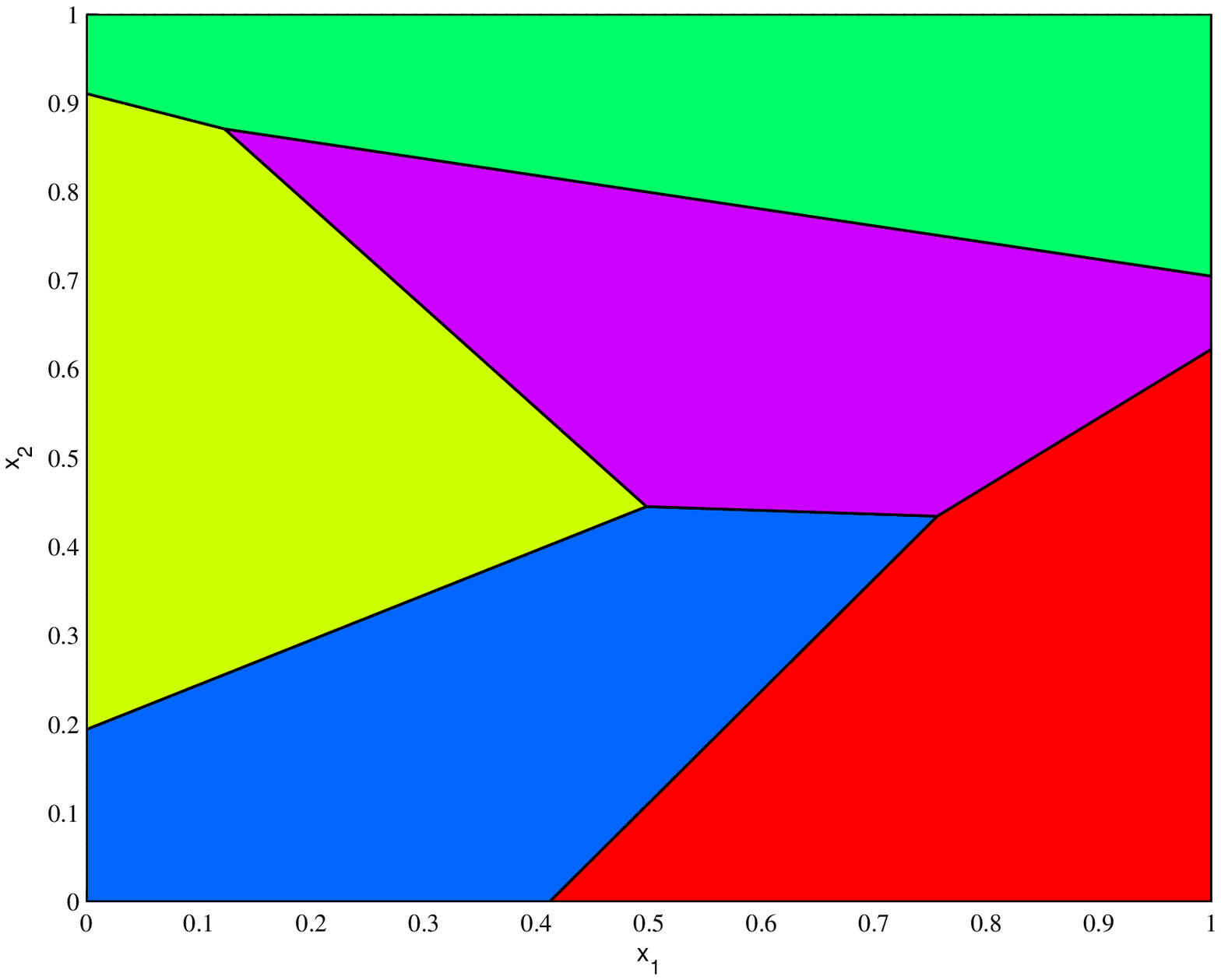,width=1.5in,height=1.5in} %
\epsfig{figure=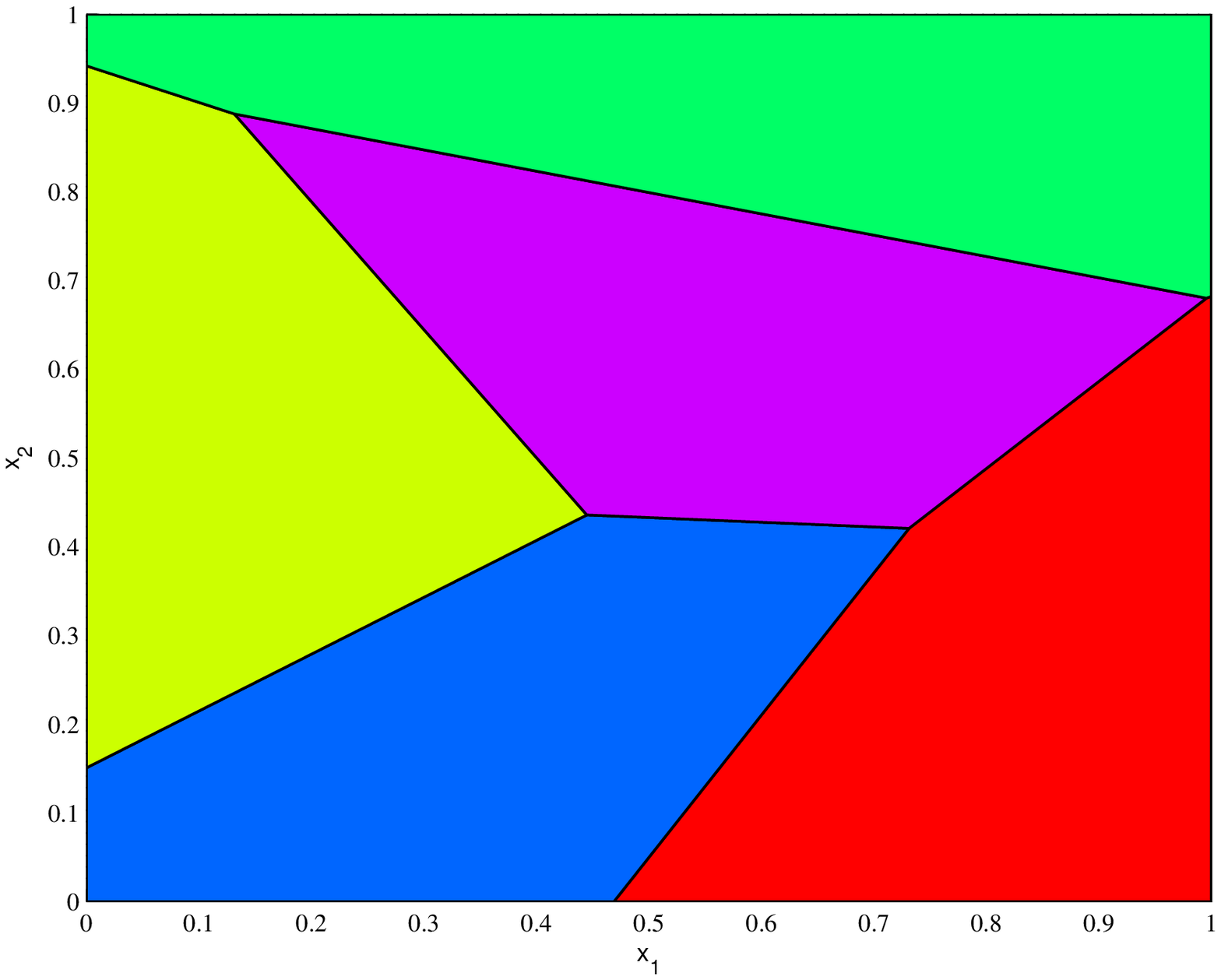,width=1.5in,height=1.5in} %
\epsfig{figure=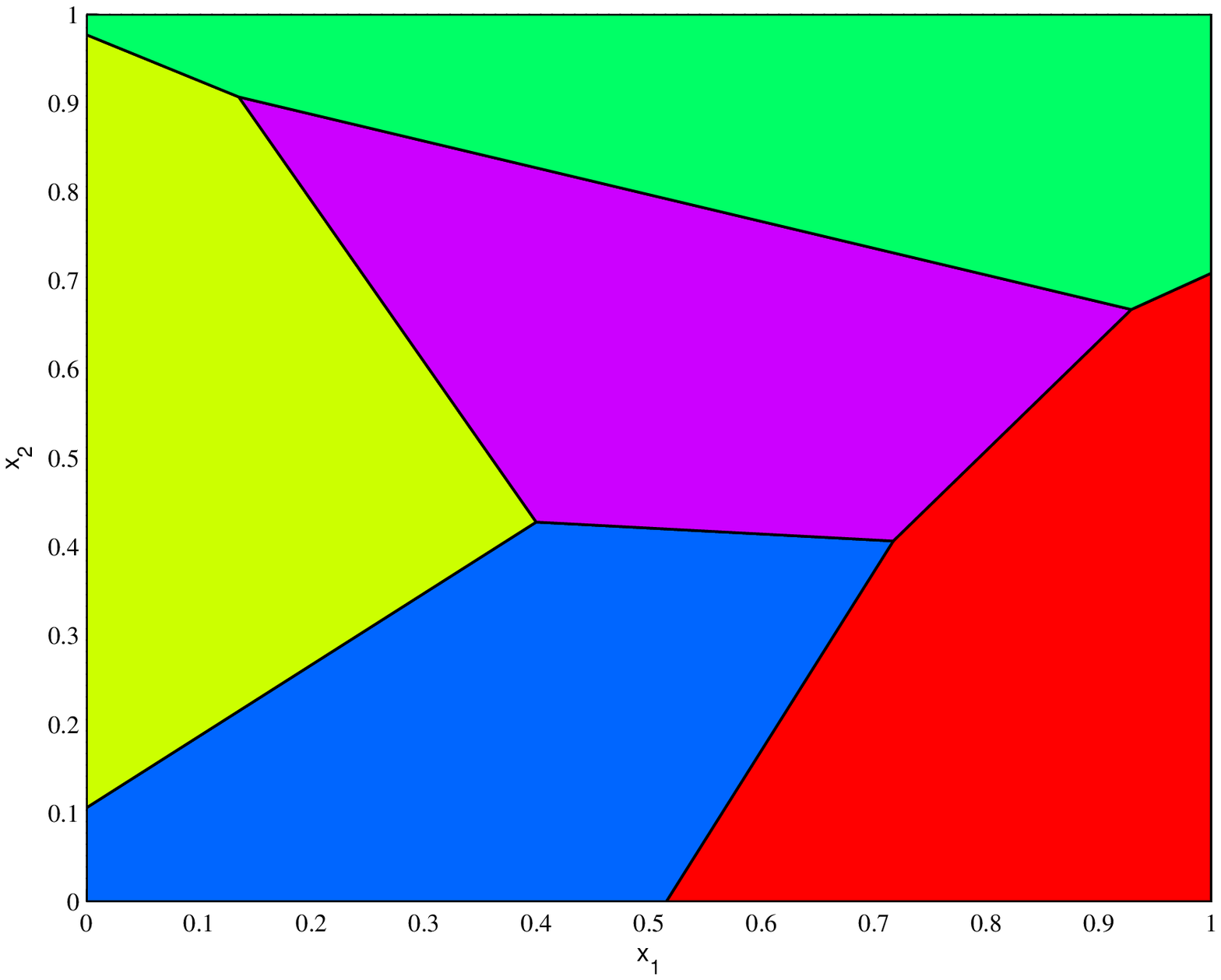,width=1.5in,height=1.5in} %
\epsfig{figure=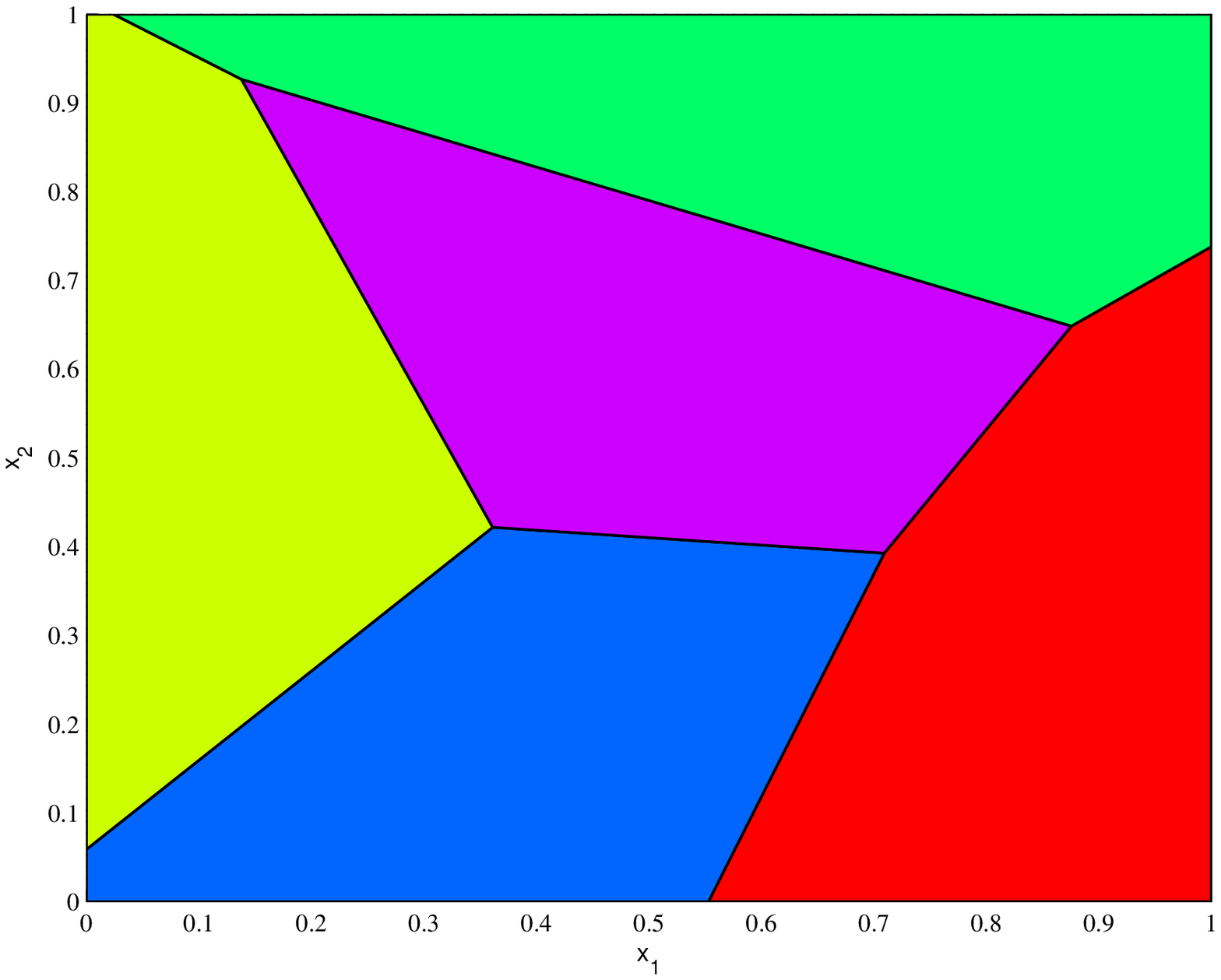,width=1.5in,height=1.5in} %
\epsfig{figure=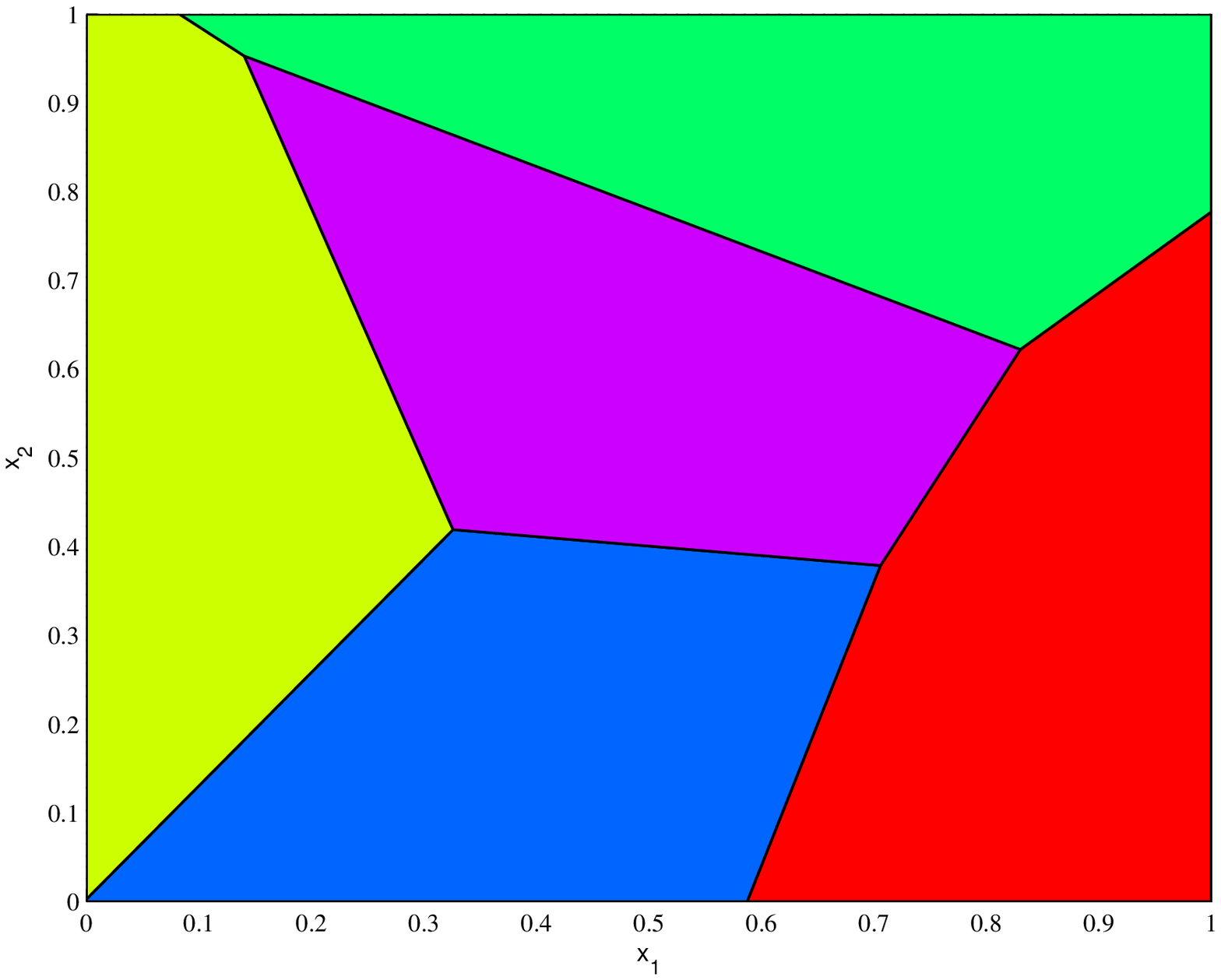,width=1.5in,height=1.5in} %
\epsfig{figure=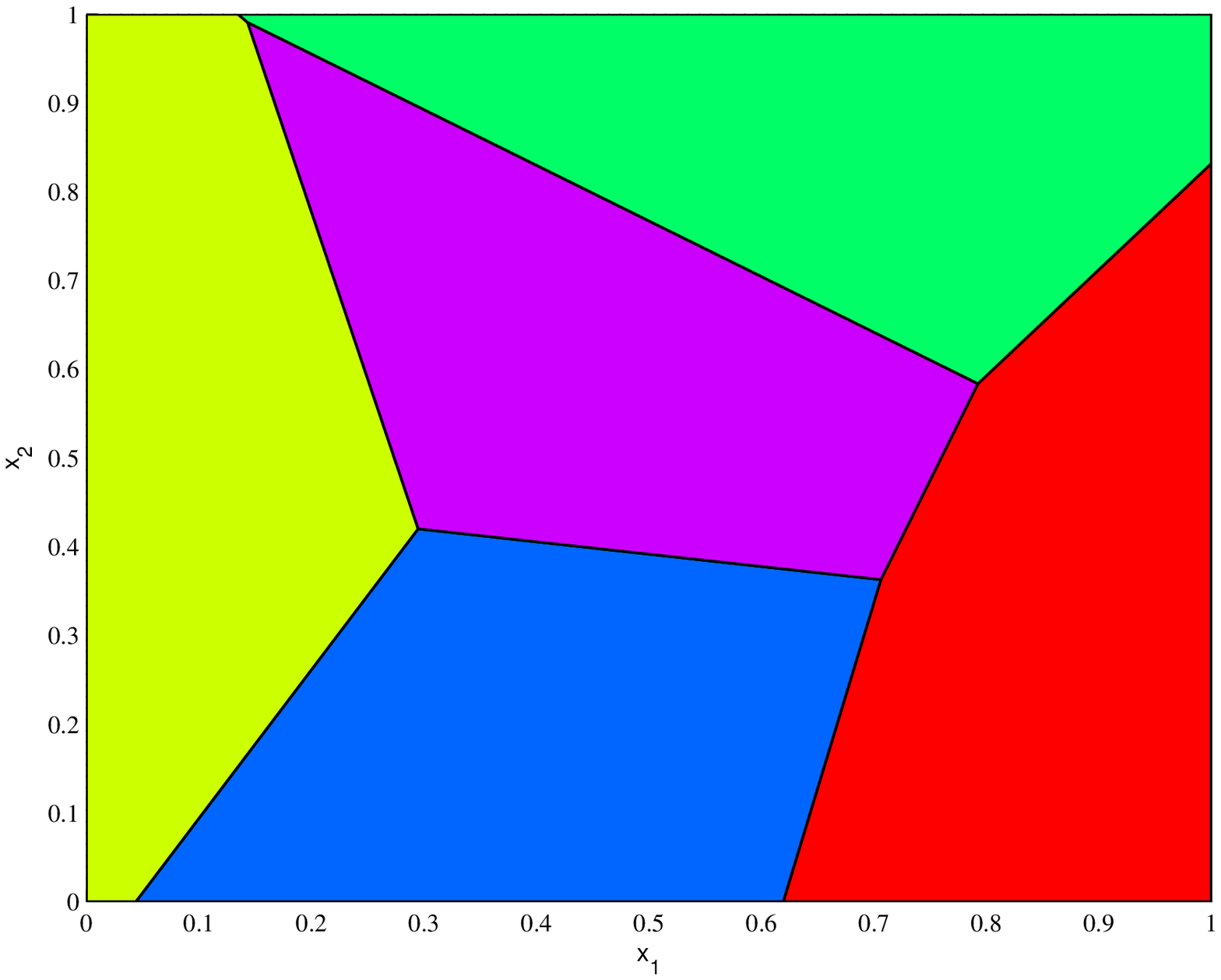,width=1.5in,height=1.5in} %
\epsfig{figure=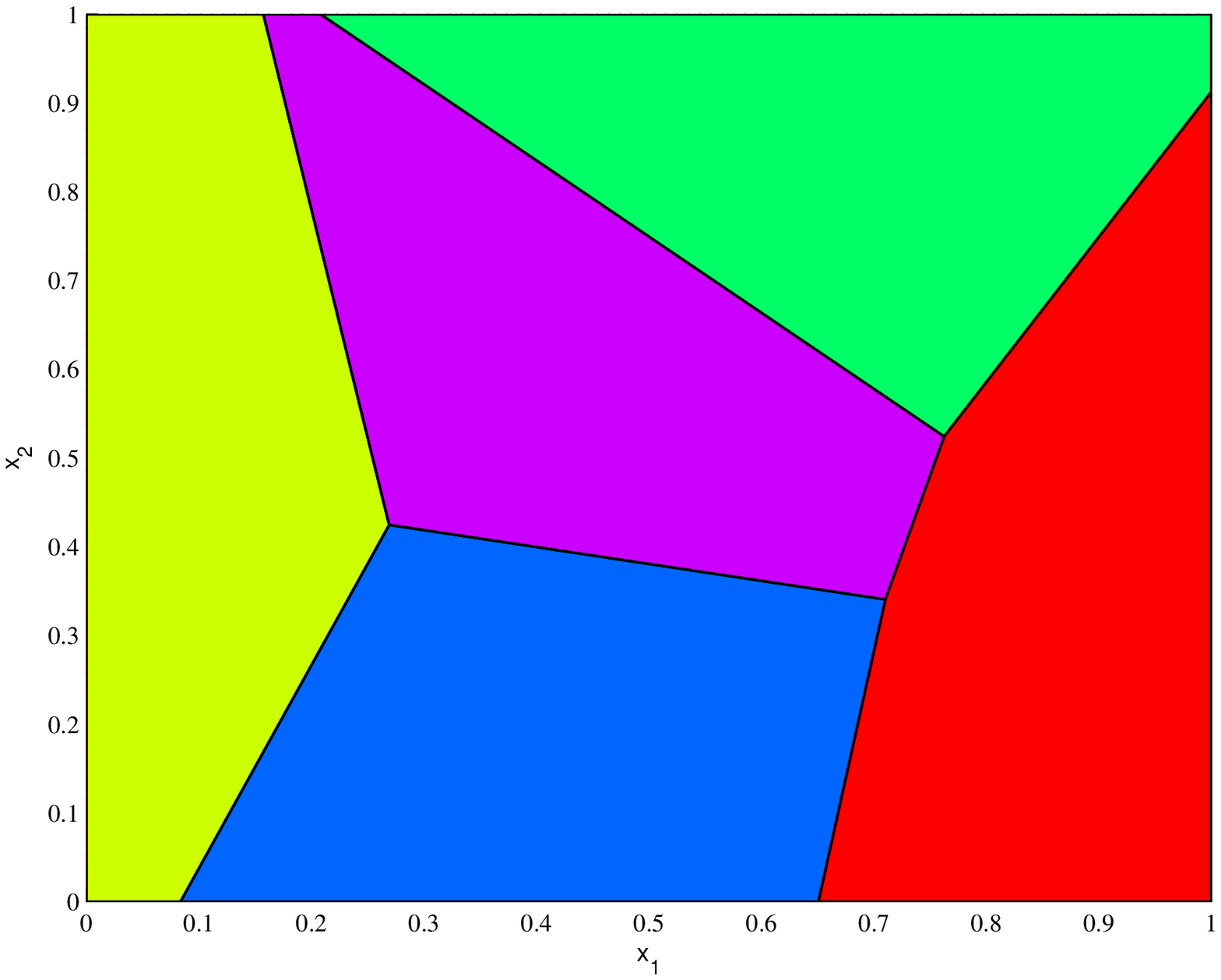,width=1.5in,height=1.5in} %
\epsfig{figure=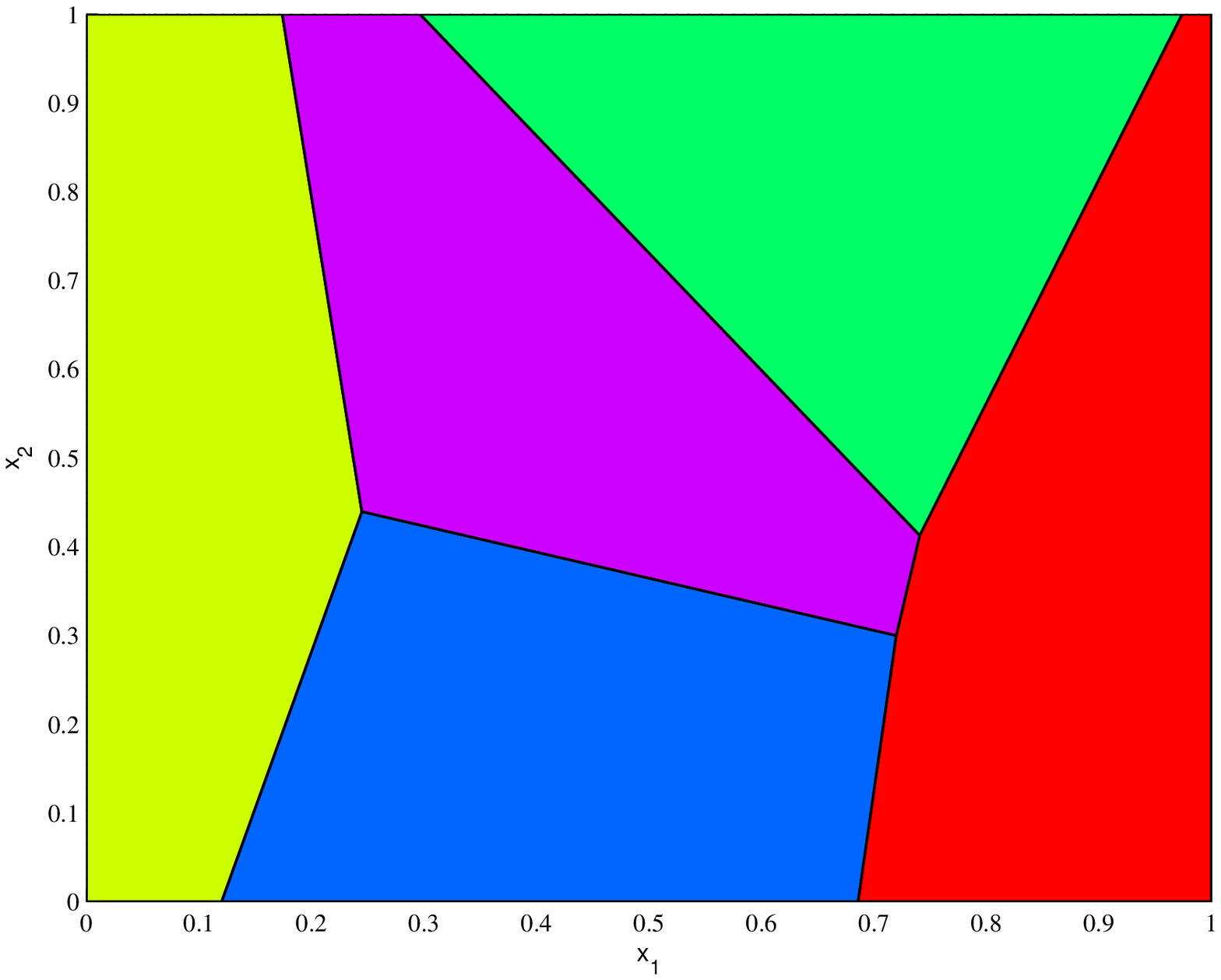,width=1.5in,height=1.5in} %
\epsfig{figure=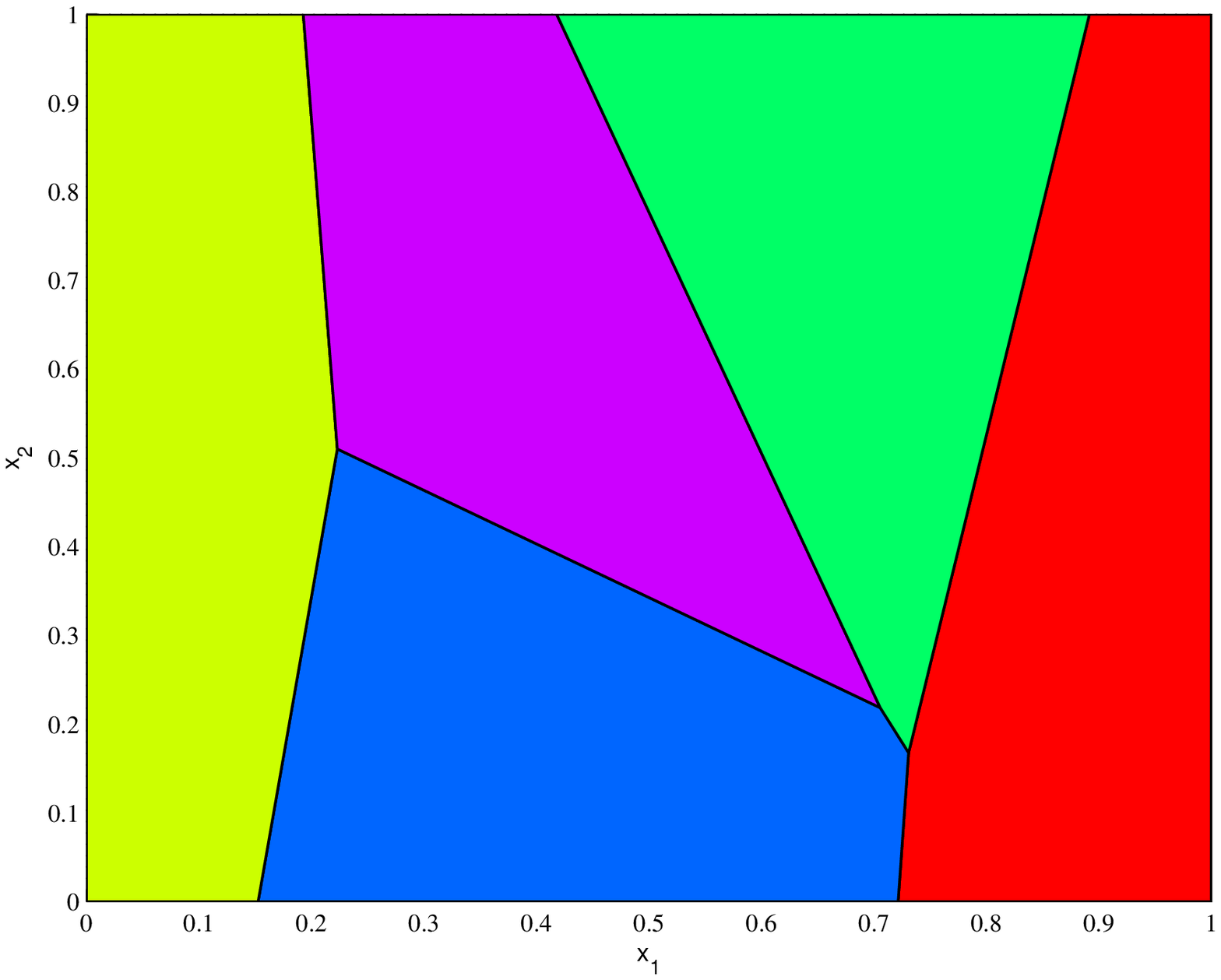,width=1.5in,height=1.5in} %
\epsfig{figure=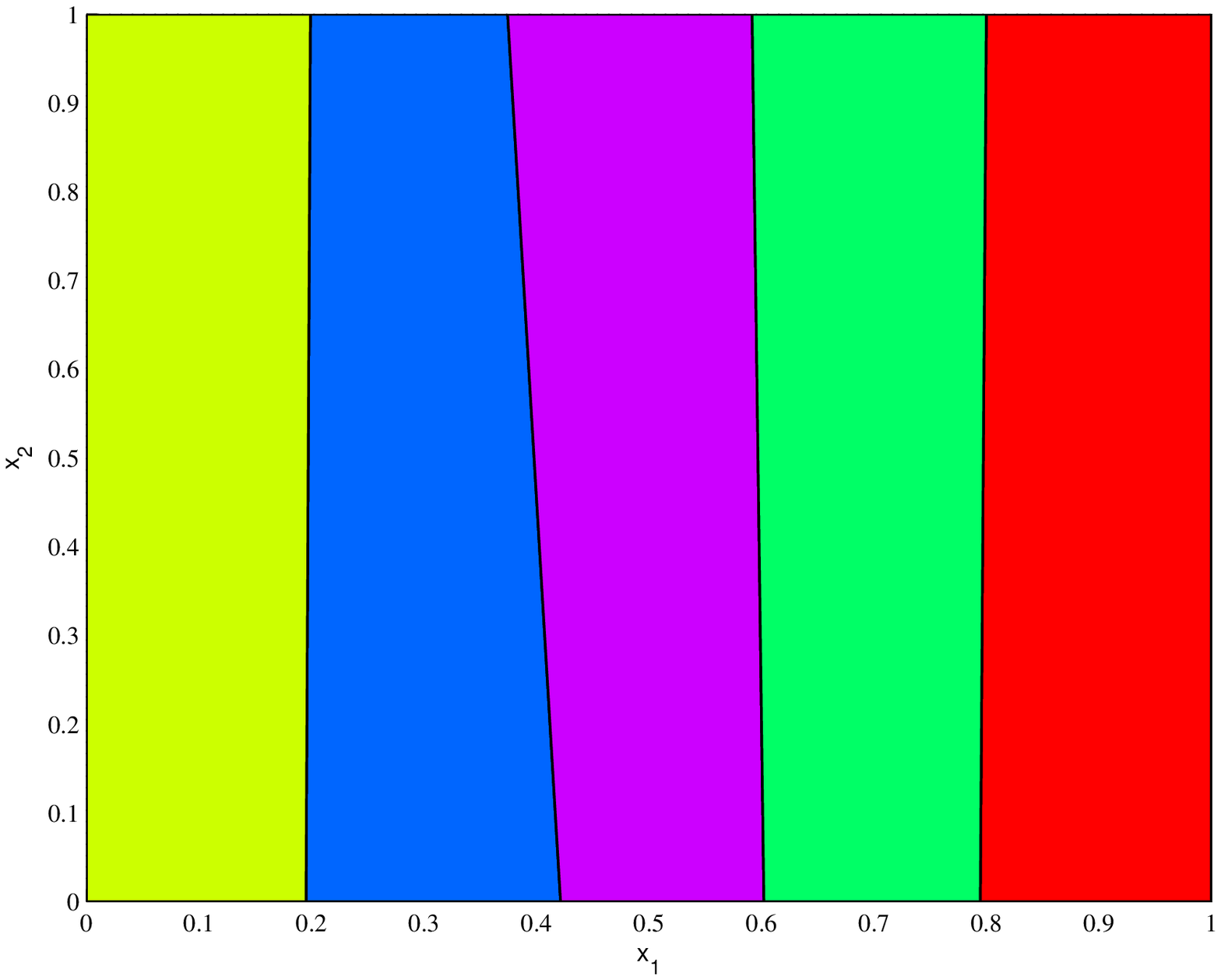,width=1.5in,height=1.5in}
\end{center}
\caption{Five sample points: evolution of the tesselation for $\protect%
\varepsilon=0$ to $\protect\varepsilon=+\infty$ (from top left to bottom
right).}
\label{Figure:Sample5evol}
\end{figure}

Running the continuation algorithm with 100  and 500 discretization steps  we obtained  partition sketched below, where the relative error on the cells area is inferior to 5\%. The evolution
of the componentwise correlations from $\varepsilon =0$ until $\varepsilon =1
$ are also sketched and compared with their exact conterparts (the above
curve). Running the algorithm with 500 discretization steps on a standard laptop took 349
seconds, and yields to the results below, where the maximal relative error
on the cells area is 1\%. The evolution of the componentwise correlations from $%
\varepsilon =0$ until $\varepsilon =1$ are also sketched and compared with
their exact counterparts.

\textbf{Ten sample points.} Taking a sample of ten
points, we obtain the following errors:
$$
\begin{tabular}{||l|l|l|l|l|l|l||}
\hline\hline
\textbf{\# steps} & \multicolumn{5}{|l|}{\textbf{Relative errors in cell
areas}} & \textbf{Time} \\ \hline
2,000 & -1.37\% & -12\% & -2.44\% & -0.09\% & 2.27\% & $\leq$ 1 h  \\ \hline
3,000 & -0.92\% & -8.23\% & 1.63\% & 0.06\% & 1.52\% & $\leq$ 1,5 h \\ \hline
& \multicolumn{6}{|l||}{\textbf{(continued)}} \\ \hline
2,000 & 0.27\% & 5.71\% & 2.12\% & 7.36\% & -1.58\% & $\leq$ 1 h \\ \hline
3,000 & 0.18\% & 3.84\% & 1.42\% & 4.94\% & -1.05\% & $\leq$ 1,5 h \\ \hline\hline
\end{tabular}
$$

We draw in Figure \ref{Figure:Sample10} the partition obtained for $%
\varepsilon =1$ using an exact method, as well as the true evolutions of the
componentwise correlations of the $x$'s and the $y$'s from $\varepsilon =0$
until $\varepsilon =1$.

\begin{figure}[h]
\begin{center}
\epsfig{figure=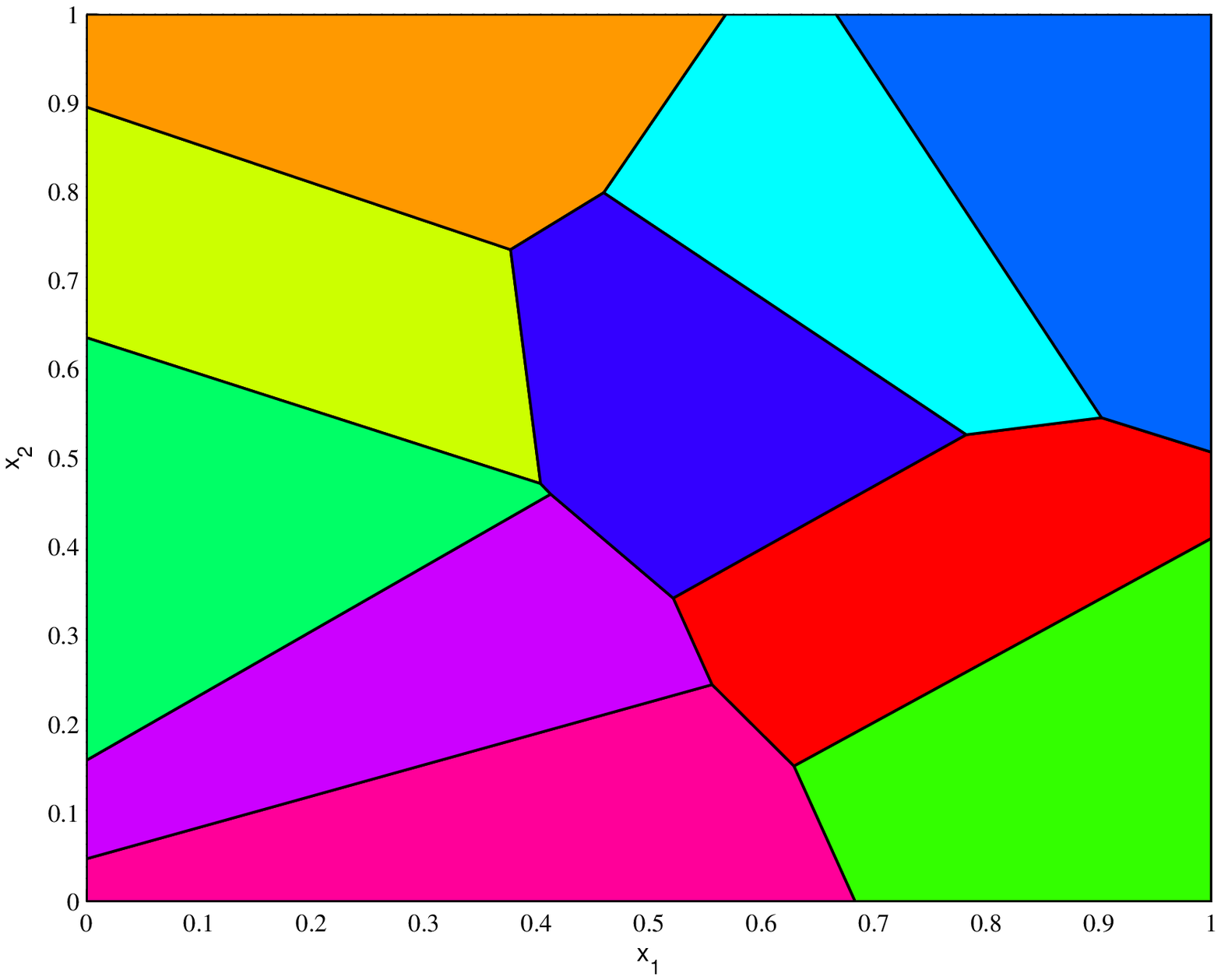, width=2.5in,height=2.5in} %
\epsfig{figure=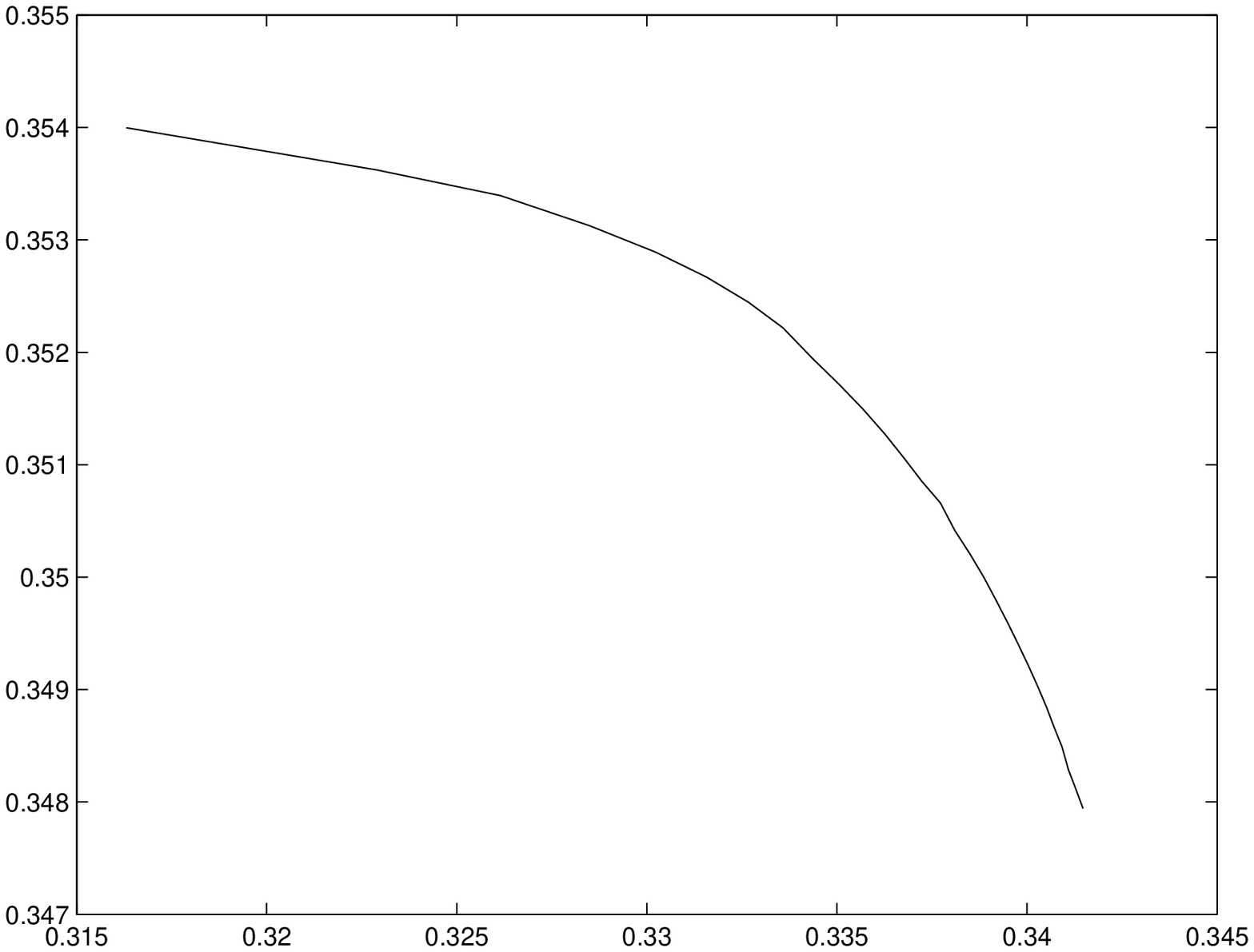,width=2.5in,height=2.5in} %
\epsfig{figure=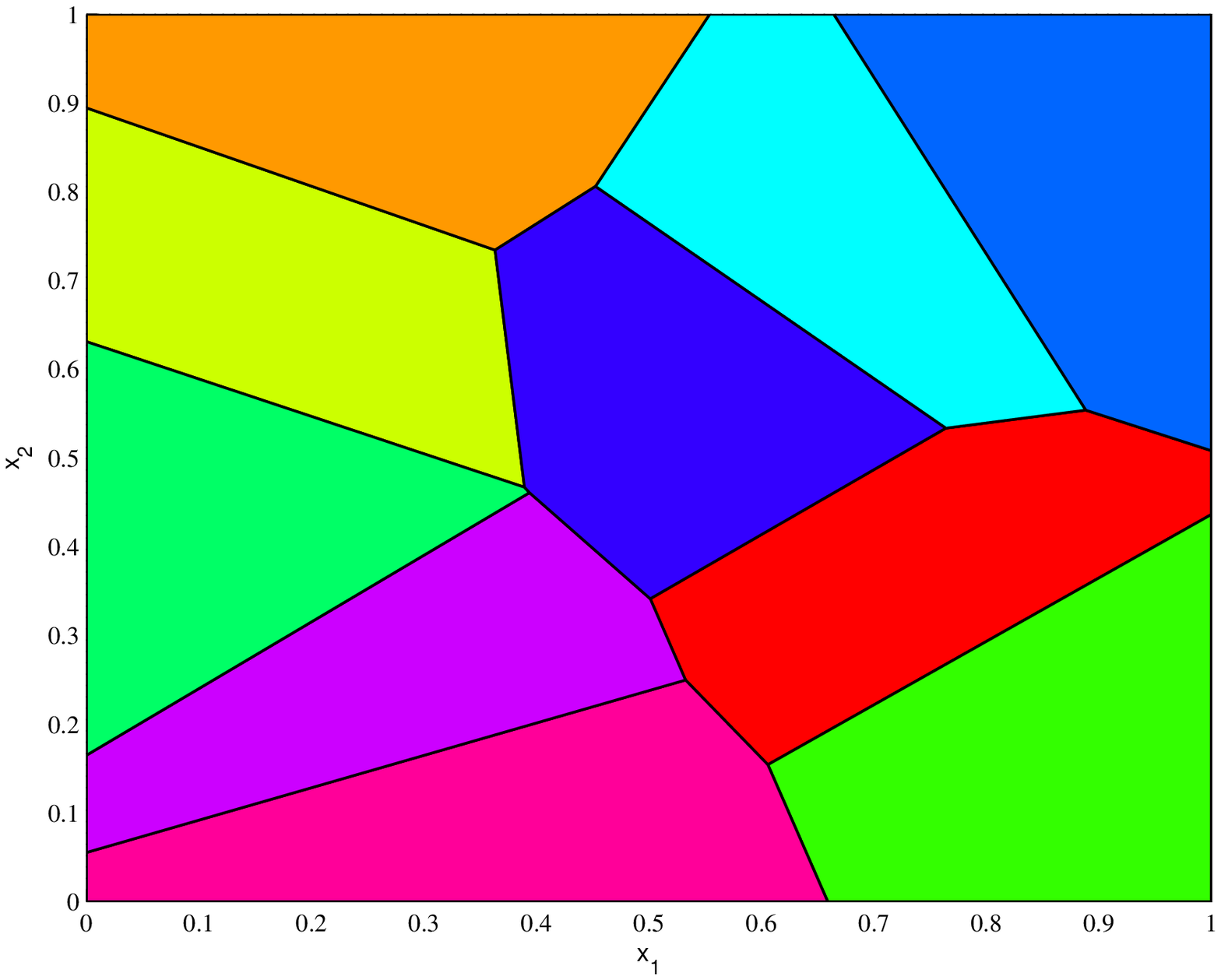,width=2.5in,height=2.5in} %
\epsfig{figure=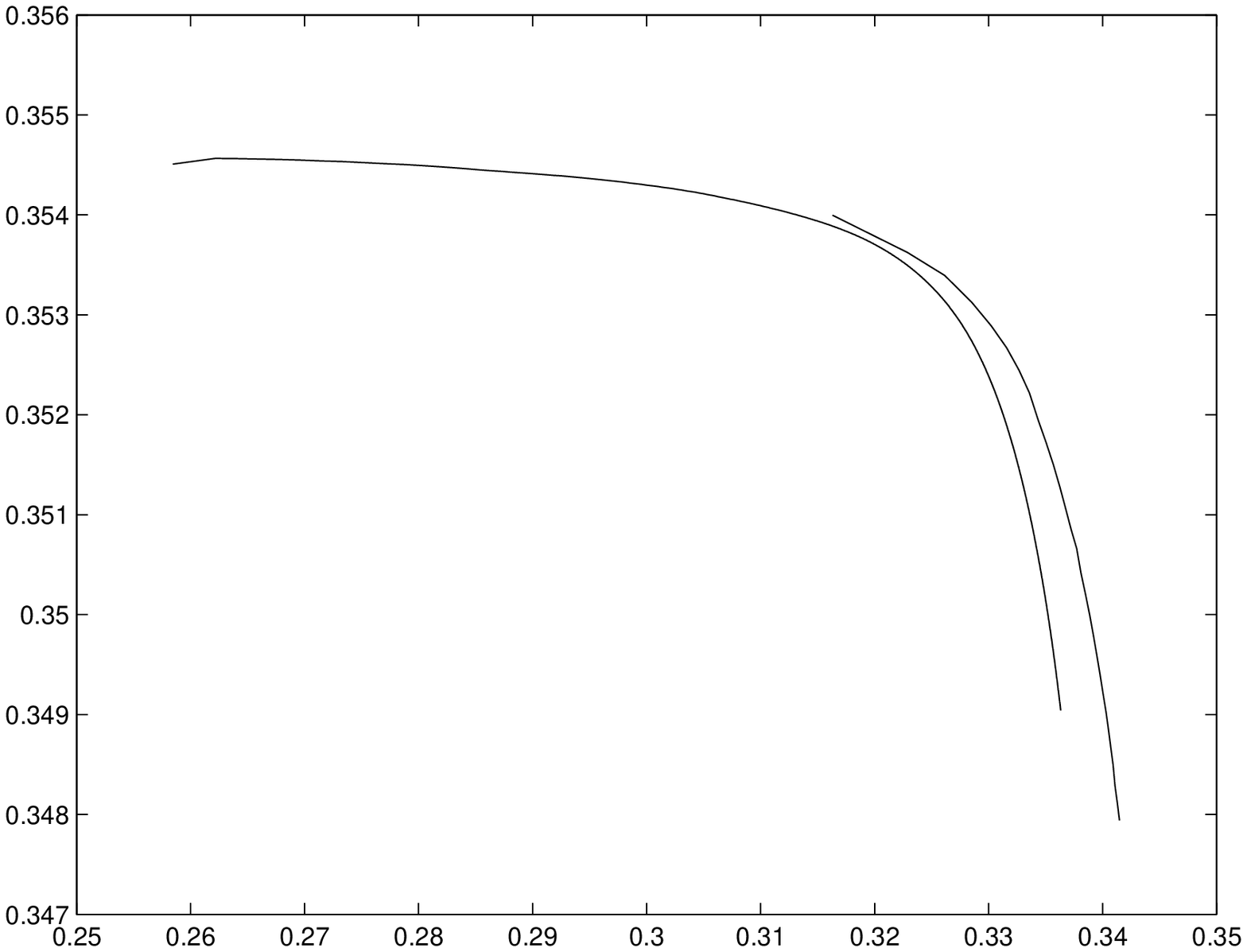, width=2.5in,height=2.5in} %
\epsfig{figure=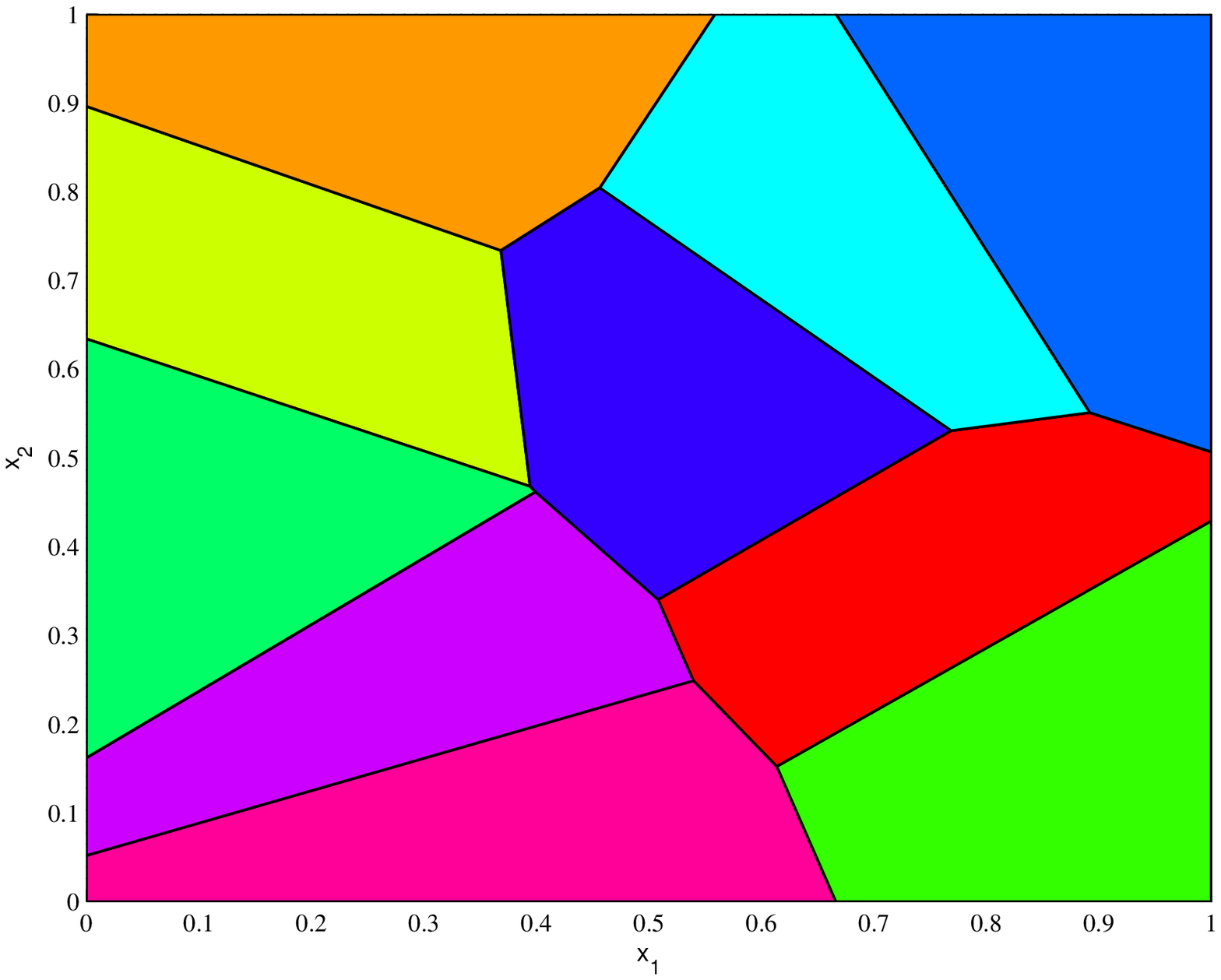,width=2.5in,height=2.5in} %
\epsfig{figure=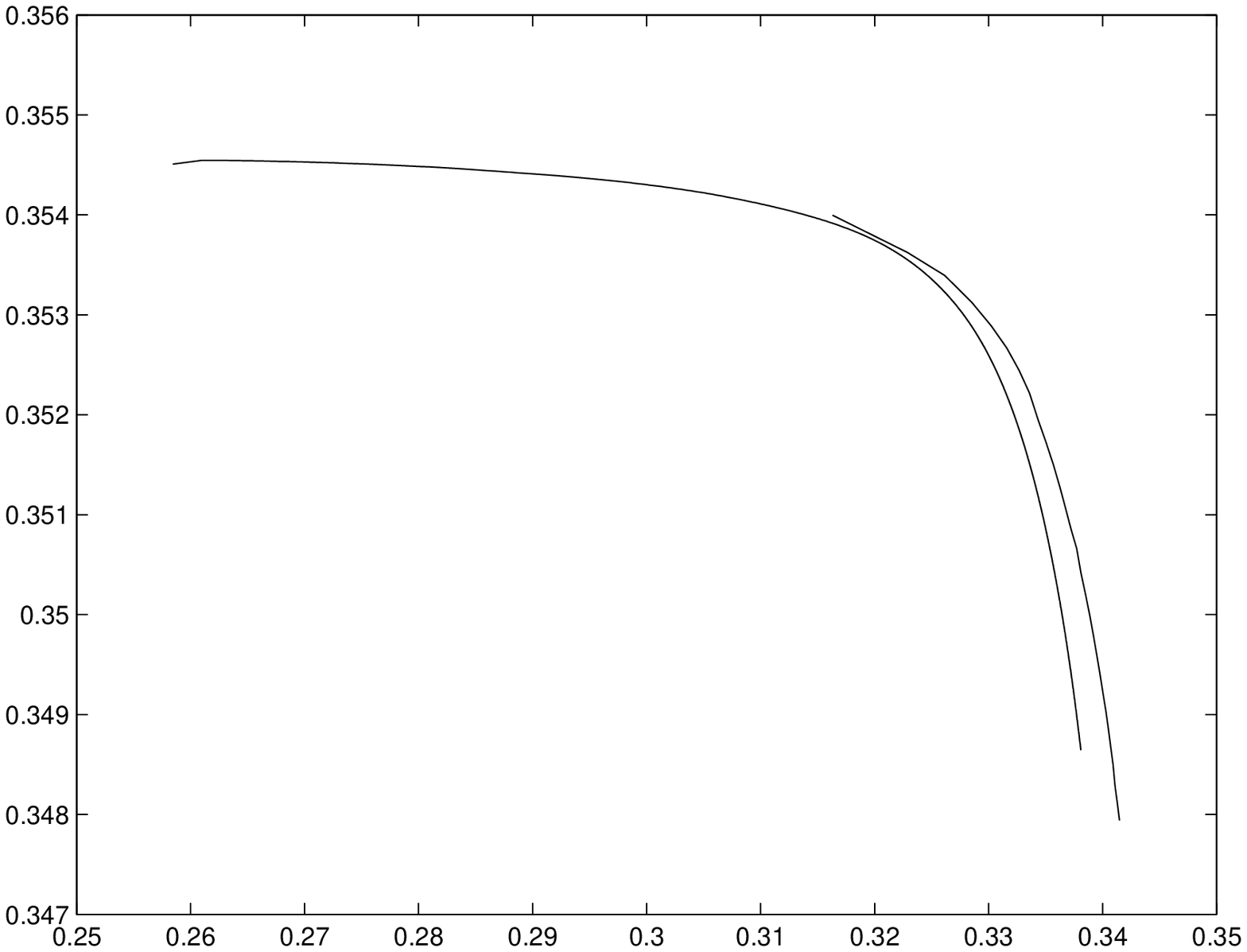, width=2.5in,height=2.5in}
\end{center}
\caption{Ten sample points. Top row: exact algorithm (gradient method).
Middle row: continuation algorithm, 2000 steps. Bottom row: continuation
algorithm, 3000 steps.}
\label{Figure:Sample10}
\end{figure}

\begin{figure}[h]
\begin{center}
\epsfig{figure=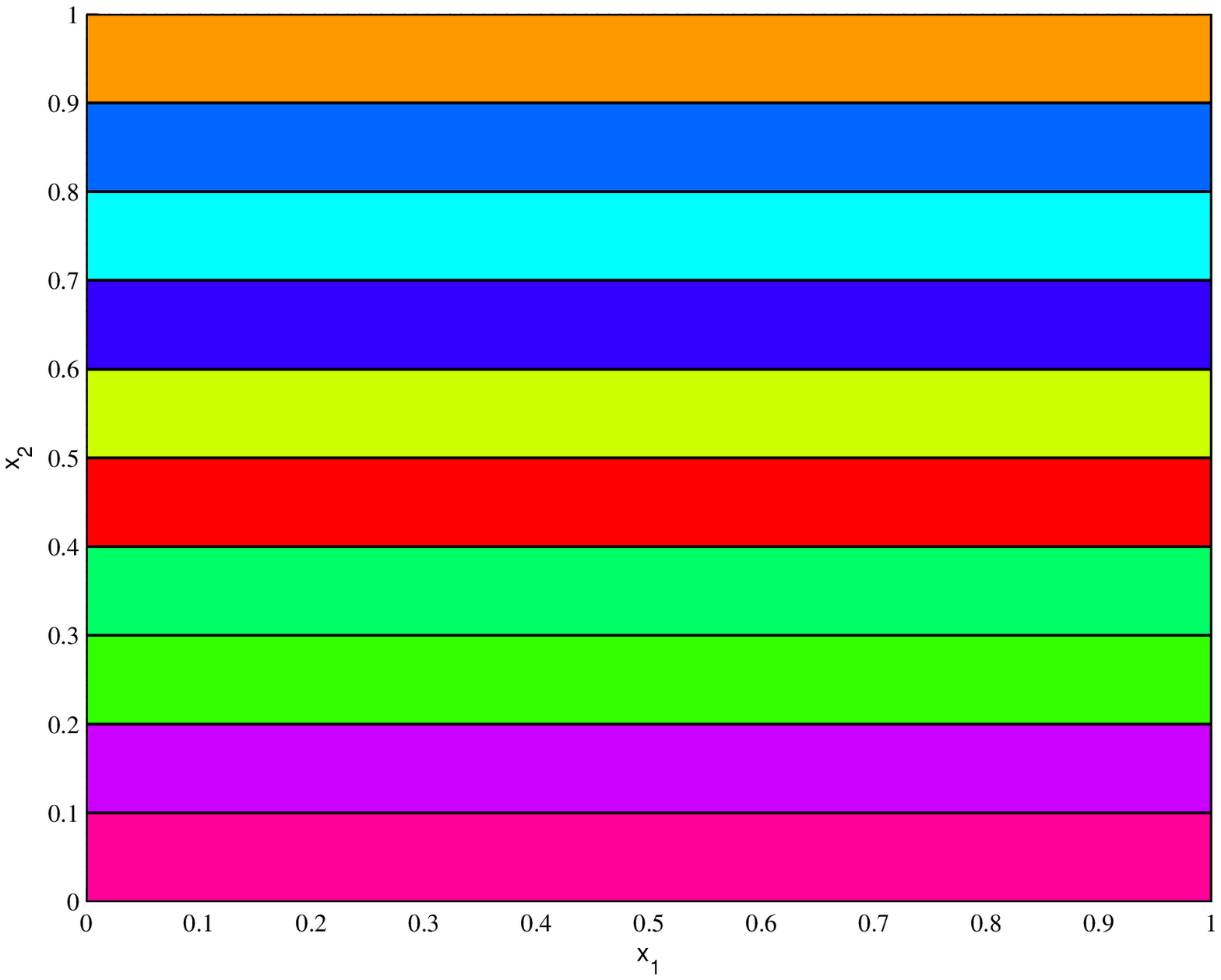, width=1.5in,height=1.5in} %
\epsfig{figure=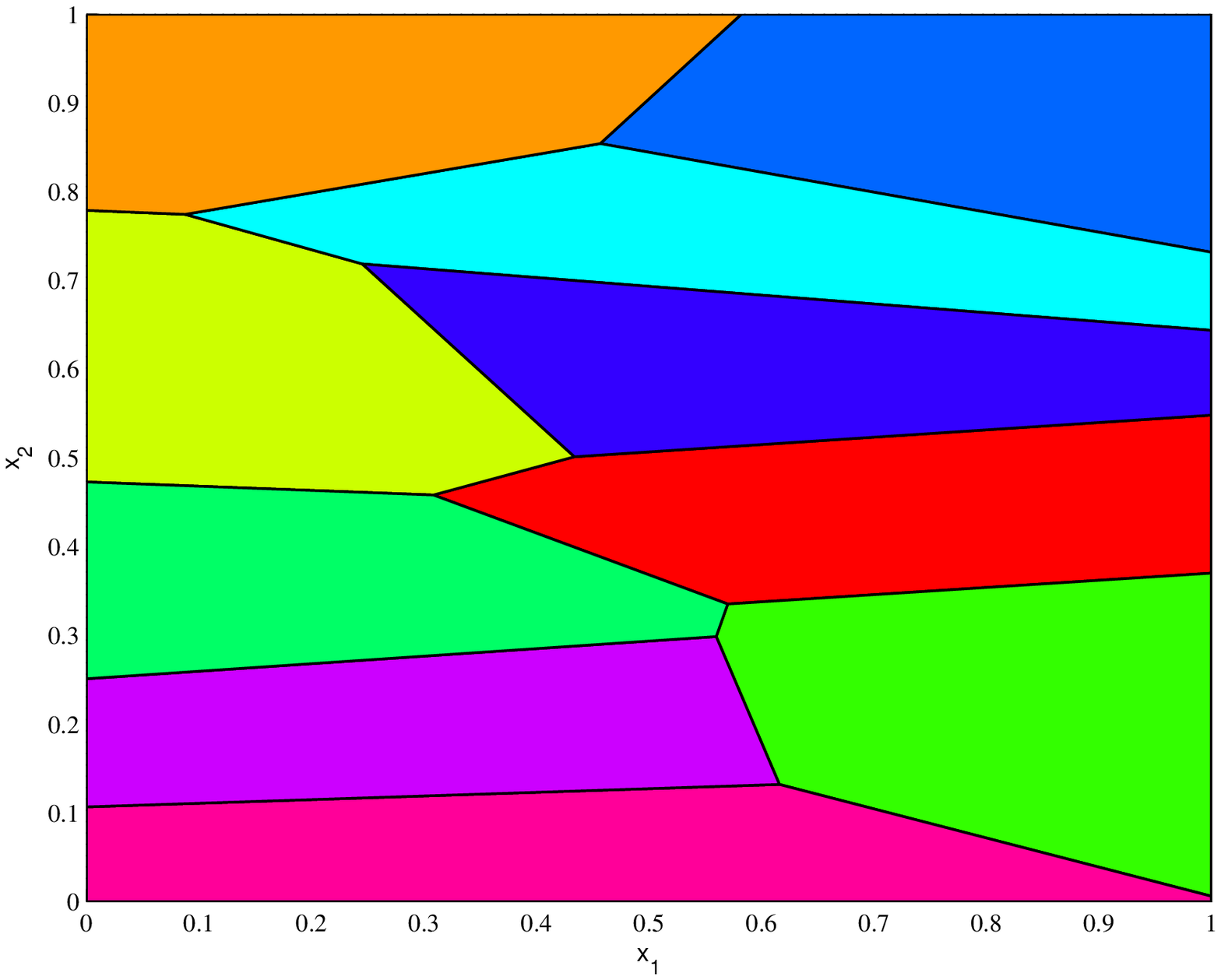, width=1.5in,height=1.5in} %
\epsfig{figure=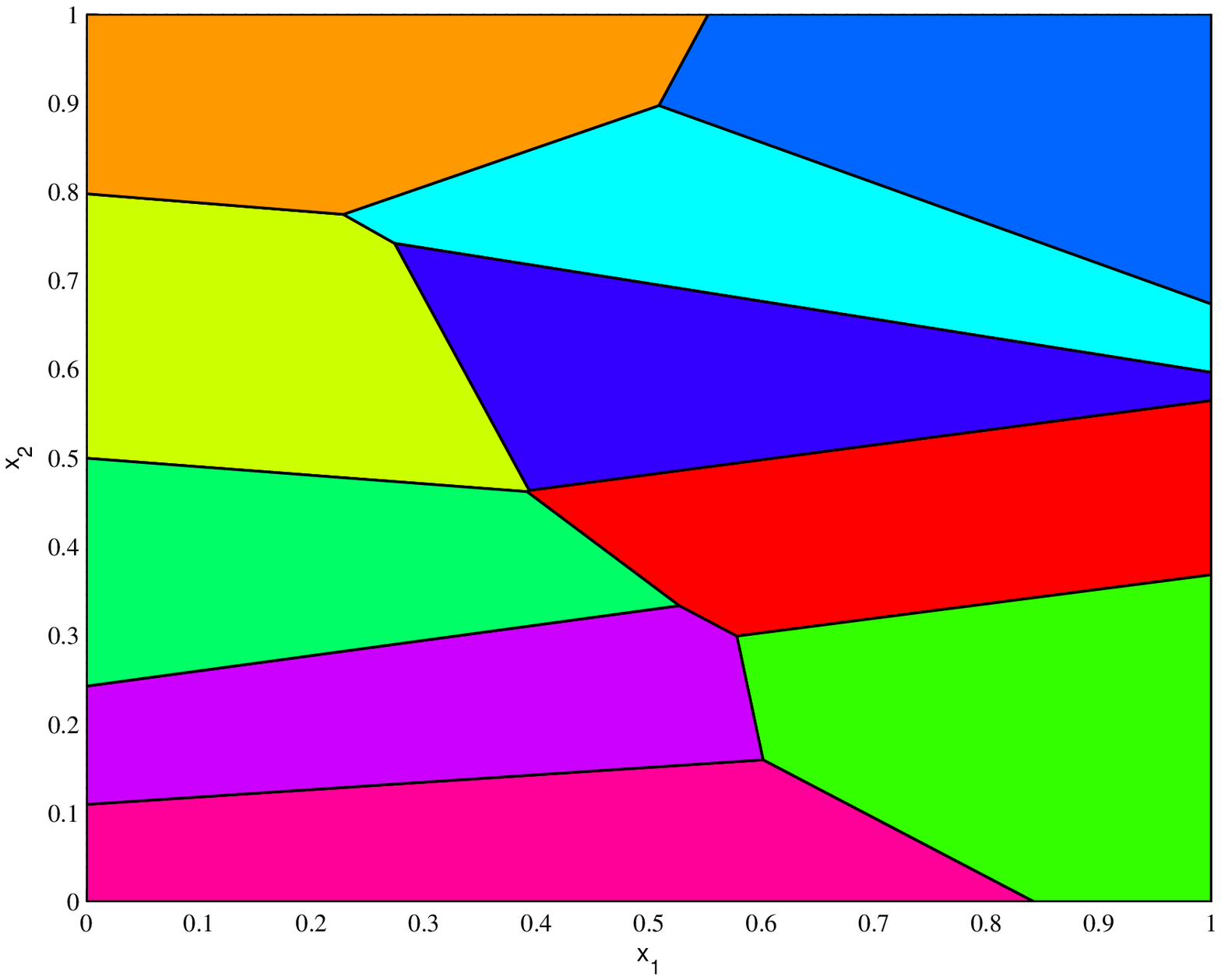,width=1.5in,height=1.5in} %
\epsfig{figure=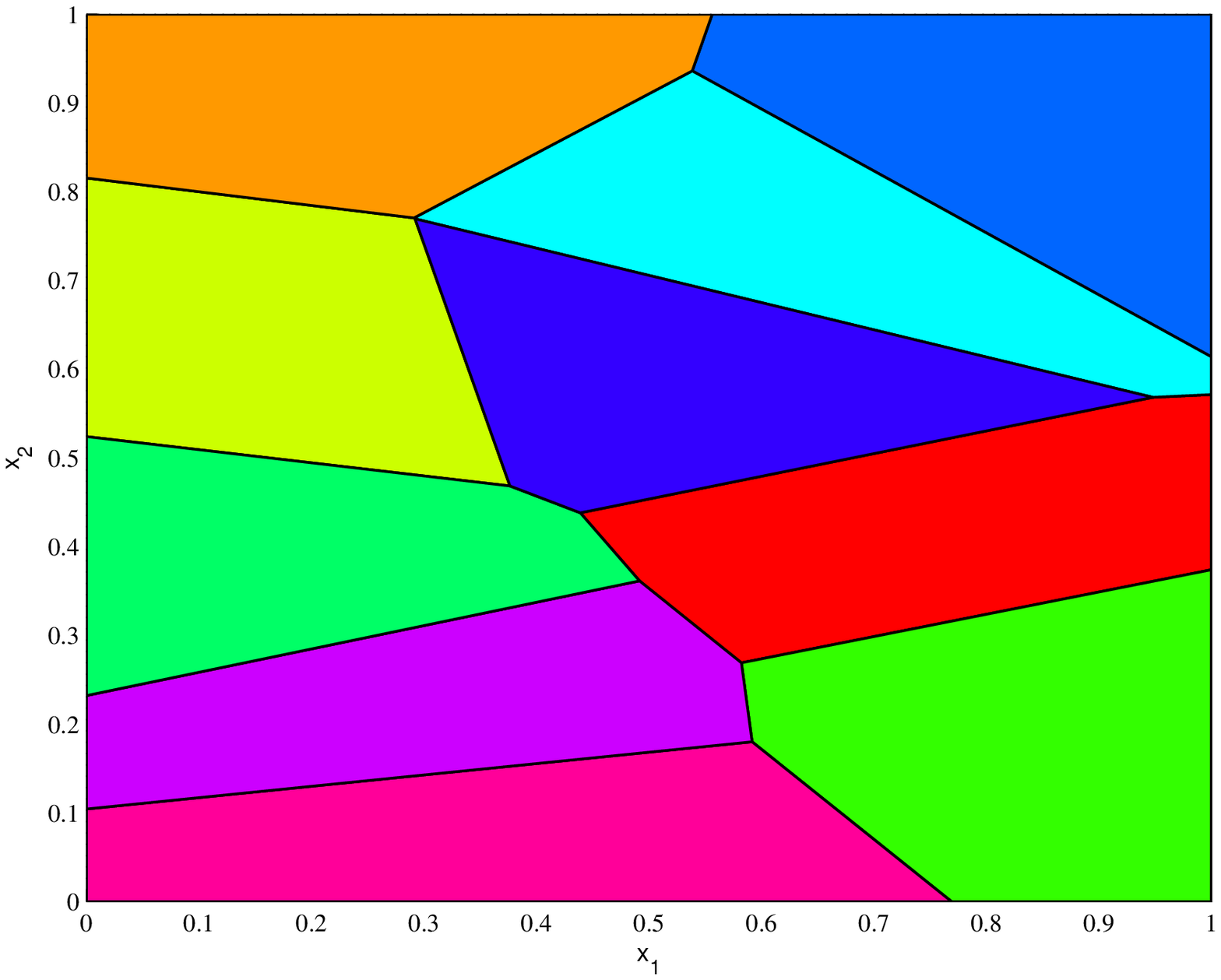,width=1.5in,height=1.5in} %
\epsfig{figure=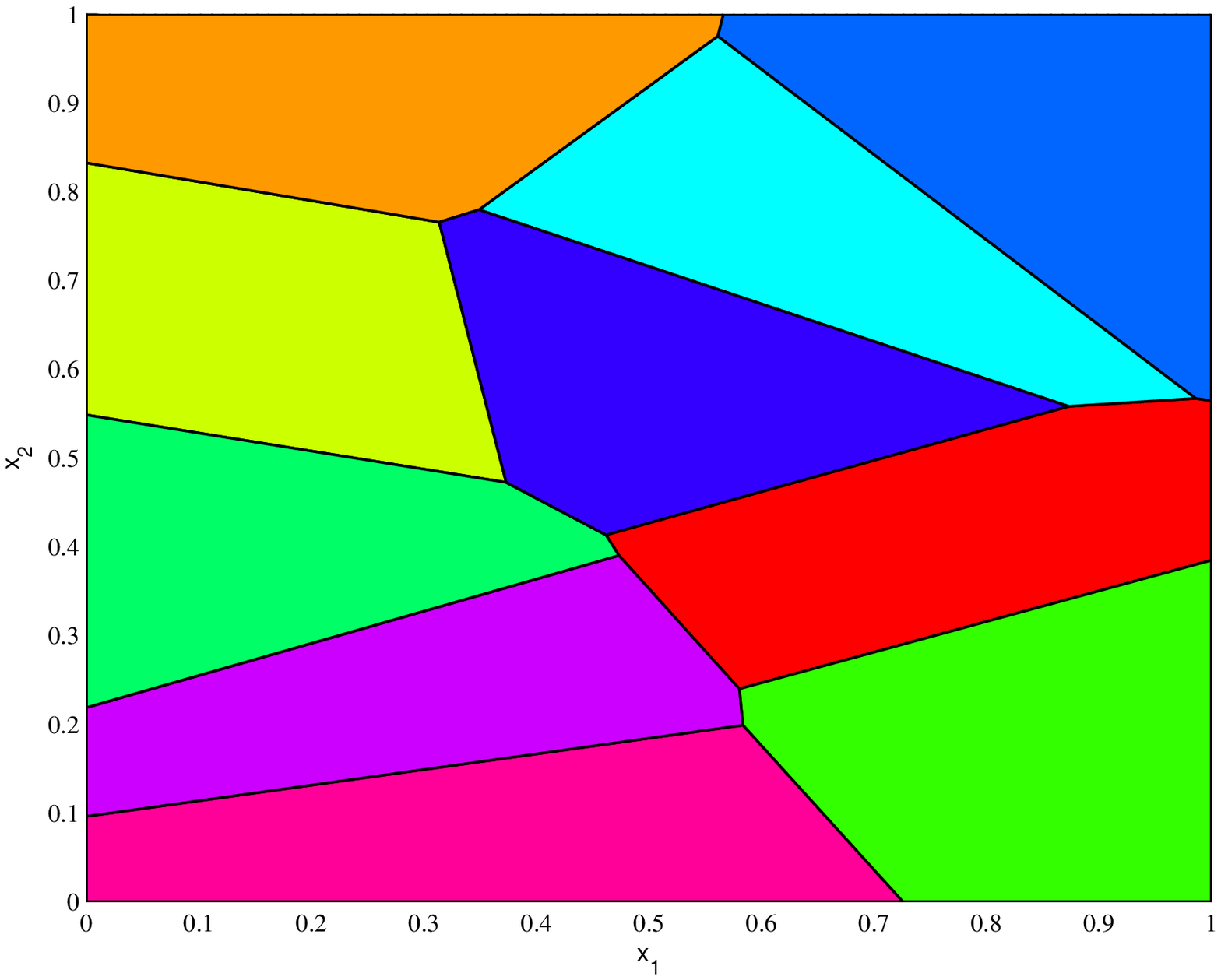,width=1.5in,height=1.5in} %
\epsfig{figure=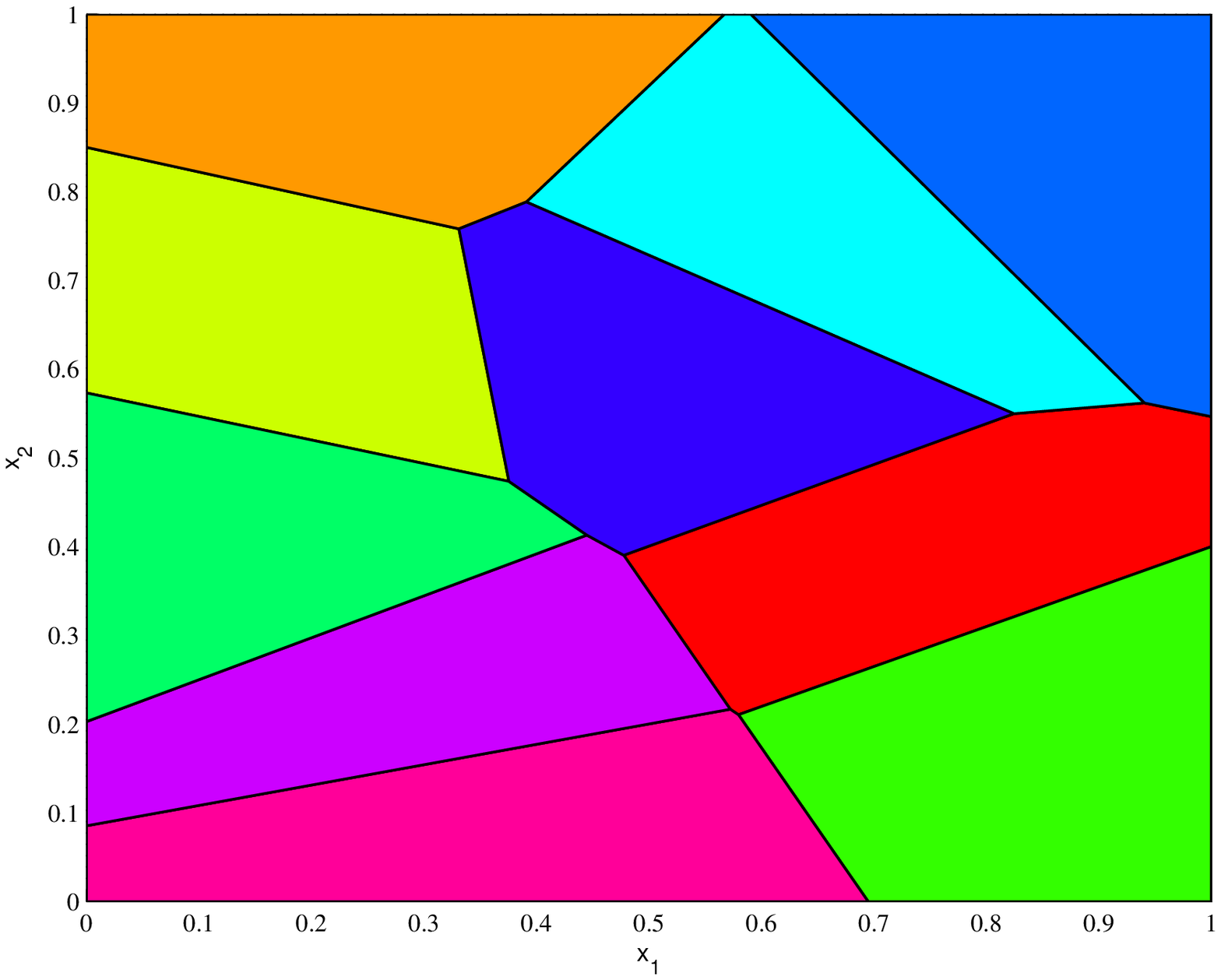,width=1.5in,height=1.5in} %
\epsfig{figure=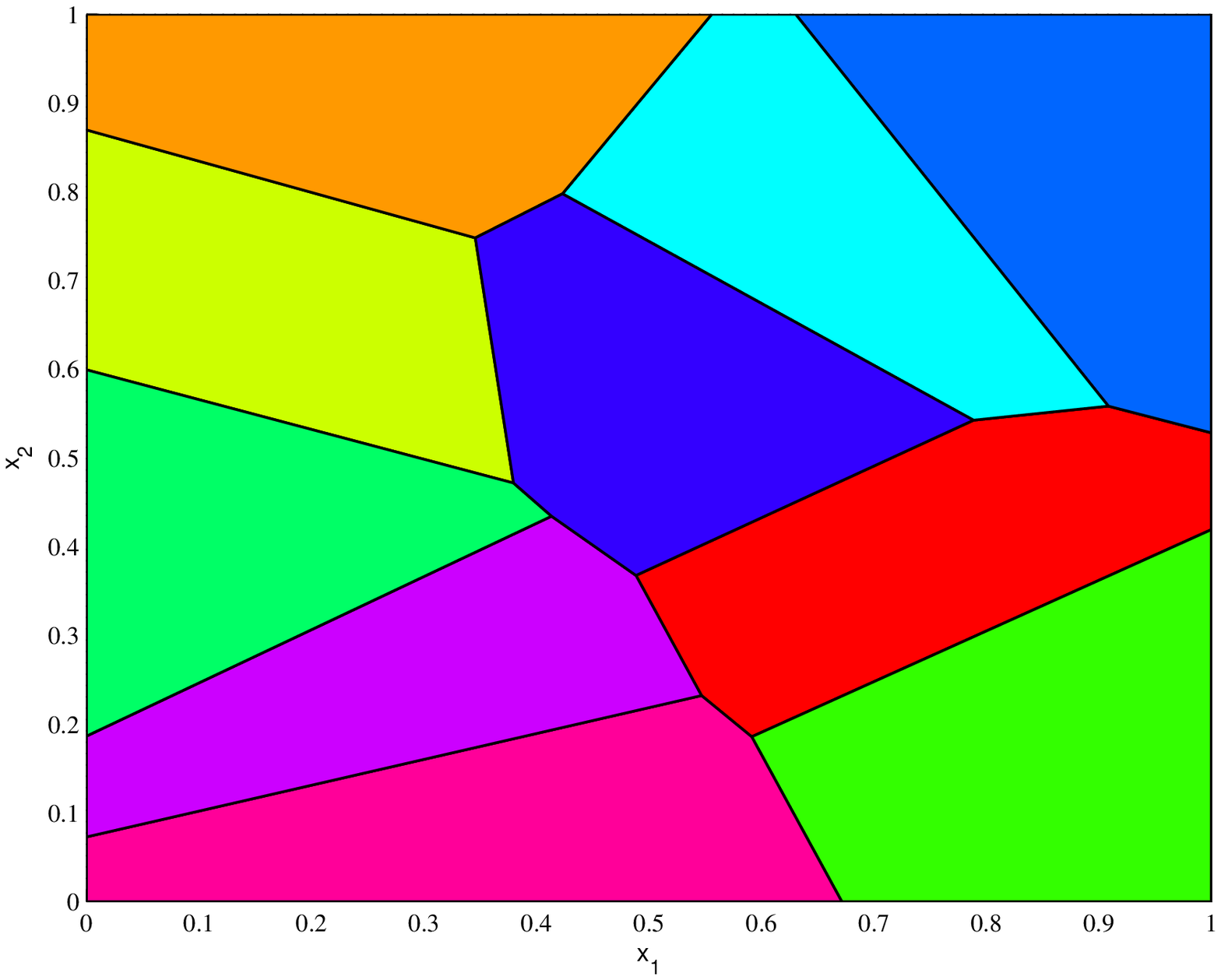,width=1.5in,height=1.5in} %
\epsfig{figure=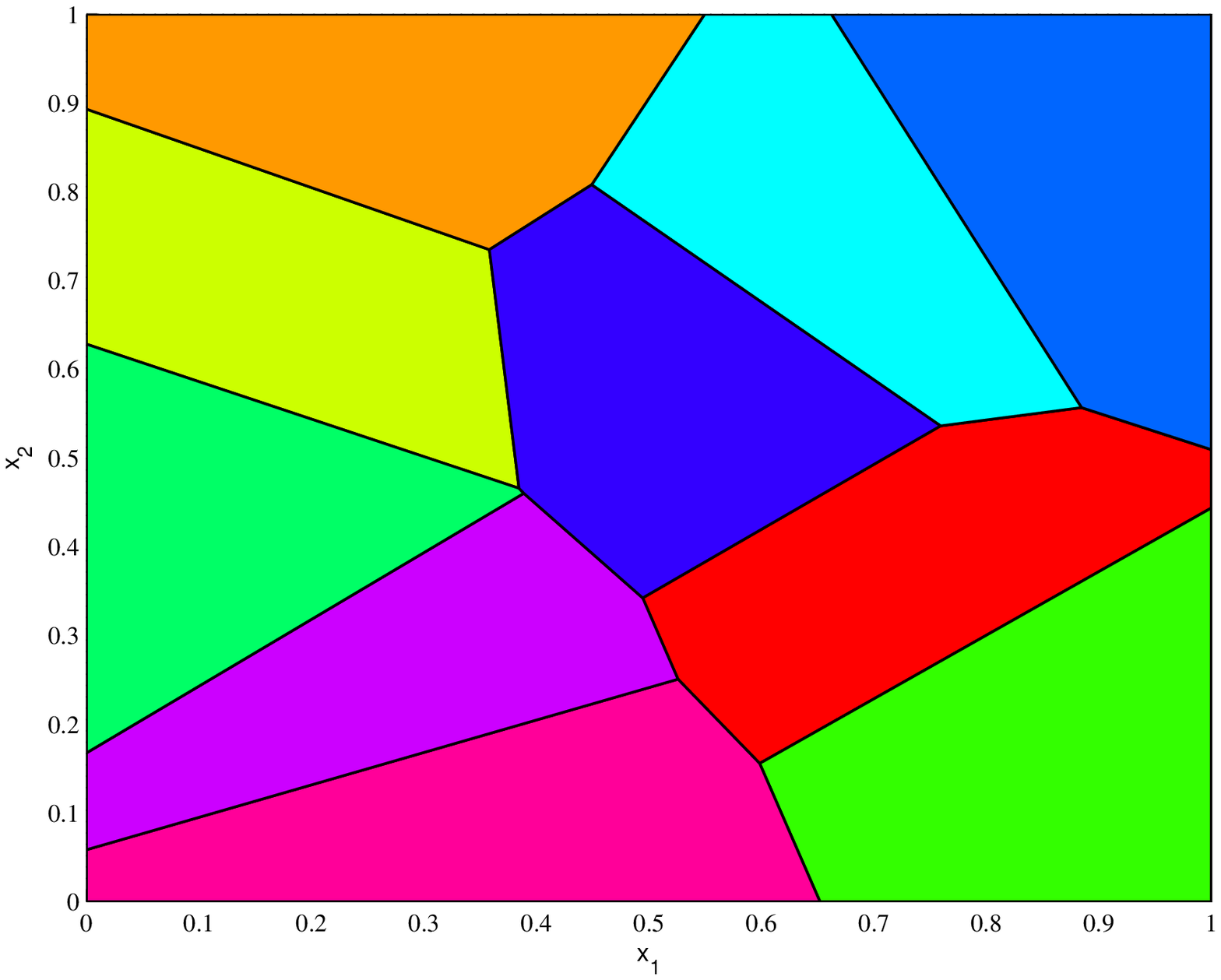,width=1.5in,height=1.5in} %
\epsfig{figure=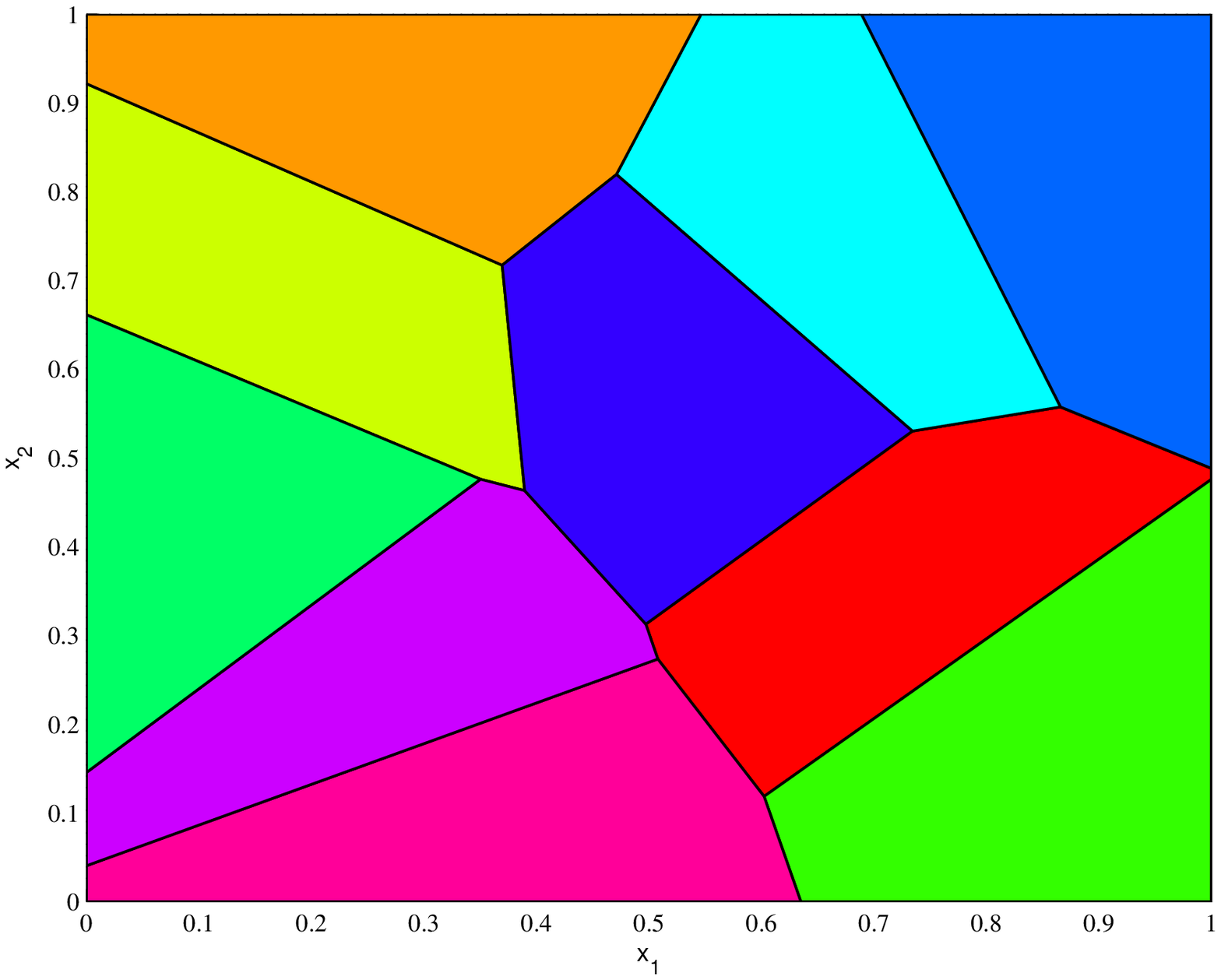,width=1.5in,height=1.5in} %
\epsfig{figure=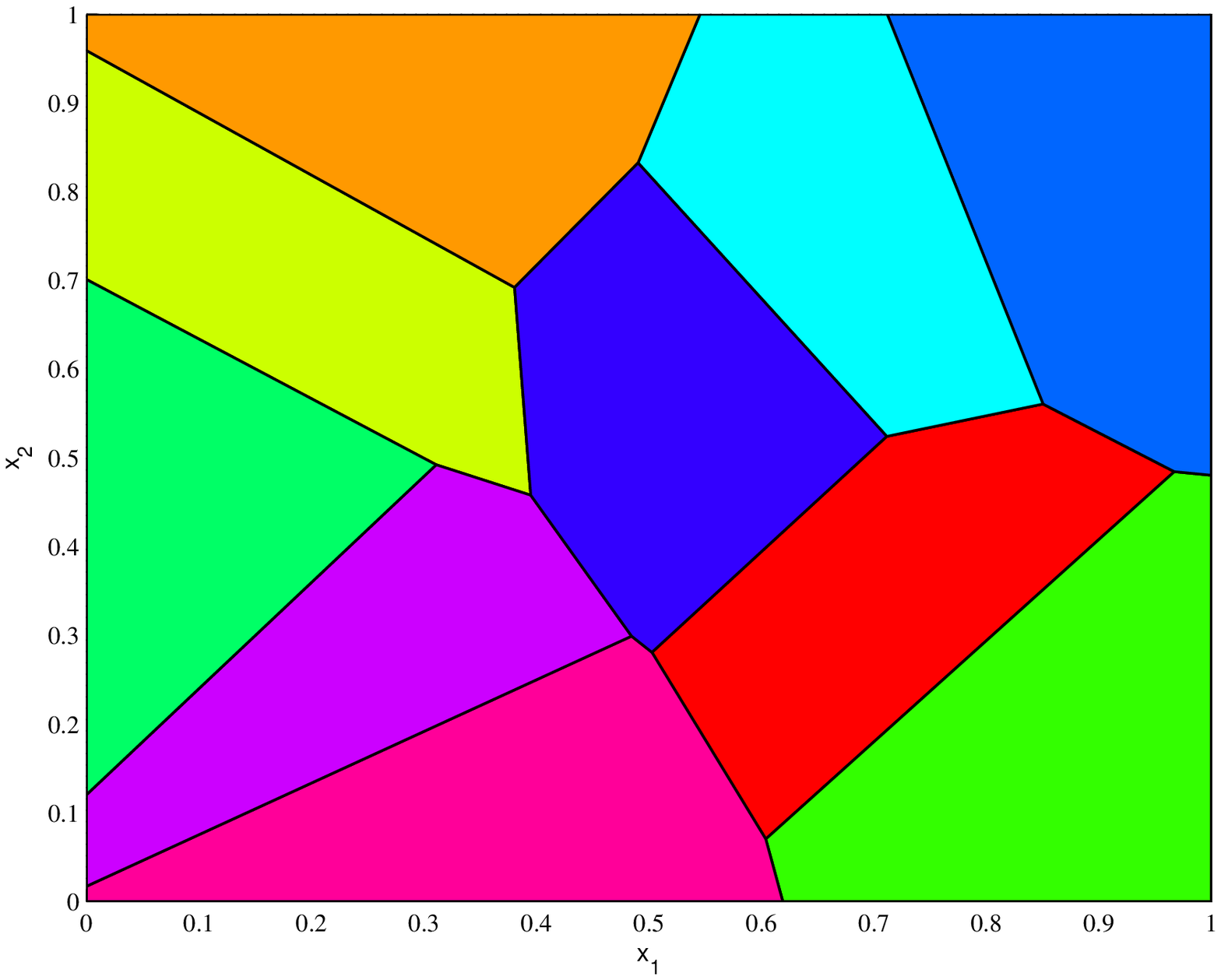,width=1.5in,height=1.5in} %
\epsfig{figure=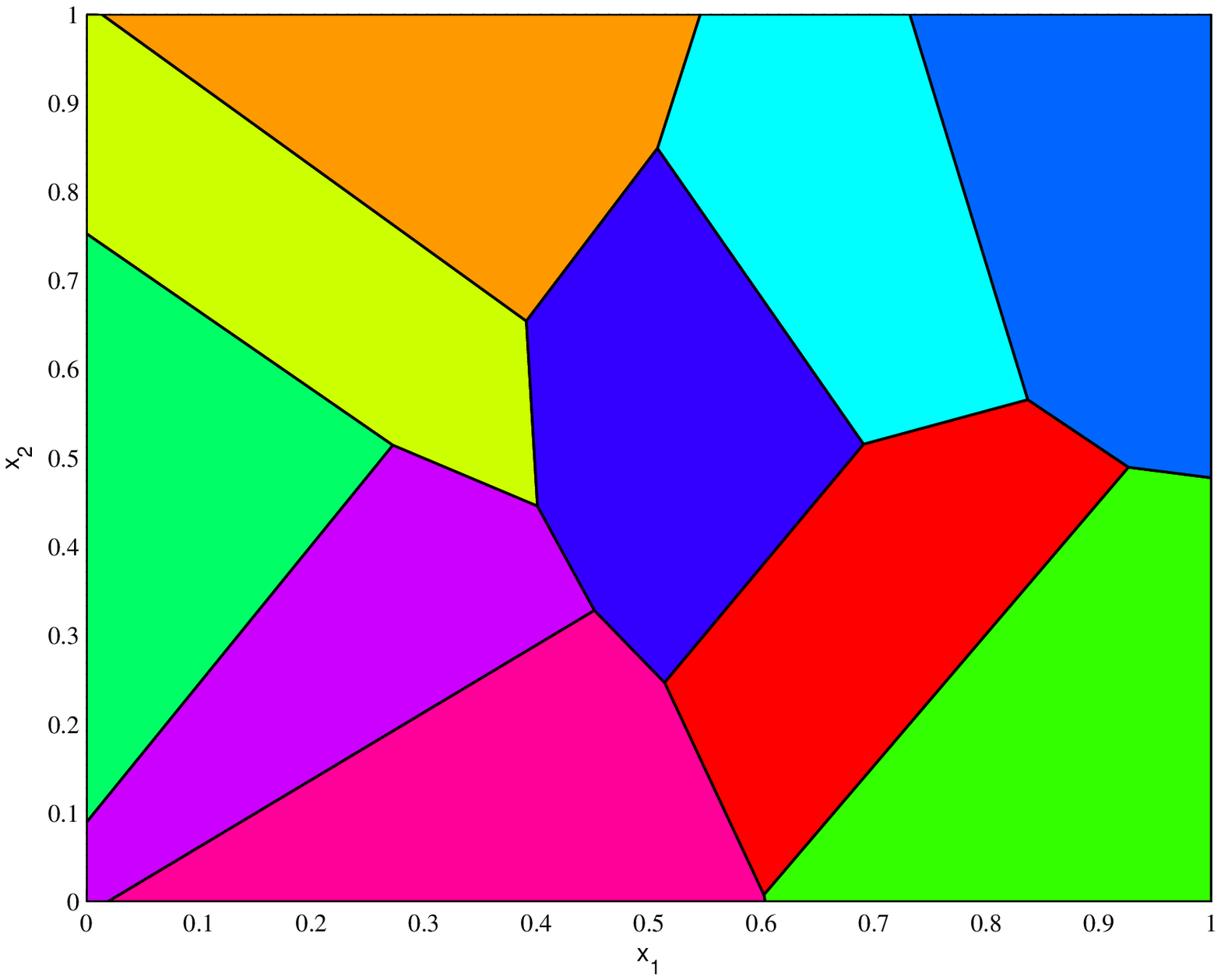,width=1.5in,height=1.5in} %
\epsfig{figure=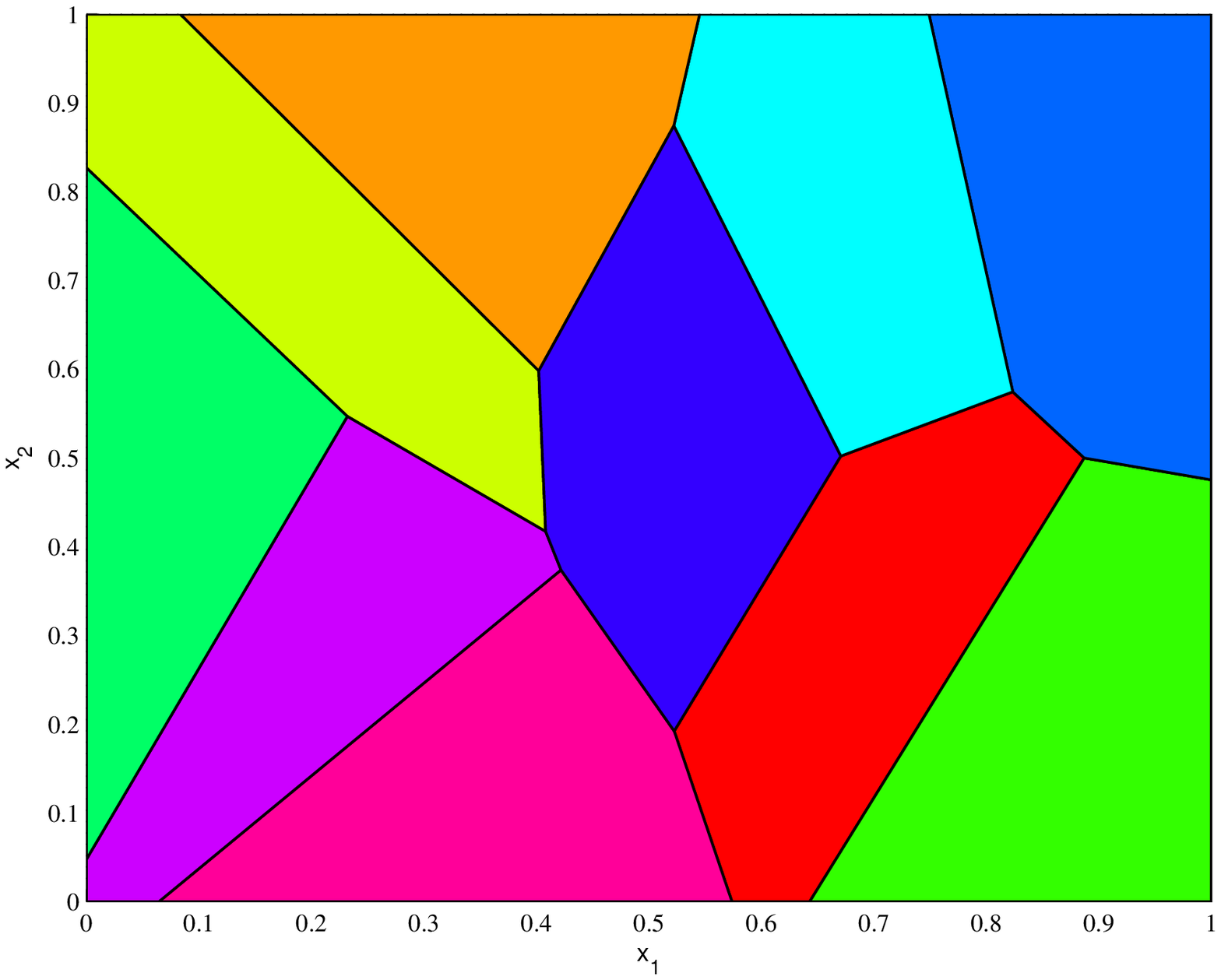,width=1.5in,height=1.5in} %
\epsfig{figure=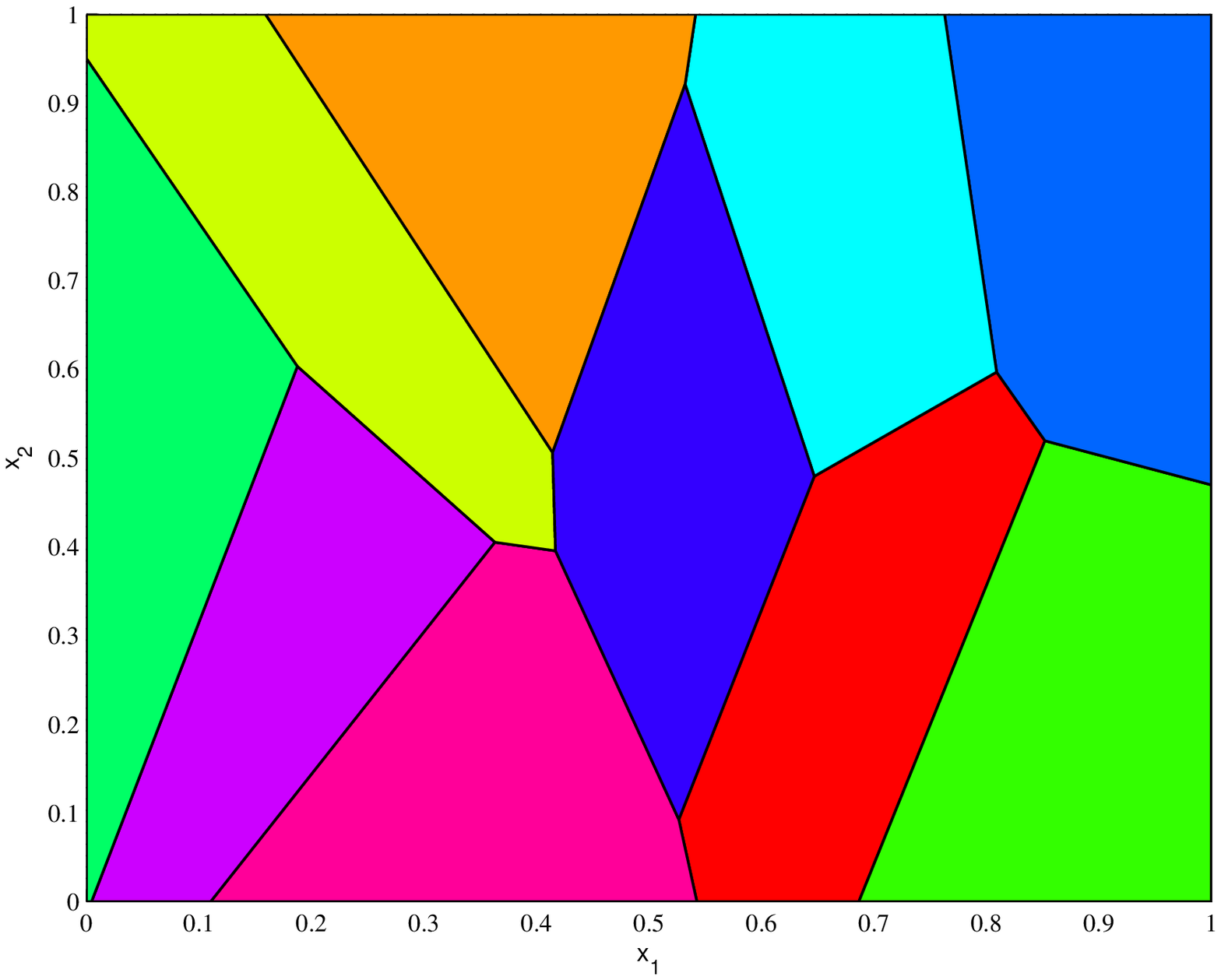,width=1.5in,height=1.5in} %
\epsfig{figure=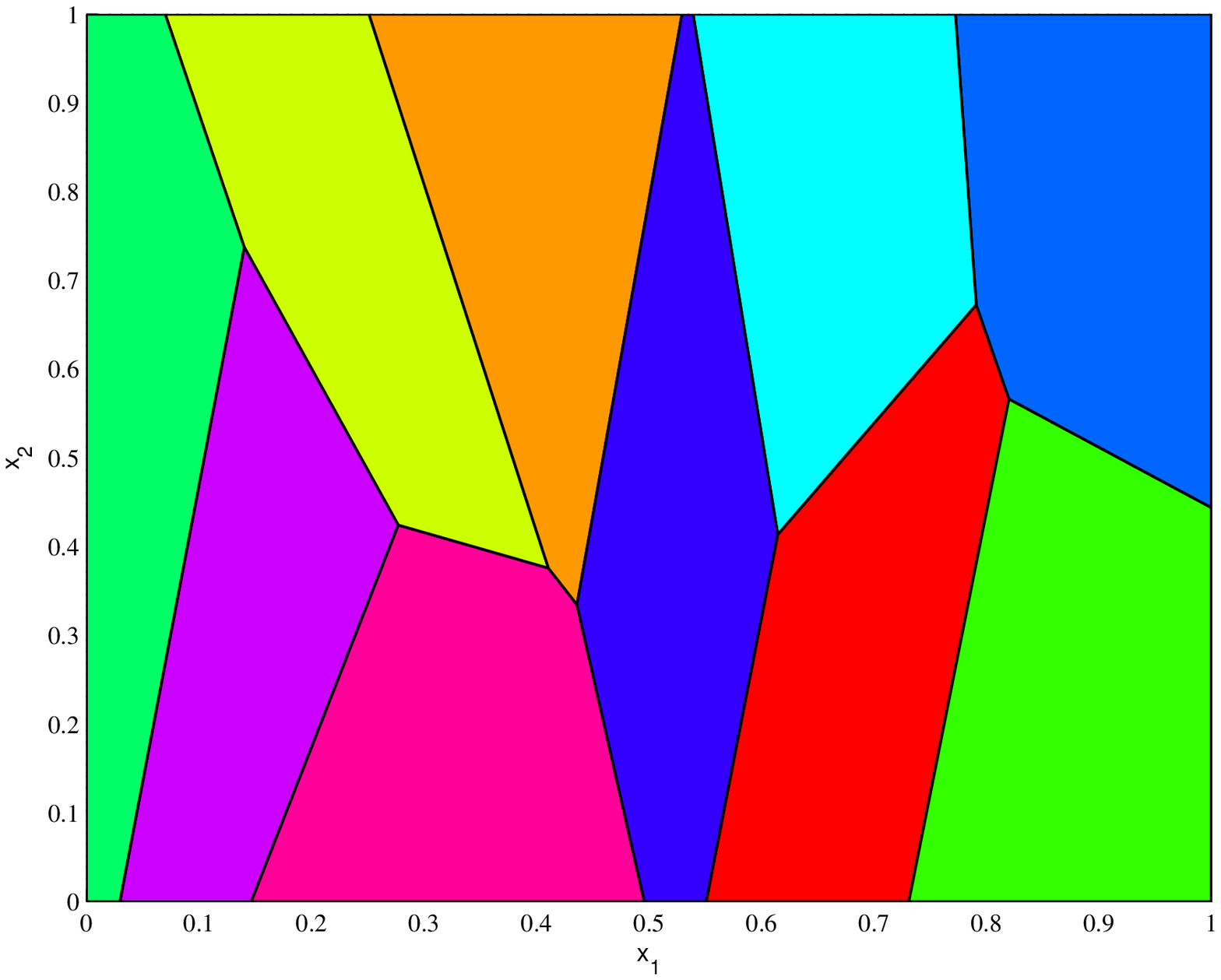,width=1.5in,height=1.5in} %
\epsfig{figure=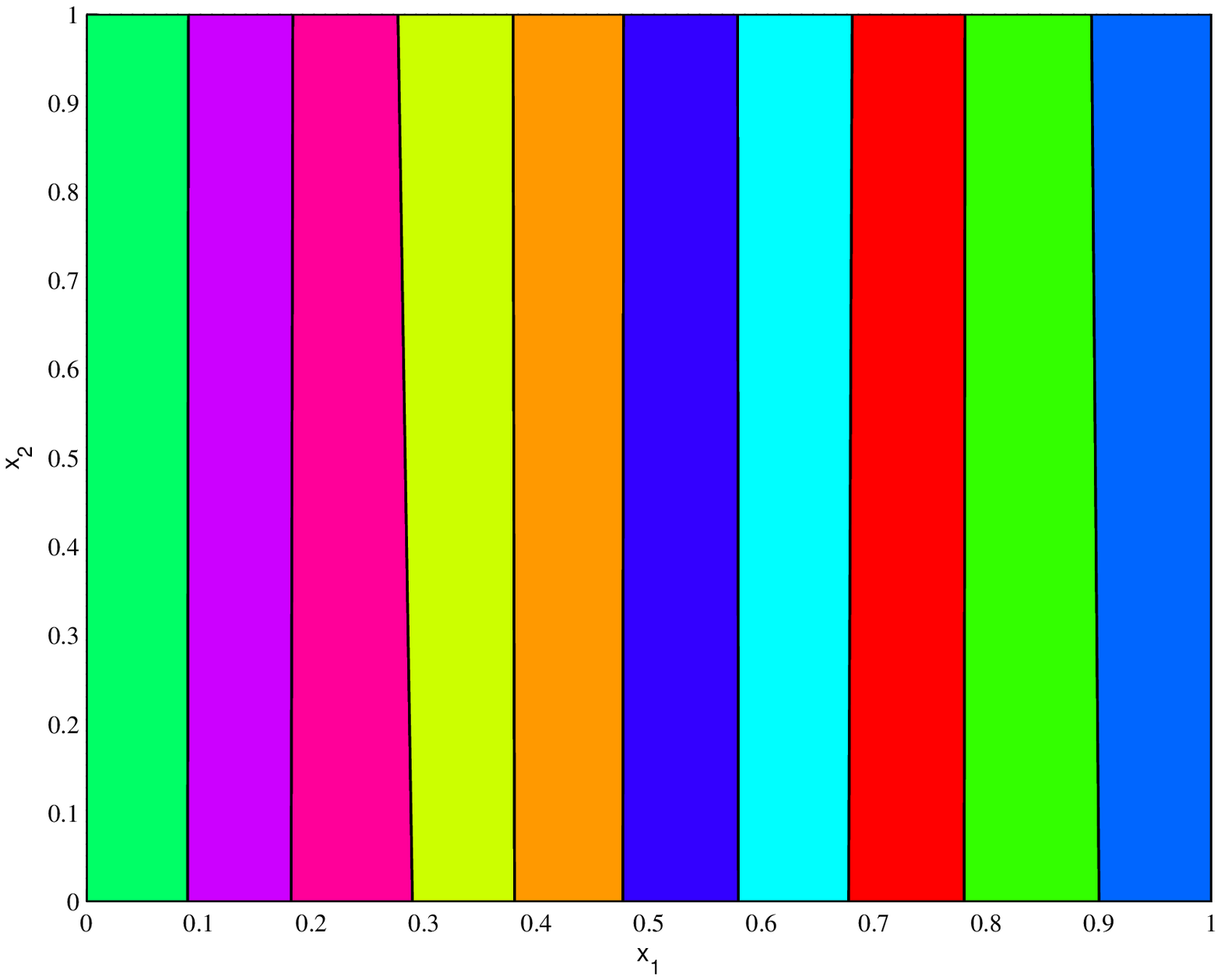,width=1.5in,height=1.5in}
\end{center}
\caption{Ten sample points: evolution of the tesselation for $\protect%
\varepsilon=0$ to $\protect\varepsilon=+\infty$ (from top left to bottom
right).}
\label{Figure:Sample10evol}
\end{figure}

\textbf{Fifteen sample points.} With a sample of 15
points, we get the following errors:
$$
\begin{tabular}{||l|l|l|l|l|l|l||}
\hline\hline
\textbf{\# steps} & \multicolumn{5}{|l|}{\textbf{Relative errors in cell
areas}} & \textbf{Time} \\ \hline
3,000 & -2.24\% & -0.12\% & -0.58\% & 0.20\% & -3.20\% & $\simeq$ 2 h \\ \hline
10,000 & -0.67\% & -0.04\% & -0.17\% & 0.06\% & -0.97\% & $\leq$ 4 h \\ \hline
& \multicolumn{6}{|l||}{\textbf{(continued)}} \\ \hline
3,000 & 0.17\% & -2.66\% & 4.25\% & -2.43\% & 28.85\% & $\simeq$ 2 h \\ \hline
10,000 & 0.05\% & -0.82\% & 1.28\% & -0.74\% & 9.42\% & $\leq$ 4 h \\ \hline
& \multicolumn{6}{|l||}{\textbf{(continued)}} \\ \hline
3,000 & -30.11\% & 2.51\% & 1.82\% & 2.92\% & 0.63\% & $\simeq$ 2 h \\ \hline
10,000 & -9.79\% & 0.76\% & 0.55\% & 0.89\% & 0.19\% & $\leq$ 4 h \\
\hline\hline
\end{tabular}
$$

\begin{figure}[h]
\begin{center}
\epsfig{figure=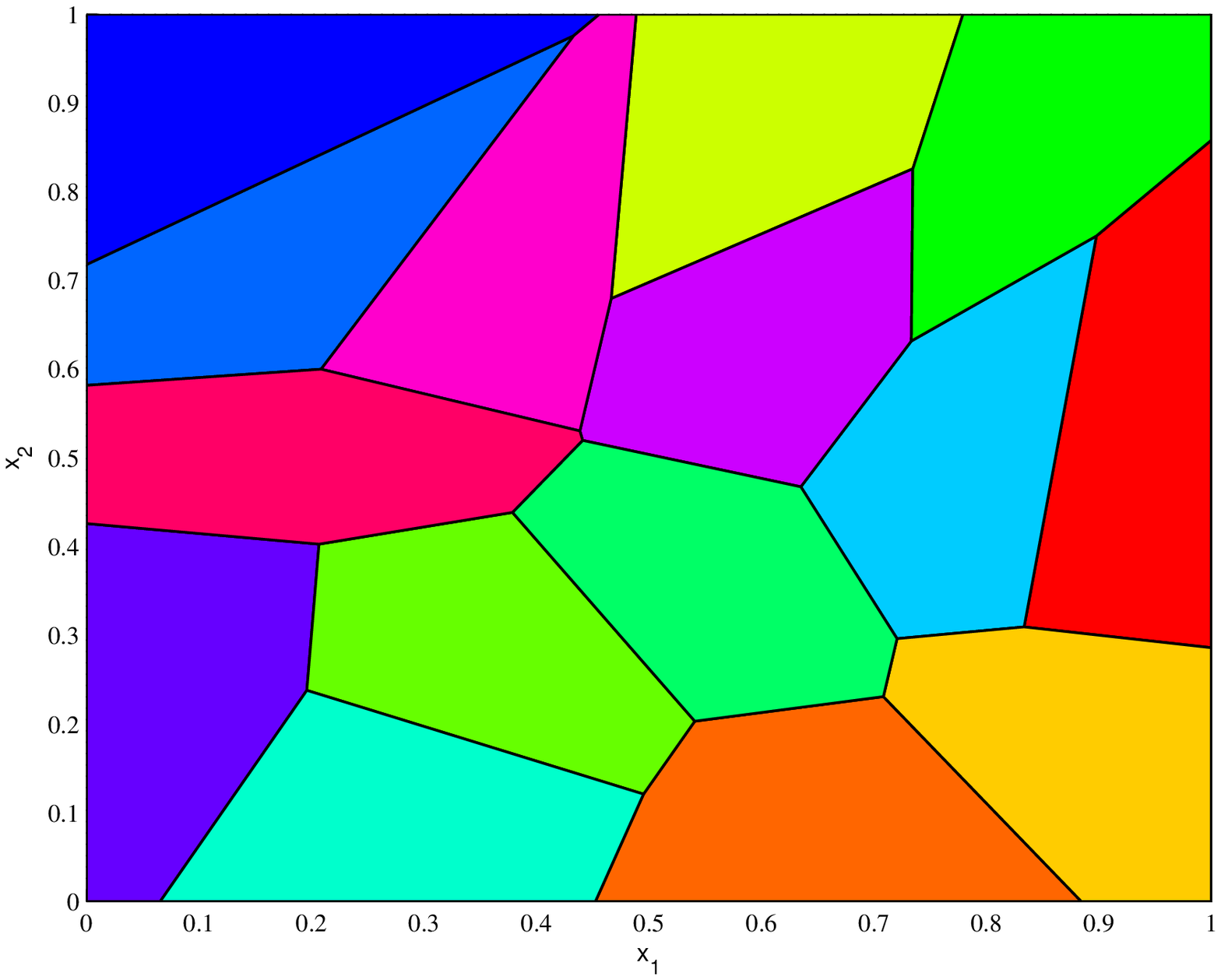, width=2.5in,height=2.5in} %
\epsfig{figure=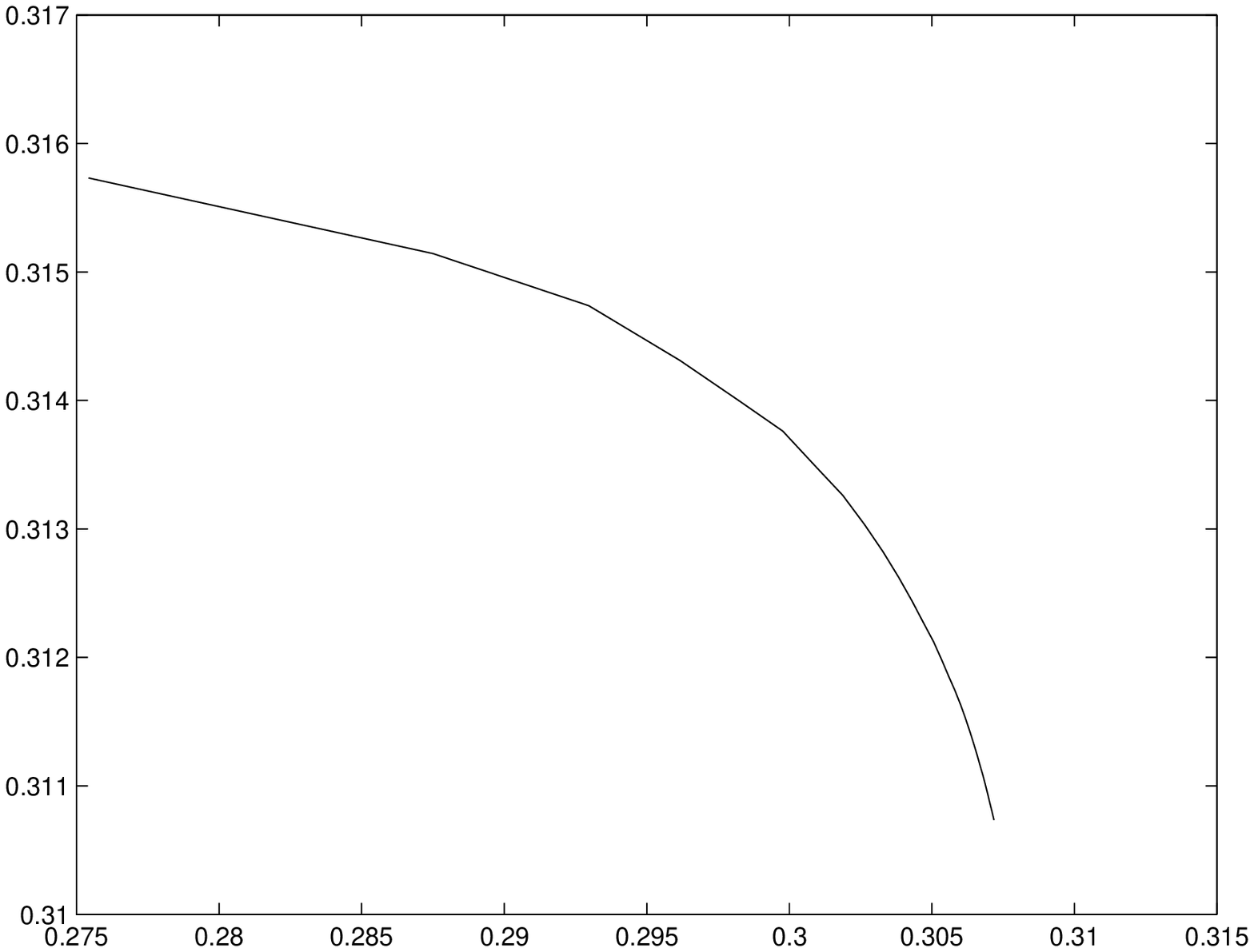,width=2.5in,height=2.5in} %
\epsfig{figure=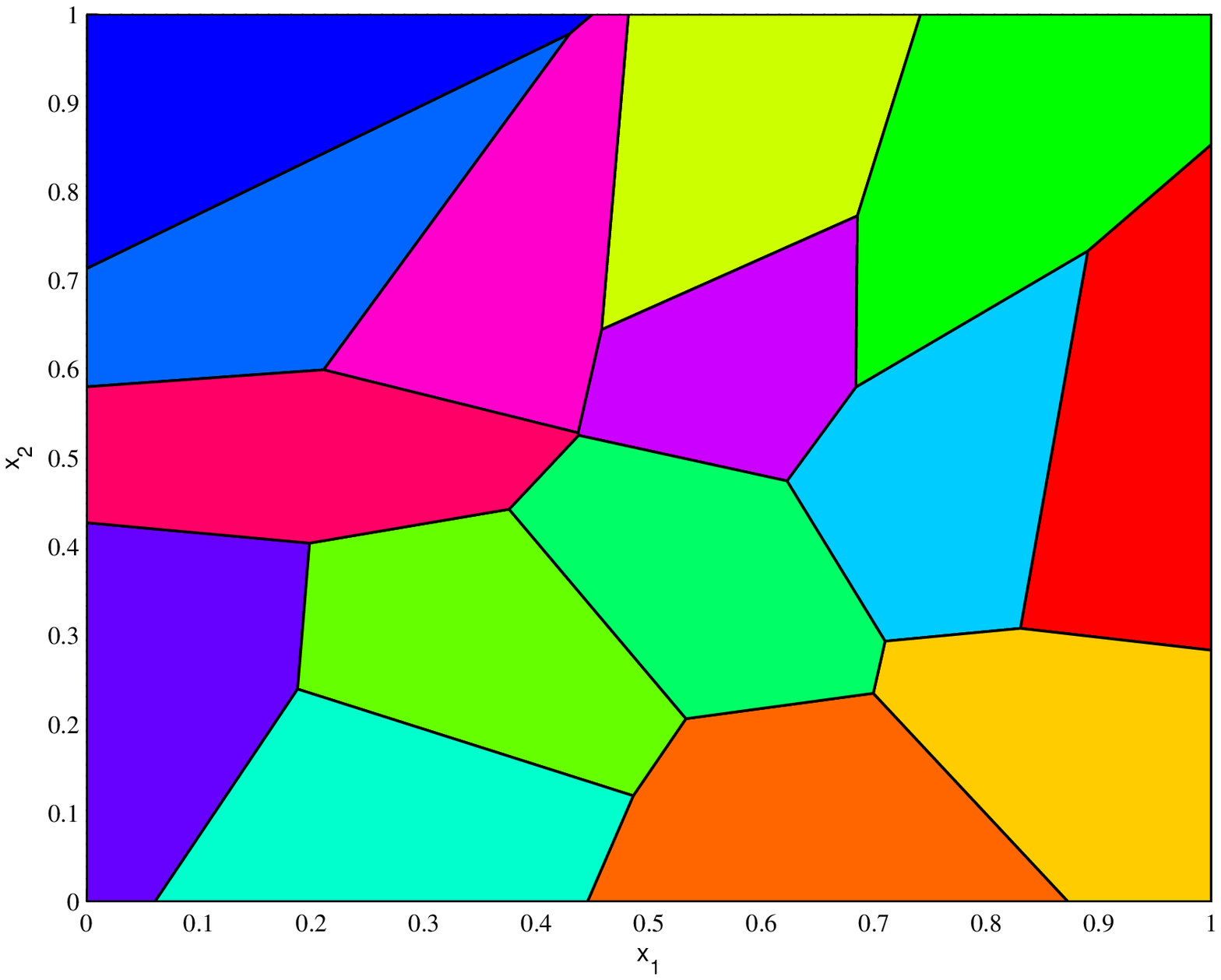,width=2.5in,height=2.5in} %
\epsfig{figure=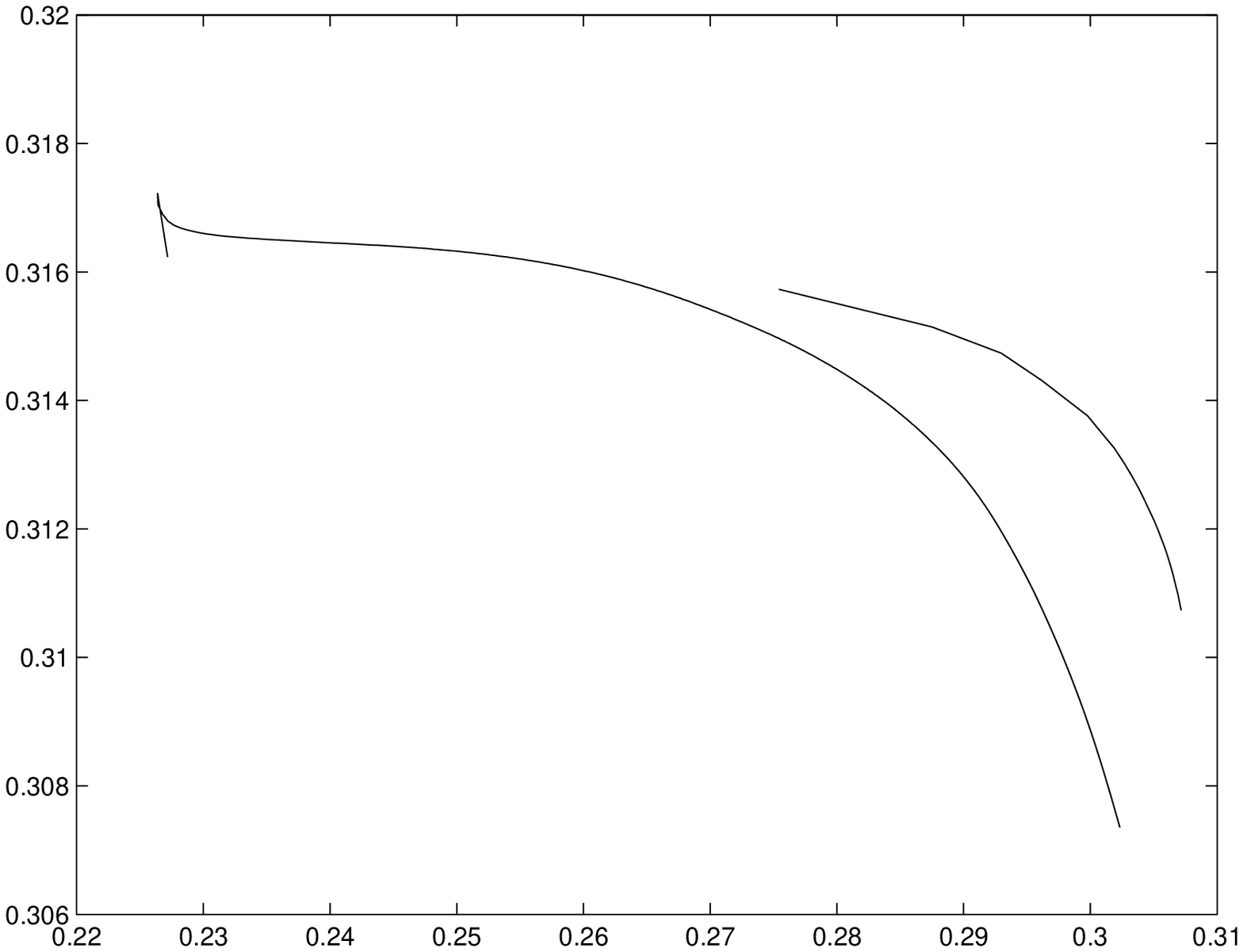, width=2.5in,height=2.5in} %
\epsfig{figure=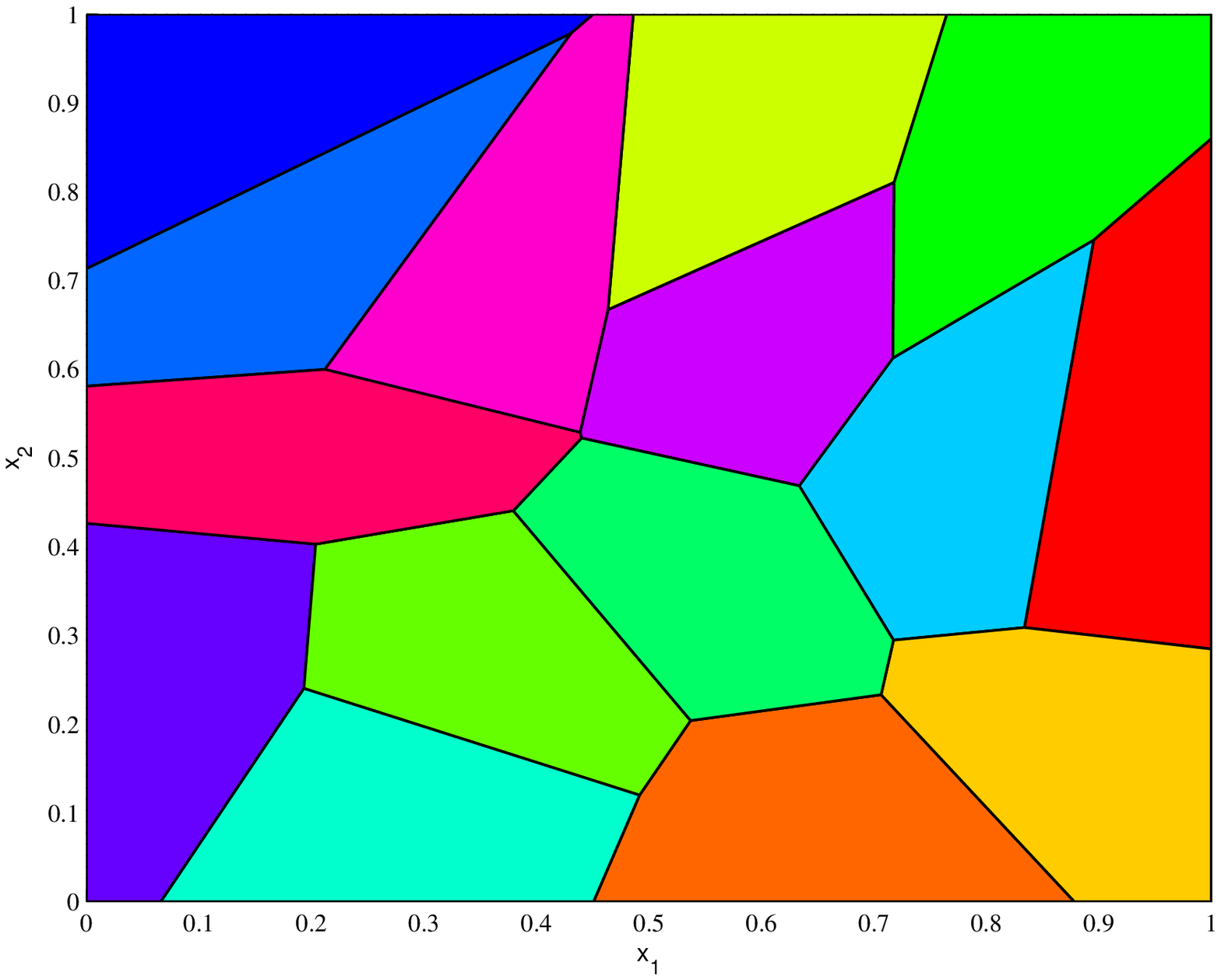,width=2.5in,height=2.5in} %
\epsfig{figure=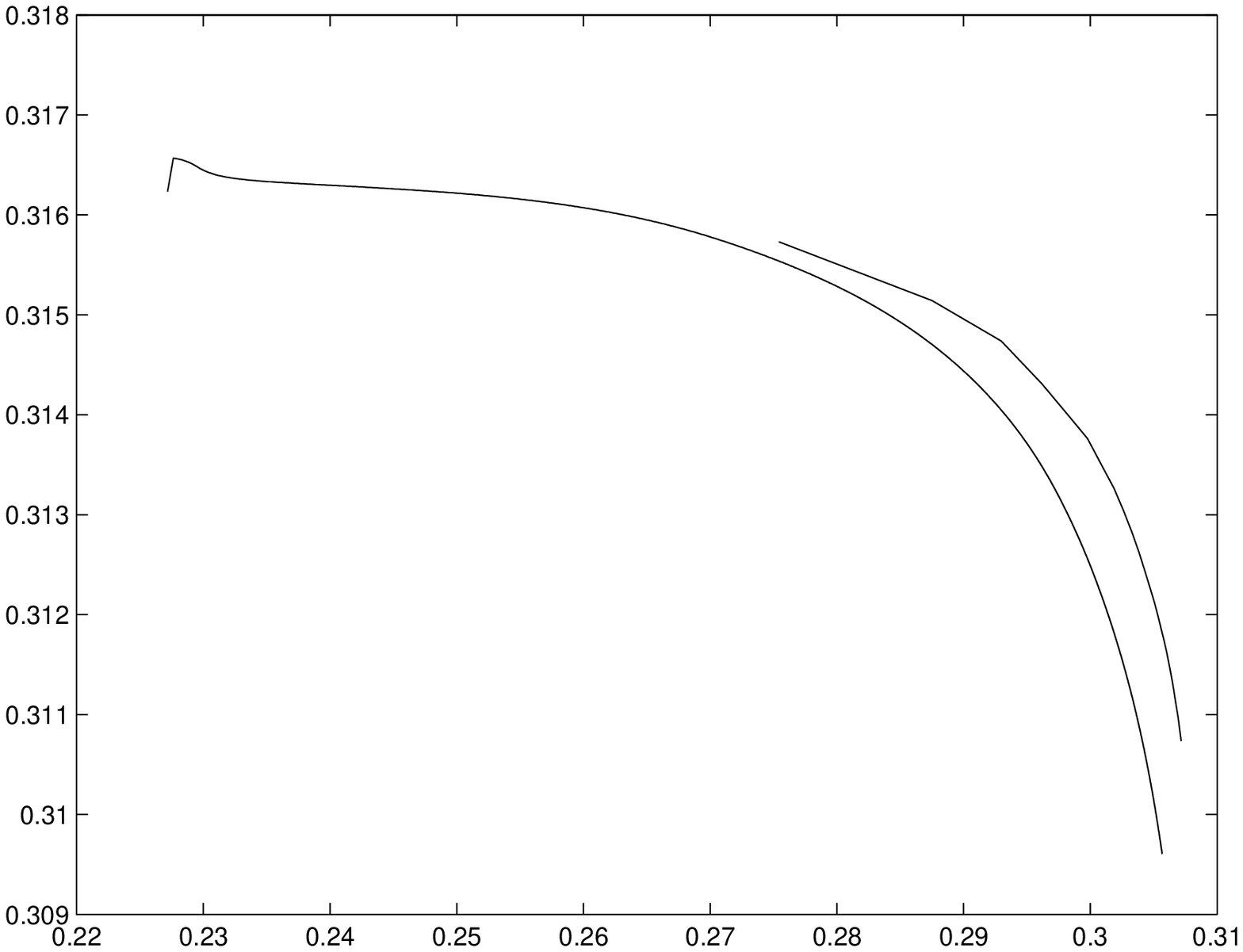, width=2.5in,height=2.5in}
\end{center}
\caption{Fifteen sample points. Top row: exact algorithm (gradient method).
Middle row: continuation algorithm, 3000 steps. Bottom row: continuation
algorithm, 10,000 steps.}
\label{Figure:Sample15}
\end{figure}

\begin{figure}[h]
\begin{center}
\epsfig{figure=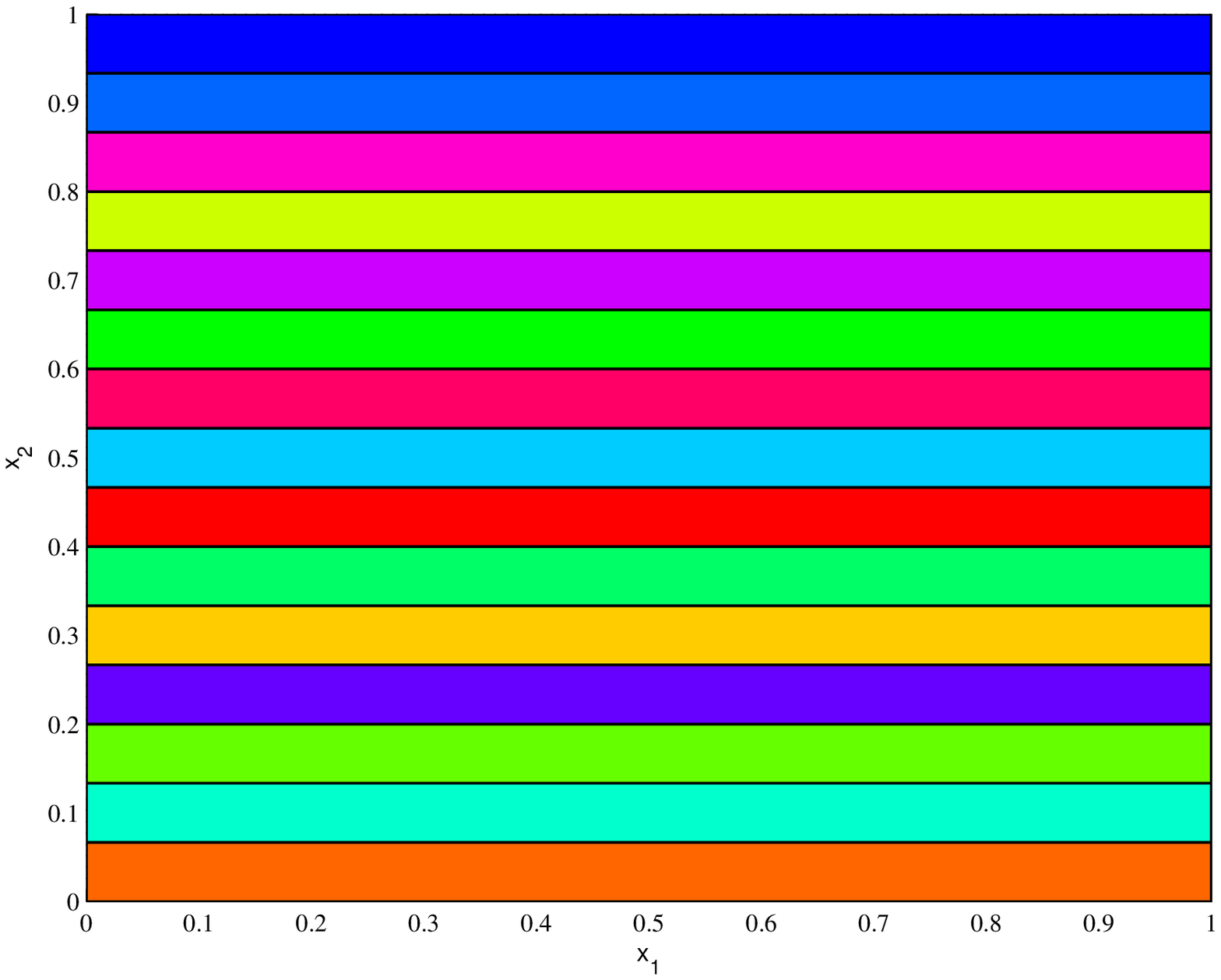, width=1.5in,height=1.5in} %
\epsfig{figure=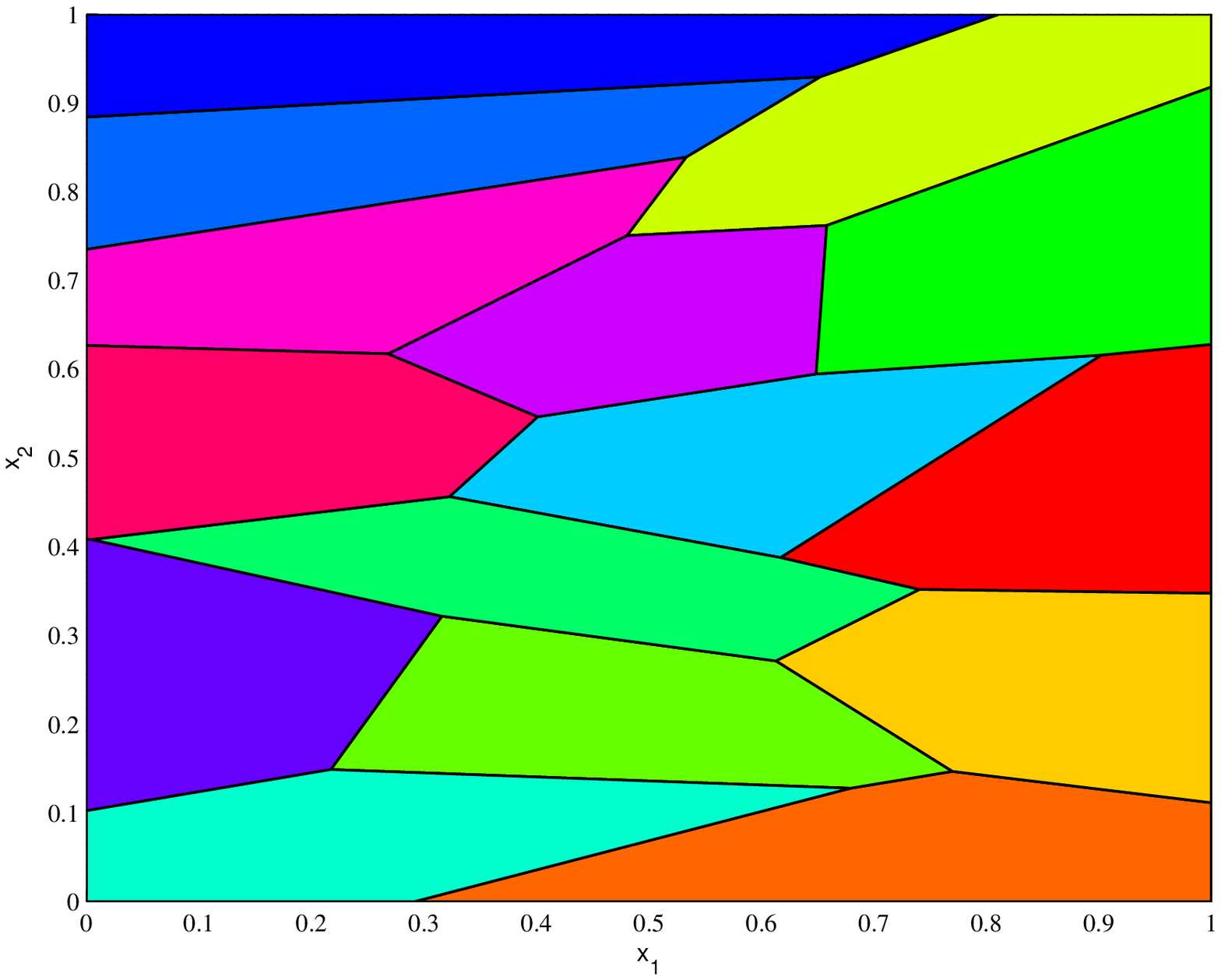, width=1.5in,height=1.5in} %
\epsfig{figure=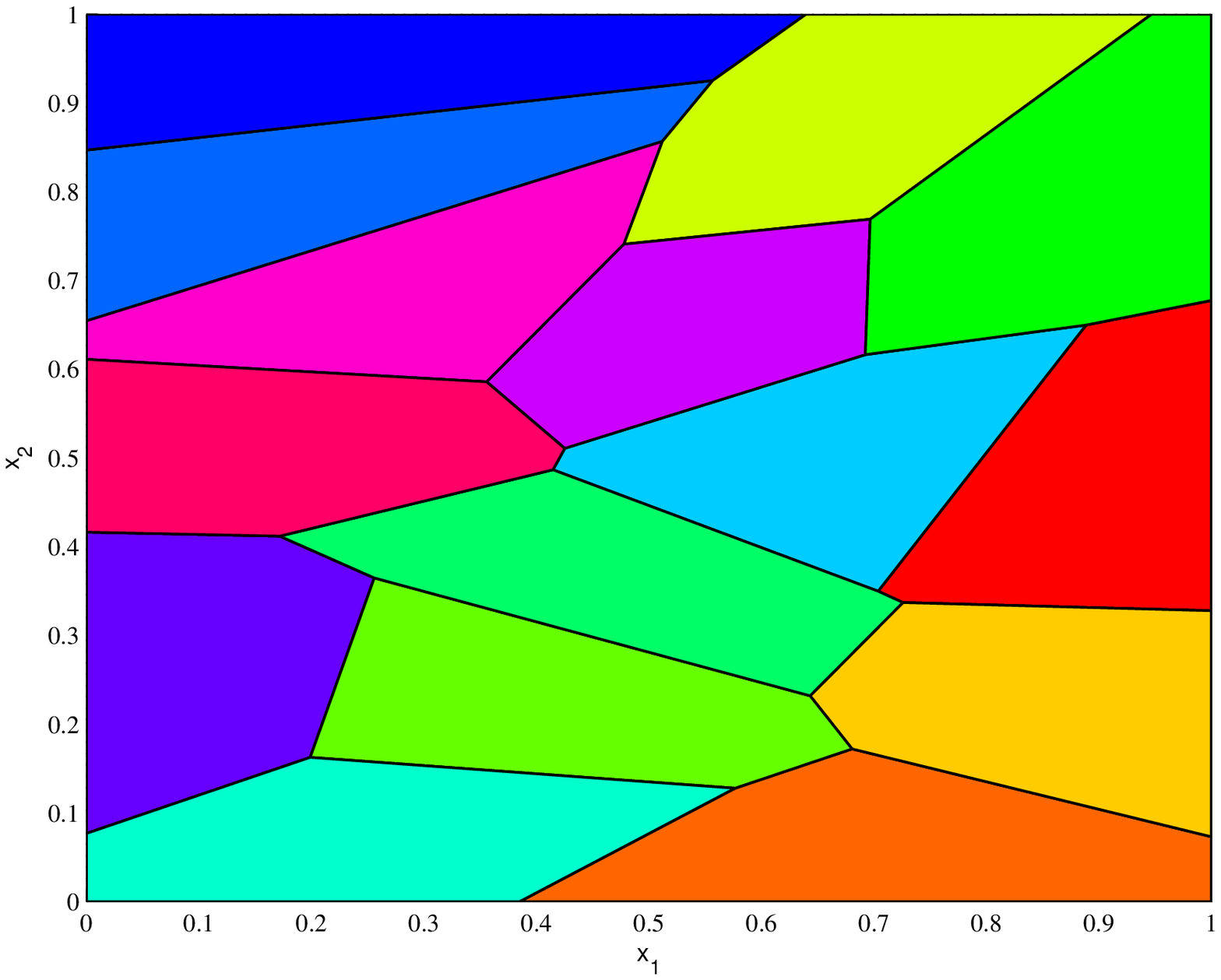,width=1.5in,height=1.5in} %
\epsfig{figure=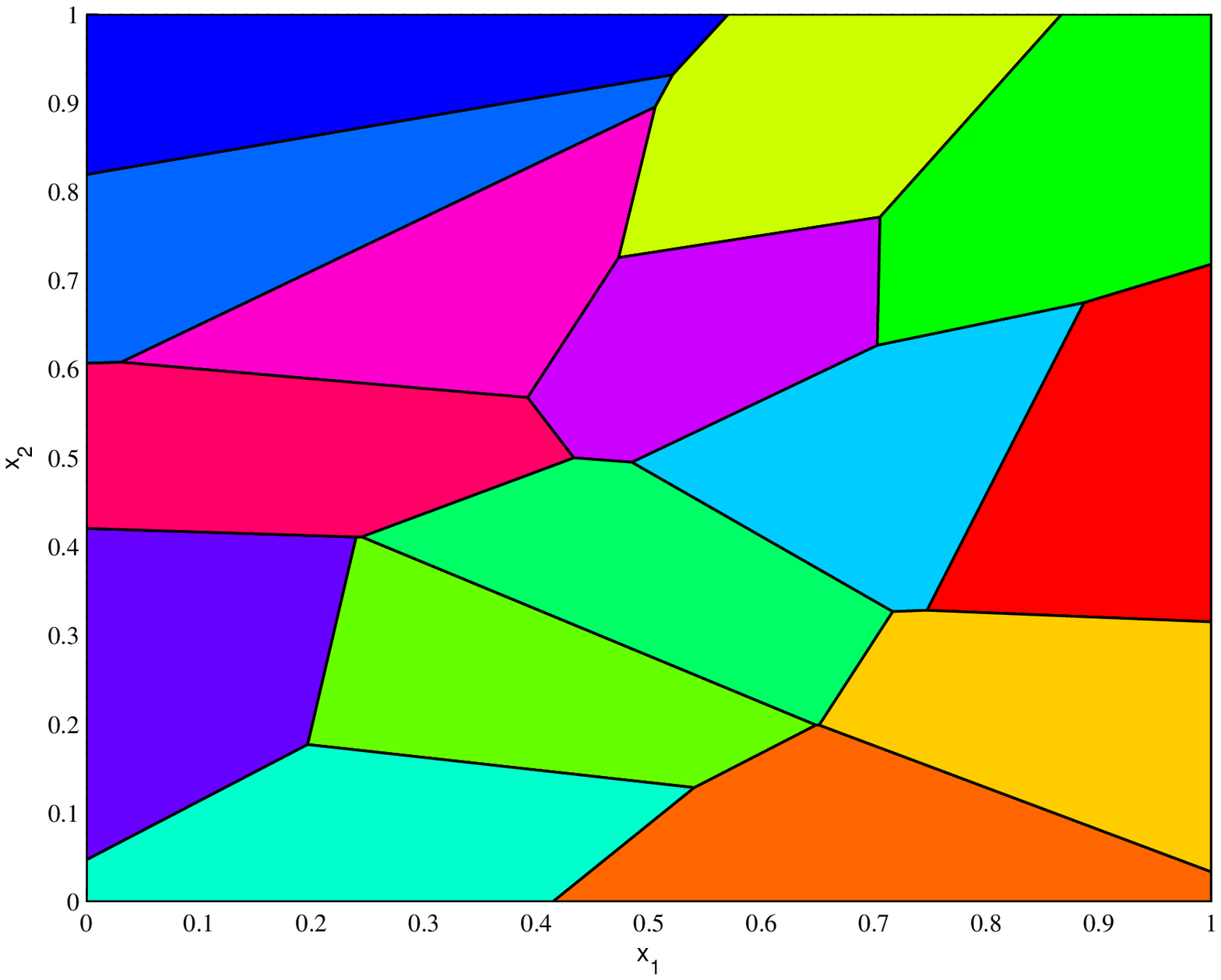,width=1.5in,height=1.5in} %
\epsfig{figure=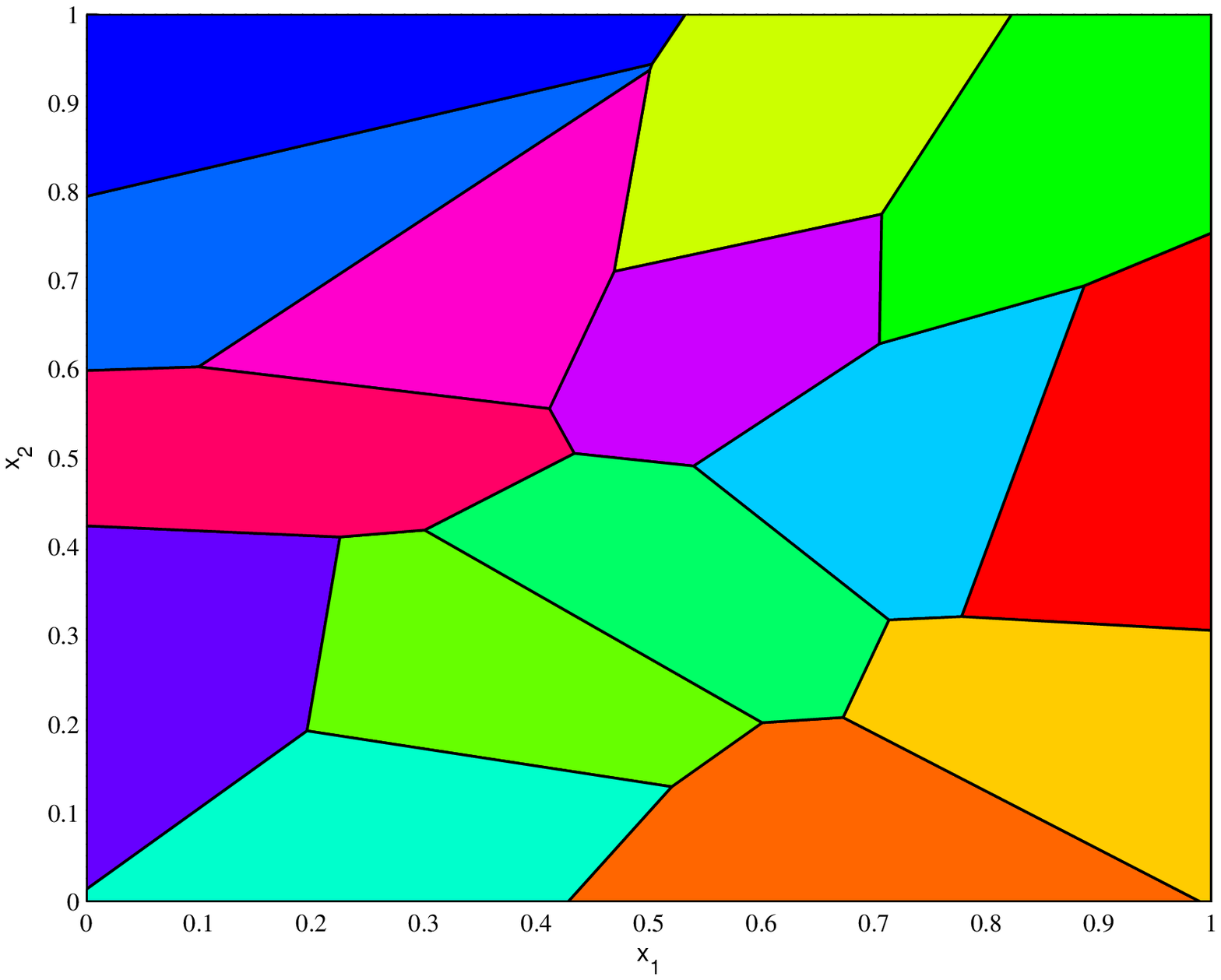,width=1.5in,height=1.5in} %
\epsfig{figure=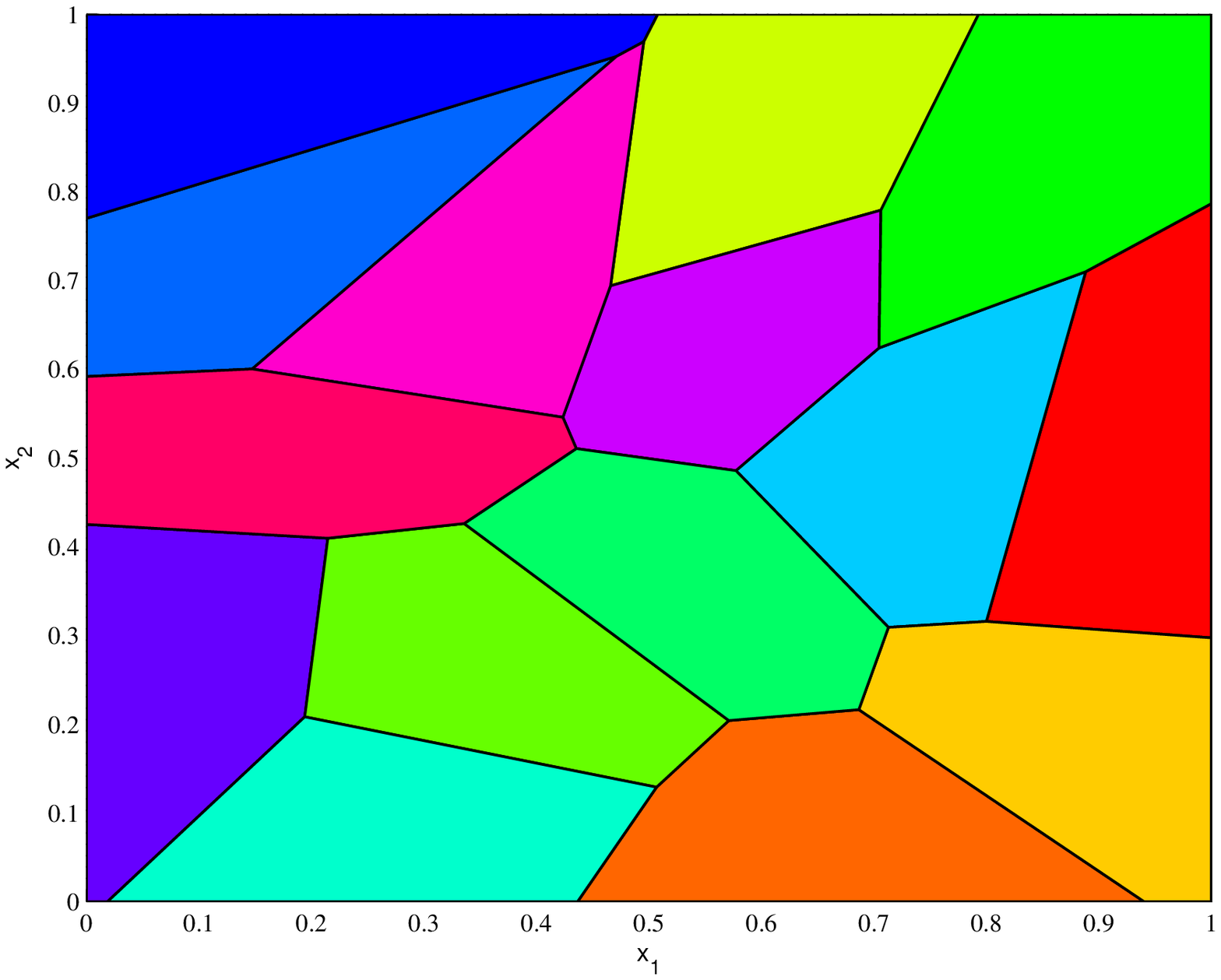,width=1.5in,height=1.5in} %
\epsfig{figure=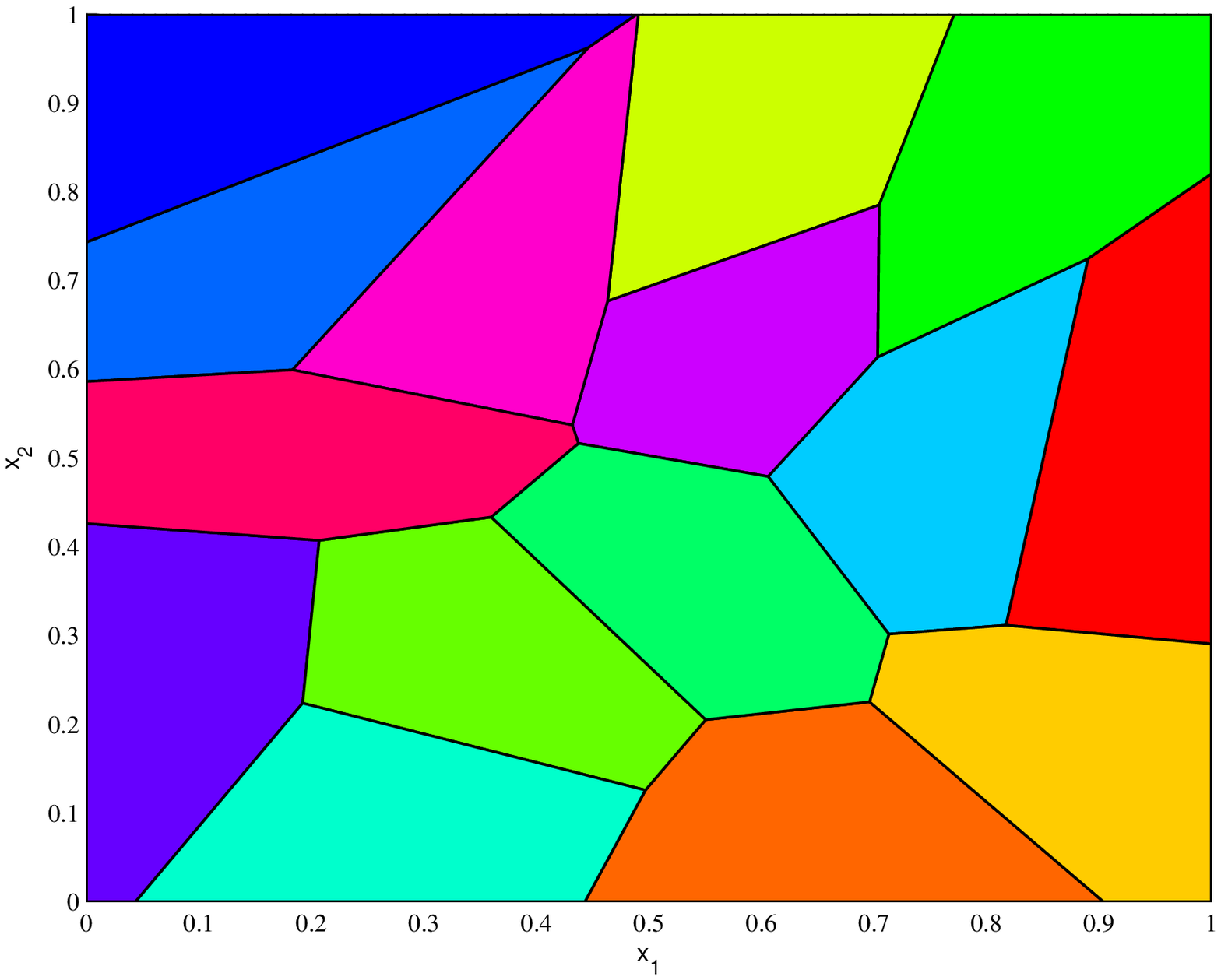,width=1.5in,height=1.5in} %
\epsfig{figure=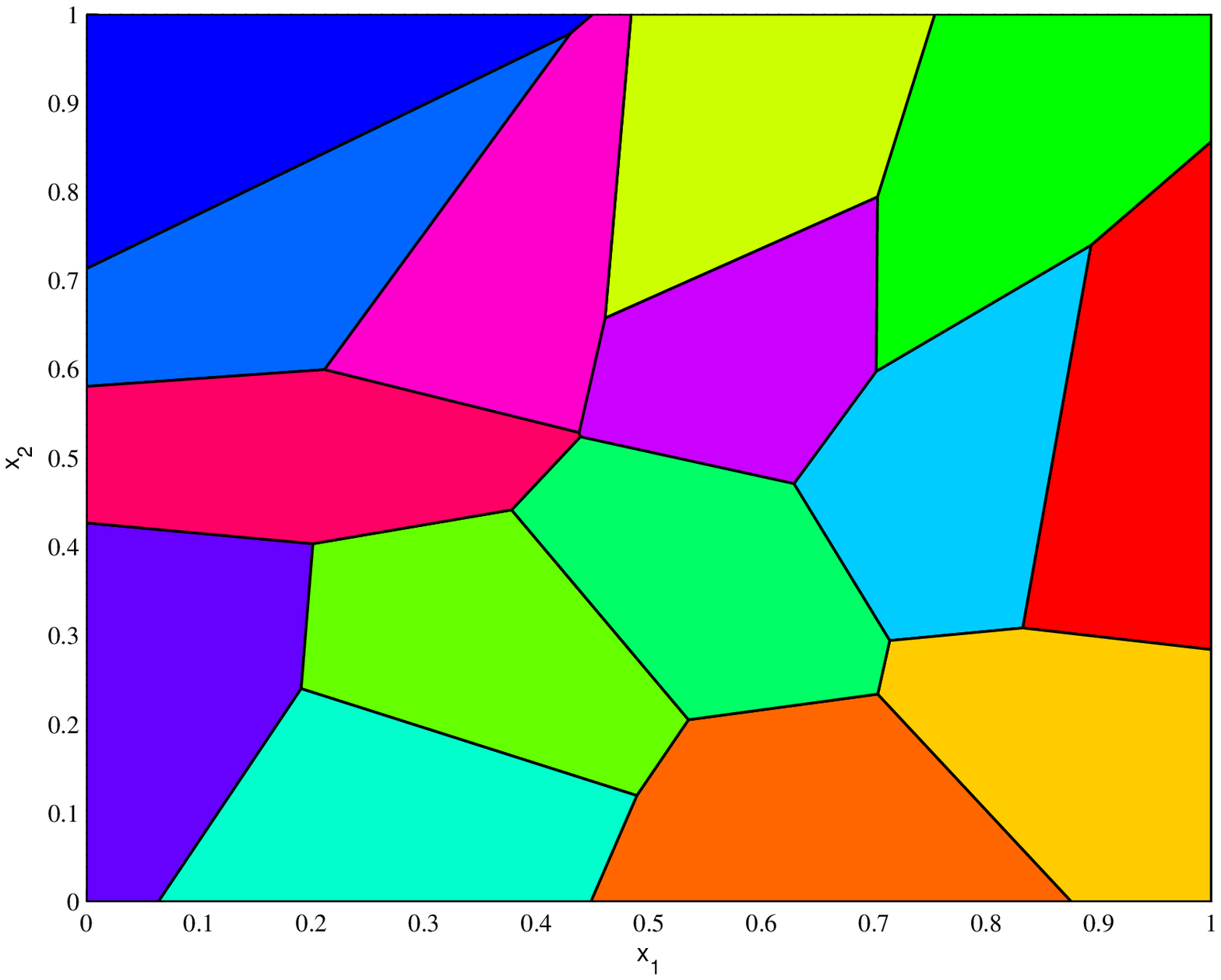,width=1.5in,height=1.5in} %
\epsfig{figure=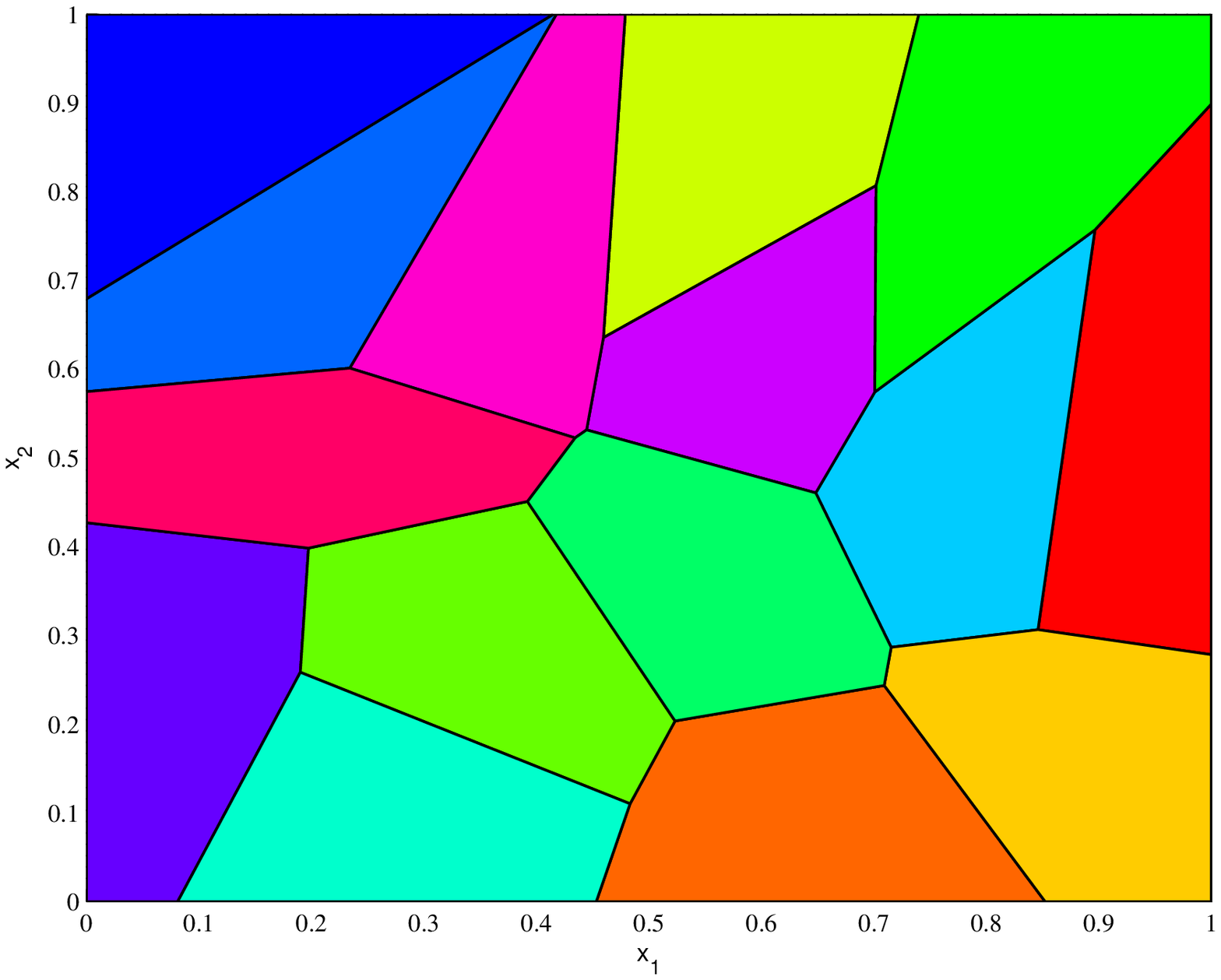,width=1.5in,height=1.5in} %
\epsfig{figure=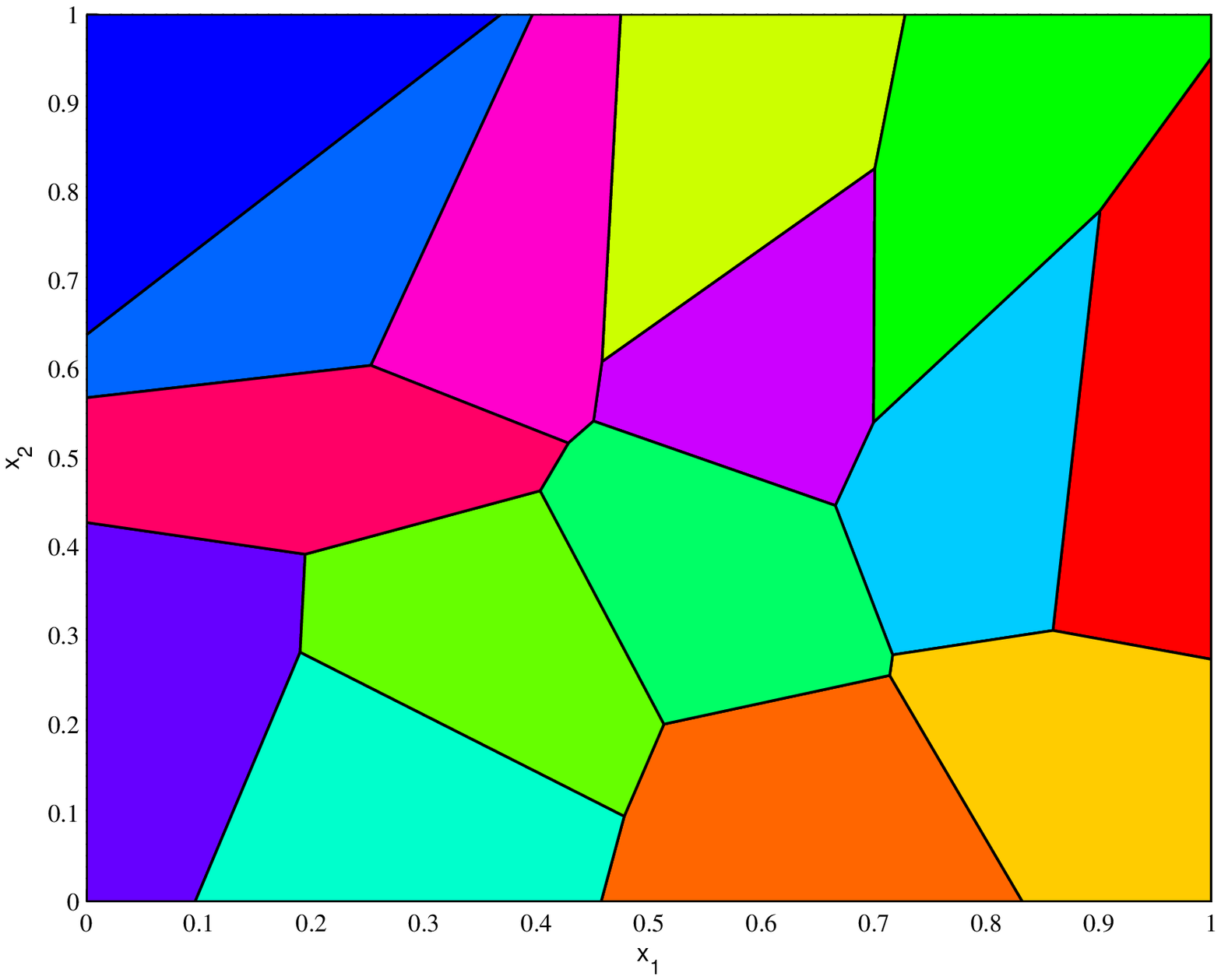,width=1.5in,height=1.5in} %
\epsfig{figure=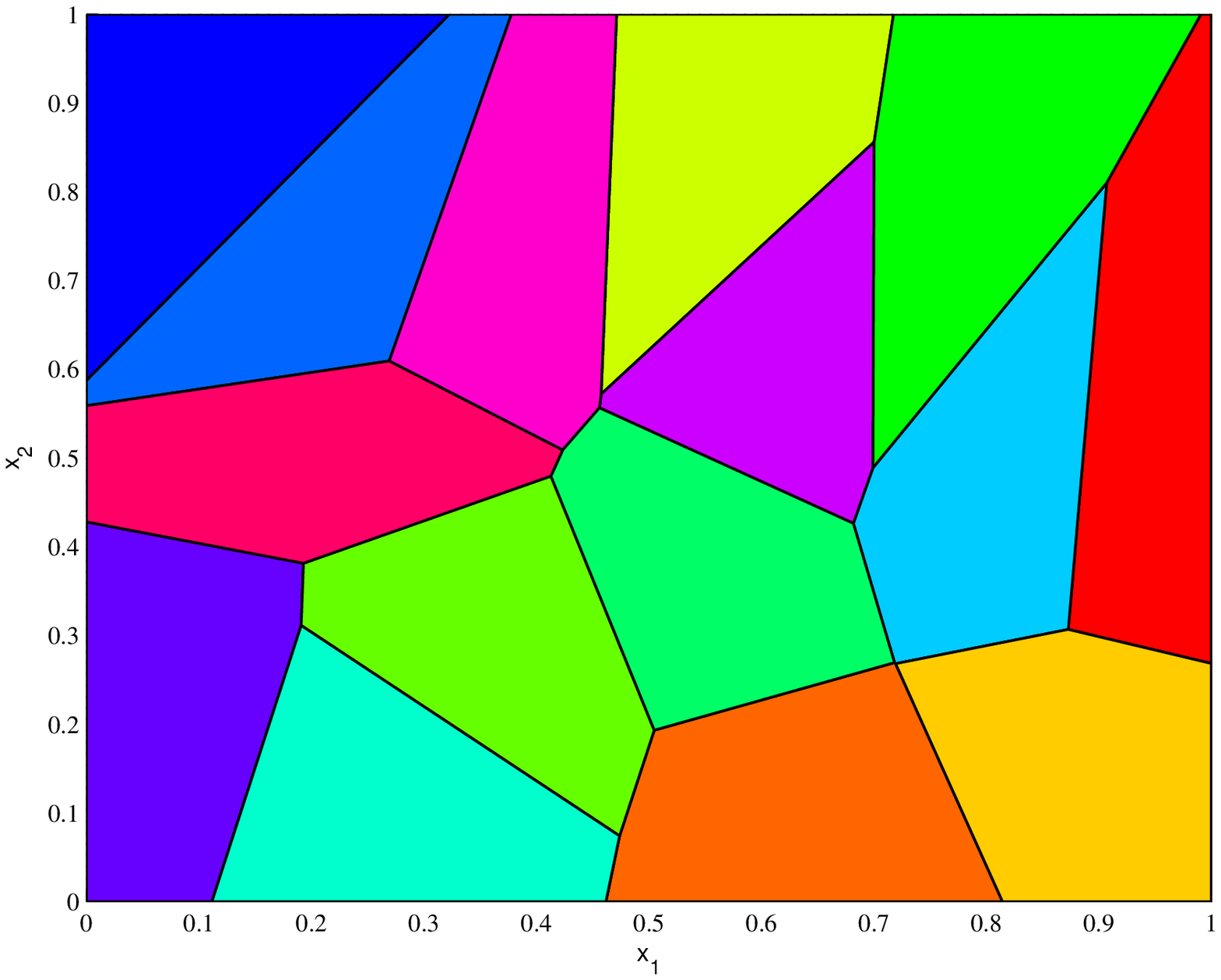,width=1.5in,height=1.5in} %
\epsfig{figure=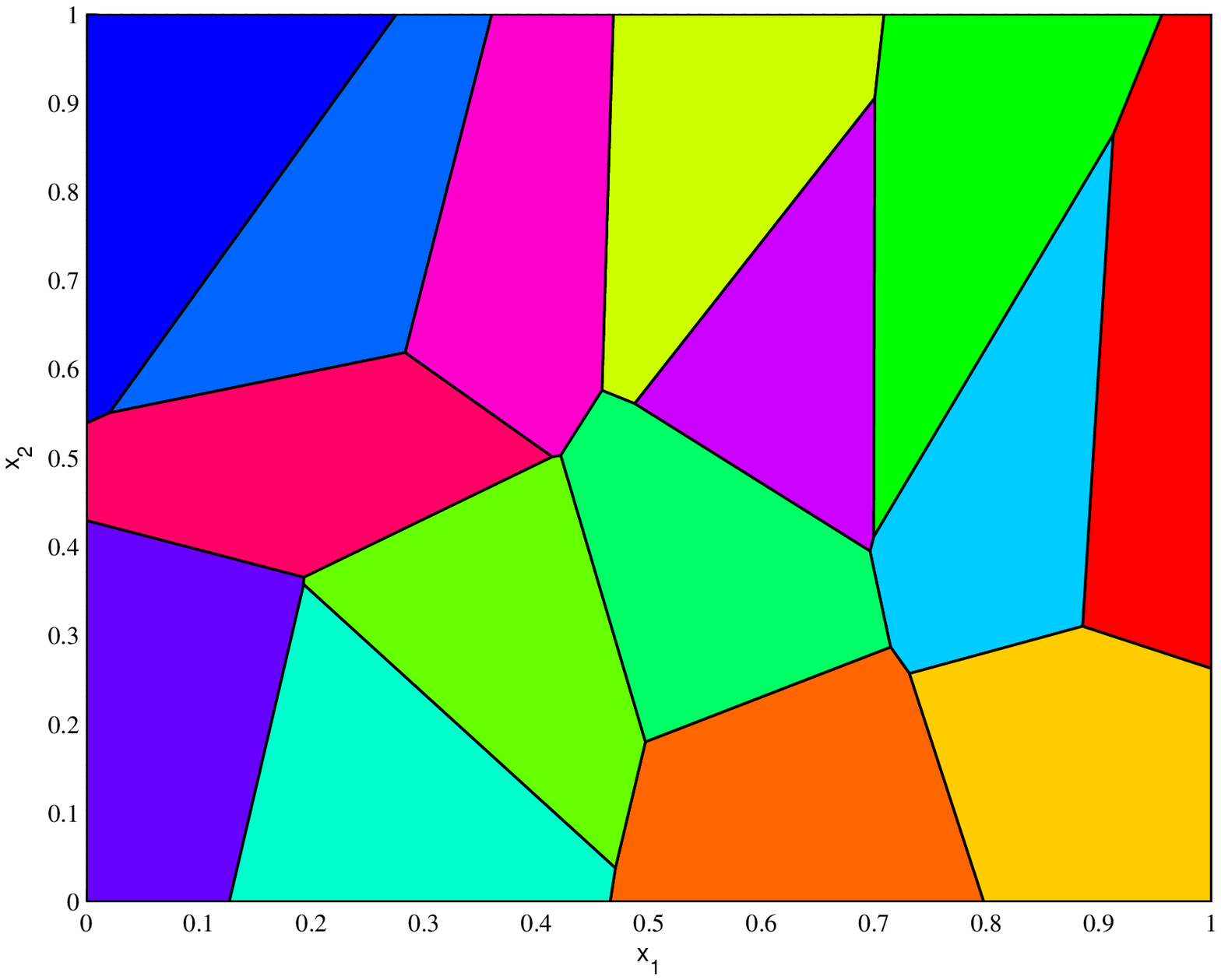,width=1.5in,height=1.5in} %
\epsfig{figure=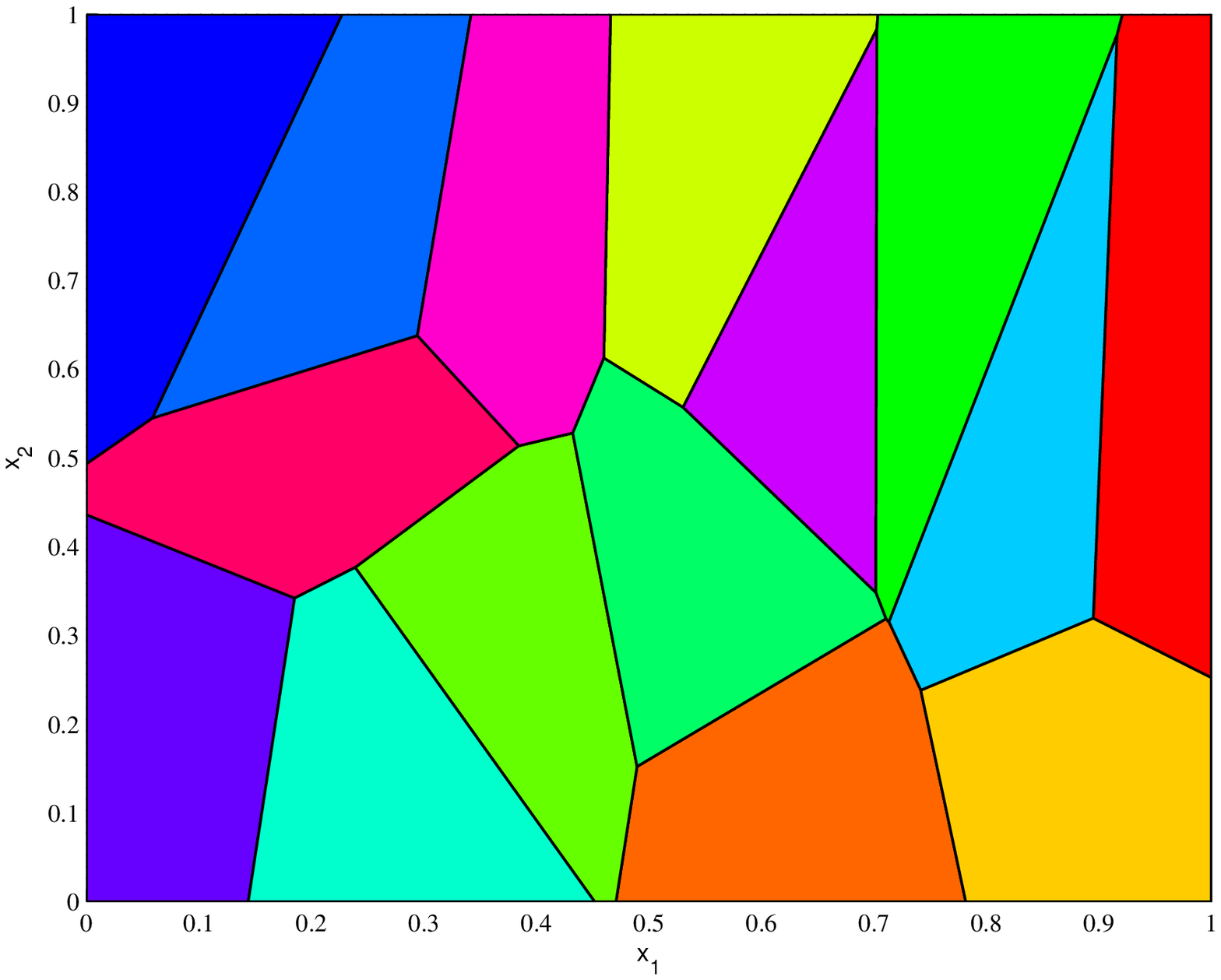,width=1.5in,height=1.5in} %
\epsfig{figure=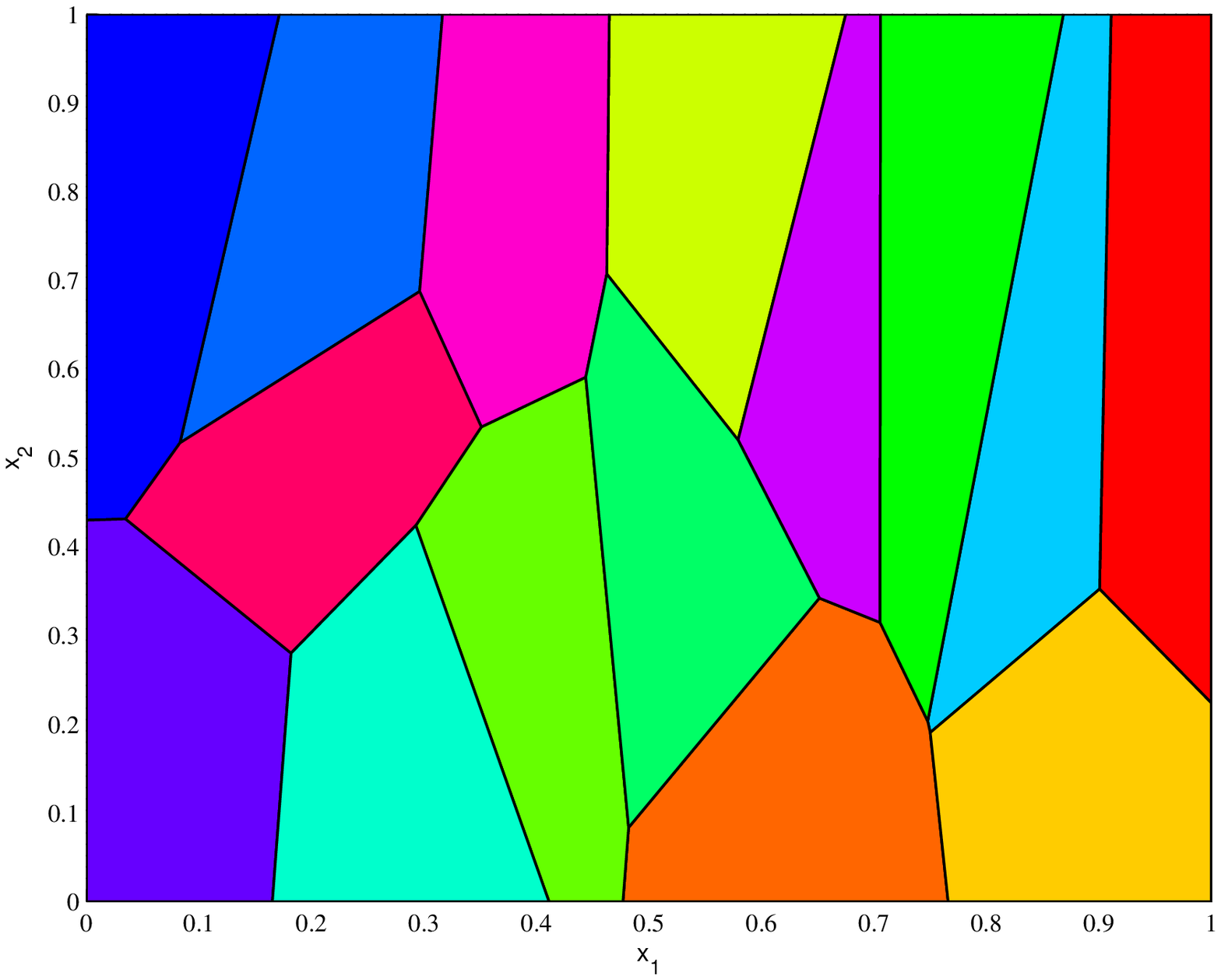,width=1.5in,height=1.5in} %
\epsfig{figure=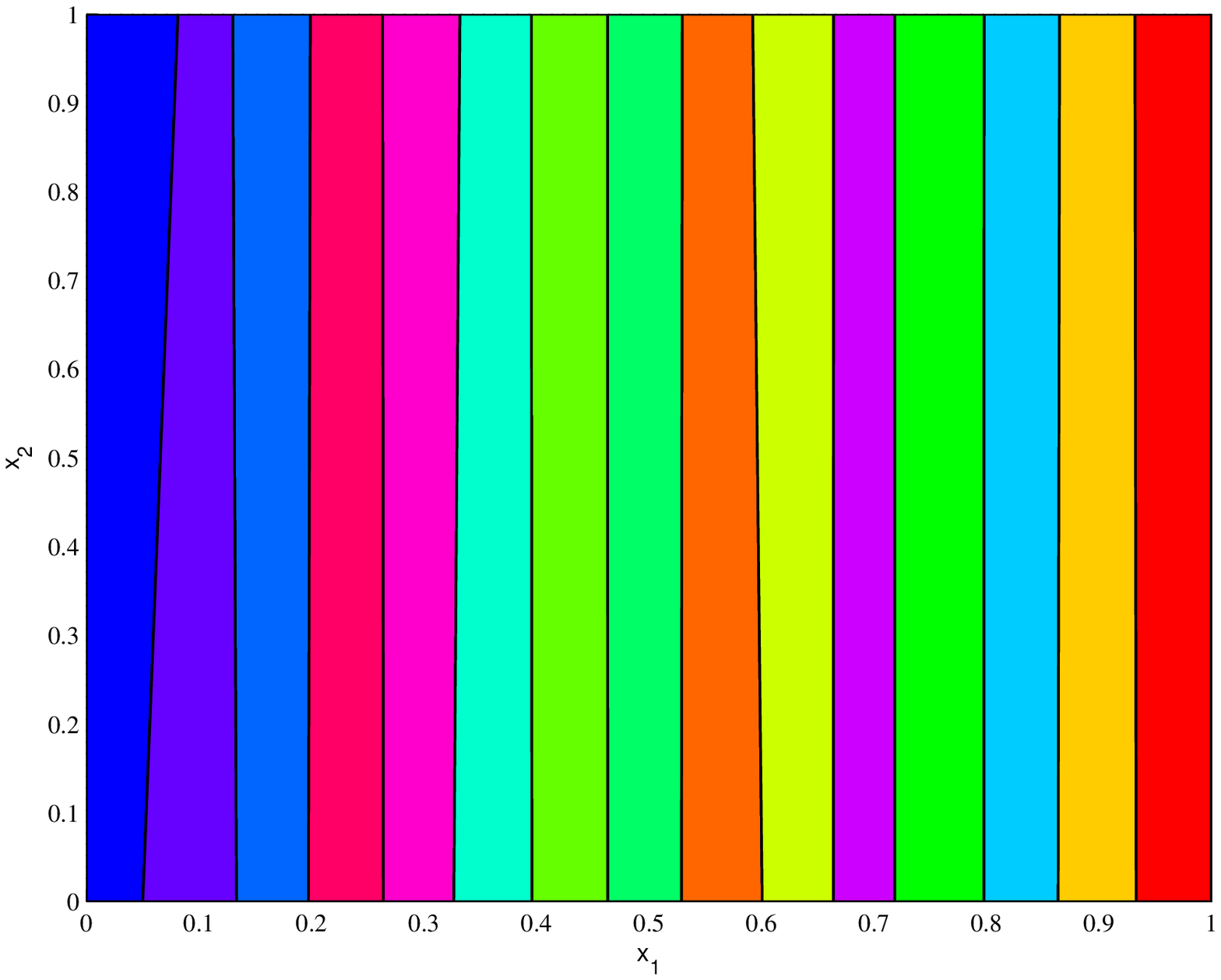,width=1.5in,height=1.5in}
\end{center}
\caption{Fifteen sample points: evolution of the tesselation for $\protect%
\varepsilon=0$ to $\protect\varepsilon=+\infty$ (from top left to bottom
right).}
\label{Figure:Sample15evol}
\end{figure}

\textbf{Acknowledgements} G.C. and F.S. gratefully acknowledge the support
of the Agence Nationale de la Recherche via the research project OTARIE. A.G. gratefully acknowledge the support of Chaire EDF-Calyon
``Finance et Développement Durable,'' of Chaire AXA ``Assurance et risques majeurs,'' and of Chaire Soci\'et\'e G\'en\'erale ``Risques Financiers''. The
authors wish to warmly thank Yann Brenier and Alessio Figalli for stimulating discussions.

\end{document}